\documentclass[reqno, 12pt]{amsart} 
\usepackage{fmtcount}
\usepackage[expansion=false]{microtype} 
\usepackage{amsfonts,amsthm,amsmath,amssymb,amscd,mathrsfs} 
\allowdisplaybreaks  
\usepackage{latexsym} 
\usepackage{xcolor, color} 
\usepackage[colorlinks=true,linkcolor=blue,citecolor=teal,urlcolor=black]{hyperref} 
\usepackage{graphicx}  
\usepackage{indentfirst} 

\usepackage{mathtools}

\renewcommand{\MR}[1]{}

\usepackage{lmodern} 
\usepackage{soul}     
\usepackage{wasysym}  

\theoremstyle{plain} 
\newtheorem{Thm}{Theorem}[section] 
\newtheorem{Lem}[Thm]{Lemma}     
\newtheorem{Prop}[Thm]{Proposition}
\newtheorem{Cor}[Thm]{Corollary}
\theoremstyle{definition}

\theoremstyle{remark}
\newtheorem{Rem}[Thm]{Remark}
\numberwithin{equation}{section} 

\newcommand{\beq}{\begin{equation}}            
	\newcommand{\eeq}{\end{equation}}
\newcommand{\ben}{\begin{eqnarray}}         
	\newcommand{\een}{\end{eqnarray}}
\newcommand{\beno}{\begin{eqnarray*}}
	\newcommand{\eeno}{\end{eqnarray*}}

\newcommand{\lt}{\left}
\newcommand{\rt}{\right}

\newcommand{\ltl}{\langle}
\newcommand{\rtr}{\rangle}

\newcommand{\px}{\partial_x}
\newcommand{\py}{\partial_y}

\newcommand{\p}{\partial_{i,ph}}
\newcommand{\q}{\partial_{i,fr}}

\newcommand{\qqquad}{\qquad\quad}
\newcommand{\alabel}{\stepcounter{equation}\tag{\theequation}\label}  

\usepackage[
letterpaper,                 
textheight=8.35in,           
headsep=20pt,                
footskip=36pt,               
marginparwidth=0pt,          
marginparsep=0pt,            
left=0.75in, 
right=0.75in
]{geometry}
\begin{document}
	\title[Blow-up suppression for the nematic liquid crystal flow]{Blow-up suppression for the nematic liquid crystal flow via Couette flow on $\mathbb{R}^2$}
	
	\author{Yubo~Chen}
	\address[Yubo~Chen]{School of Mathematical Sciences, Dalian University of Technology, Dalian, 116024,  China}
	\email{1220823215@mail.dlut.edu.cn}
	
	\author{Wendong~Wang}
	\address[Wendong~Wang]{School of Mathematical Sciences, Dalian University of Technology, Dalian, 116024,  China}
	\email{wendong@dlut.edu.cn}

    \author{Juncheng~Wei}
    \address[Juncheng~Wei]{Department of Mathematics, Chinese University of Hong Kong, Shatin, NT, Hong Kong}
    \email{wei@math.cuhk.edu.hk}
	
	\author{Guoxu~Yang}
	\address[Guoxu~Yang]{School of Mathematical Sciences, Dalian University of Technology, Dalian, 116024,  China}
	\email{guoxu\_dlut@outlook.com}

	\begin{abstract}
		As is well known, for the harmonic heat flow or liquid crystal flow in two-dimension, the solution may blow up when the initial energy is greater than $8\pi$. Motivated by Lai--Lin--Wang--Wei--Zhou (CPAM, 2022), where singular solutions were constructed  in the presence of small-scale velocity fields, it is natural to ask whether large-scale velocities may play a stabilizing role, preventing the concentration of blow-up. Here we show that the blow-up phenomenon can be suppressed by a Couette flow whose amplitude is large enough under a weak assumption on the anisotropic norm of the initial data. In particular, we construct examples with initial energy exceeding $8\pi$ that satisfy our assumptions.
	\end{abstract}
	

	\maketitle
	
	
	\tableofcontents

	\section{Introduction} 
	Liquid crystals represent a state of matter exhibiting properties intermediate between those of conventional liquids and solid crystals. Among the various mesophases, the nematic phase is the most prominent, characterized by rod-like molecules that possess long-range orientational order but lack positional order. The theoretical description of nematic liquid crystals is generally categorized into a hierarchy of three models: the microscopic Doi--Onsager kinetic theory \cite{O1949,DEE1988}, the macroscopic Landau--de Gennes Q-tensor theory \cite{d1974}, and the Ericksen--Leslie vector theory \cite{E1962,L1968}. The first provides a molecular kinetic description, while the latter two offer macroscopic continuum descriptions.
	
	In this paper, we 
    consider a simplified version of the Ericksen--Leslie system, which models the flow of nematic liquid crystals in the whole plane:
    \begin{align} \label{eq:main0} \left\{
\begin{aligned} 
			\partial_t v+v \cdot \nabla v-\nu \Delta v+\nabla P & =-\lambda \nabla \cdot(\nabla n \odot \nabla n) & & \text { in } \,\mathbb{R}^2 \times(0,+\infty), \\
			\nabla \cdot v & =0 & & \text { in } \, \mathbb{R}^2 \times(0,+\infty), \\
			\partial_t n + v \cdot \nabla n & =\gamma\left(\Delta n+|\nabla n|^2 n\right) & & \text { in } \,\mathbb{R}^2 \times(0,+\infty),
	\end{aligned}\right. \end{align}
	where $v(x, y, t): \mathbb{R}^2 \times(0,+\infty) \rightarrow \mathbb{R}^2$ represents the velocity field of the flow, $n(x, y, t): \mathbb{R}^2 \times(0,+\infty) \rightarrow S^2$, the unit sphere in $\mathbb{R}^3$, is a unit-vector field that represents the macroscopic molecular orientation of the liquid crystal material and $P(x, y, t): \mathbb{R}^2 \times(0,+\infty) \rightarrow \mathbb{R}$ represents the pressure function. The constants $\nu, \lambda$, and $\gamma$ are positive constants that represent viscosity, the competition between kinetic energy and potential energy, and microscopic elastic relaxation time for the molecular orientation field.  $\nabla n \odot \nabla n$ denotes the $2 \times 2$ matrix whose $(i, j)$-th entry is given by $\nabla_i n \cdot \nabla_j n$ for $1 \leq i, j \leq 2$.
	
	The system described above is a simplified version of the Ericksen--Leslie model, 
    originally developed between 1958 and 1968 \cite{E1962,L1968}. In the static case, this model reduces to the Oseen--Frank model. It provides a macroscopic continuum description of the material's evolution, governed by the interaction between the fluid velocity field $v(x, y, t)$ and the macroscopic director field $n(x, y, t)$, which represents the microscopic orientation of rod-like liquid crystals. Structurally, system \eqref{eq:main0} represents a coupling between the non-homogeneous Navier--Stokes equations and the transported flow of harmonic maps. The mathematical analysis of \eqref{eq:main0} was initiated by Lin--Liu \cite{LL1995,LL1996} in a series of papers during the 1990s.  In two dimensional space, global existence of weak solutions of \eqref{eq:main0} with finite singular time  was obtained by Lin--Lin--Wang in \cite{LLW2010} and Hong in \cite{H2011}, respectively. The uniqueness of weak solutions was proved by Lin--Wang in \cite{LW2010} and Xu--Zhang in \cite{XZ2012}. For the general form of \eqref{eq:main0} with Oseen--Frank energy,  Hong--Xin in \cite{HX2012} proved the global existence. For the general form of \eqref{eq:main0} with Oseen--Frank energy and Leslie stress, the global well-posed results of weak solutions were established by Huang-Lin-Wang \cite{HLW2014} and Wang--Wang \cite{WW2014} independently, and later the uniqueness result was obtained by Wang--Wang--Zhang in \cite{WWZ2016}. Also, Lei--Li--Zhang \cite{LLZ2014} gave a new proof of the global wellposedness of smooth solutions for a class of large initial data in energy space under a natural geometric angle condition. Recently, Lai--Lin--Wang--Wei--Zhou \cite{LLWWZ2022} developed a new inner-outer gluing method to construct solutions that blow up exactly at those $k$ points as $t$ goes to a finite time $T$. For the system \eqref{eq:main0} in three dimensions, which is supercritical, we refer to \cite{WZZ2013} for the local well-posedness result, \cite{HLX2014} for some blow-up criteria, \cite{HLLW2016} for finite time singularity, \cite{LW2016} and \cite{KW2025} for the existence of global weak solutions under some conditions and the references therein.

	Motivated by \cite{LLWWZ2022}, where singular solutions were constructed for small velocity, we investigate the global well-posed of the system \eqref{eq:main0} with a class of large data of the velocity.
	For simplicity, assume that $\nu = \lambda = \gamma =1$ in system \eqref{eq:main0}. It is easy to verify that $(\tilde{v}, \tilde{n})=(U,e_1)$ with the Couette flow $U=(A y, 0)$ is a stationary solution to the system \eqref{eq:main0}, where $A$ is the amplitude of the Couette flow. Denote the perturbation $u(t, x, y)=v(t, x, y)-U(y)$ and $d(t,x,y) = n(t,x,y) - e_1$, then $\left(u, d\right)$ satisfies
	\begin{equation}\begin{aligned} \label{eq:main1}
			\left\{\begin{array}{l}
				\partial_t u+A y \partial_x u+\binom{A u^2}{0}+u \cdot \nabla u-\Delta u+\nabla p=- \nabla\cdot\lt(\nabla d \odot \nabla d\rt), \\
				\nabla \cdot u=0, \\
				\partial_t d +A y \partial_x d +  u \cdot \nabla d - \Delta d =       |\nabla d|^2 (d + e_1),\\
				\left.u \right|_{t=0} = u_{\text{in}}, \quad\left.d \right|_{t=0} = d_{\text{in}}.
			\end{array}\right.
	\end{aligned}\end{equation}
	Direct computations imply
	\begin{align*}
		\nabla \cdot (\nabla d \odot \nabla d) = \nabla \frac{|\nabla d|^2}{2} +   \binom{\px d \cdot\Delta d}{\py d \cdot \Delta d}. 
	\end{align*}
	Define
	$$
	\omega = \partial_y u^1-\partial_x u^2, \quad u=\nabla^{\perp} \Phi=\left(\partial_y \Phi,-\partial_x \Phi\right)^T,
	$$
	then
	$
	\Delta \Phi=\omega
	$
	and $\omega$ satisfies
	$$
	\partial_t \omega+A y \partial_x \omega-\Delta \omega+u \cdot \nabla \omega= -\py\left(  \px d\cdot\Delta d\right)  + \px \left(  \py d\cdot\Delta d\right)  .
	$$
	After the time rescaling $t \mapsto A^{-1} t$, one rewrites the system \eqref{eq:main1} by
	\begin{equation}\begin{aligned} \label{eq:main}
			\left\{\begin{array}{l}
				\partial_t \omega+y \partial_x \omega-\frac{1}{A} \Delta \omega=-\frac{1}{A}\left[ u \cdot \nabla \omega   -\px \left(  \py d\cdot\Delta d\right) + \py\left(  \px d\cdot\Delta d\right) \right], \\
				u=\nabla^{\perp} \Delta^{-1} \omega,\\
				\partial_t d + y \partial_x d   - \frac1A \Delta d =   -\frac1A \left[ u\cdot\nabla d   -|\nabla d|^2 (d + e_1)\right] .
			\end{array}\right.
	\end{aligned}\end{equation}

	\subsection{Main result}
	To state the main theorem, we first introduce some notations. Denote $D_x := -i \partial_x$, understood as a Fourier multiplier operator. All norms involving $D_x$ are defined on the Fourier side and $\langle \cdot \rangle := \sqrt{1 + (\cdot)^2}$. For $m,\epsilon \geq 0$ and certain smooth enough $f$, define $\|(\cdot,*)\|^2:= \|\cdot\|^2 + \|*\|^2$, the anisotropic norm
	\begin{align*}
		\|f\|_{Y_{m, \epsilon}} := \lt\|  \ltl D_x \rtr^m \ltl \frac1{D_x}\rtr^\epsilon  f\rt\|_{ L^2},\alabel{eq:def of Y}
	\end{align*}
	and the space-time norm
	\begin{align*}
		\|f\|_{X_{a,m,\epsilon}}^2&:= \lt\|  e^{a A^{-\frac{1}{3}}\left|D_x\right|^{\frac{2}{3} }t}     \ltl D_x \rtr^m \ltl \frac1{D_x}\rtr^\epsilon f \rt\|_{L^\infty L^2}^2      + \frac{1}{A} \lt\|e^{a A^{-\frac{1}{3}}\left|D_x\right|^{\frac{2}{3} }t}  \ltl D_x \rtr^m \ltl \frac1{D_x}\rtr^\epsilon  \nabla f\rt\|_{L^2 L^2}^2   \\
		& \quad + \frac{1}{ A^{\frac{1}{3}}}\lt\|e^{a A^{-\frac{1}{3}}\left|D_x\right|^{\frac{2}{3} } t }   \ltl D_x \rtr^m \ltl \frac1{D_x}\rtr^\epsilon \left|D_x\right|^{\frac{1}{3}} f\rt\|_{L^2 L^2}^2  +   \lt\|e^{a A^{-\frac{1}{3}}\left|D_x\right|^{\frac{2}{3} }t } \ltl D_x \rtr^m \ltl \frac1{D_x}\rtr^\epsilon   \partial_x \nabla  \Delta^{-1} f \rt\|_{L^2 L^2}^2.
		\alabel{eq:X norm}
	\end{align*}

	Our main result is stated as follows.
	\begin{Thm} \label{thm:Ericksen--Leslie system}
		Let $\frac13<\epsilon<\frac12<m$, $0<a<\frac{1}{16(1+2 \pi)}$ and $\delta>1$. Assume that the initial data $u_{\mathrm{in}}\in H^1(\mathbb{R}^2)$ and $d_{\mathrm{in}}$ satisfy
		\begin{align*}
			\lt \| \nabla^{\perp} \cdot   u_{\mathrm{in}} \rt \|_{Y_{m, \epsilon}} +\lt \|    |D_x|^\frac13  d_{\mathrm{in}} \rt \|_{Y_{m, \epsilon}}   + \lt \|   \lt(\px^2, \py^2 \rt)   |D_x|^\frac13   d_{\mathrm{in}} \rt \|_{Y_{m, \epsilon}} <\infty.
		\end{align*}
		There exists a positive constant $\bar{A}_1$ satisfying
        $$\bar{A}_1= C(m,\epsilon,\delta)\left( \lt \|    \lt(\px^2,\py^2\rt) |D_x|^\frac13  d_{\mathrm{in}} \rt \|_{Y_{m,\epsilon}} +\lt \| \nabla^{\perp} \cdot   u_{\mathrm{in}} \rt \|_{Y_{m, \epsilon}} 	+1\right) ^{\max\{\frac{2}{\delta-1},\, 24 \}}$$
		such that if $A>\bar{A}_1$ and
		\begin{align*}
			A^{\delta} \lt \|    |D_x|^\frac13  d_{\mathrm{in}} \rt \|_{Y_{m, \epsilon}}   \leq 1, \alabel{eq:smallness2} 
		\end{align*}
		the solutions to \eqref{eq:main} are global in time satisfying
		\begin{align*}
			\lt\|e^{a A^{-\frac{1}{3}}\left|D_x\right|^{\frac{2}{3} }t} u\rt\|_{L^\infty L^\infty} +\lt \|e^{a A^{-\frac{1}{3}}\left|D_x\right|^{\frac{2}{3} }t} \nabla d\rt\|_{L^\infty L^\infty} \leq C,
		\end{align*}
		for all $t\geq0$.
	\end{Thm}

    Note that when the velocity equation (Navier--Stokes equations) is ignored, the above perturbation system reduces to a simpler equation describing the harmonic map heat flow around the Couette flow., i.e., 
	\begin{equation}\begin{aligned} \label{eq:main2}
			\partial_t d + y \partial_x d   - \frac1A \Delta d =    \frac1A    |\nabla d|^2 (d + e_1)  .
	\end{aligned}\end{equation}

	\begin{Cor} \label{thm:harmonic map heat flow}
		Let $\frac16 <\epsilon< \frac12 <m$, $0<a<\frac{1}{16(1+2 \pi)}$ and $\delta>1$. Assume that the initial data $d_{\mathrm{in}}$ satisfies
		\begin{align*}
			\lt\|    |D_x|^\frac13  d_{\mathrm{in}}\rt\|_{Y_{m, \epsilon}}+ \lt\|  |D_x|^\frac13 \nabla^2  d_{\mathrm{in}} \rt\|_{Y_{m, \epsilon}}    <\infty.
		\end{align*}
		There exists a positive constant $\bar{A}_2$ depending on $\epsilon, m$, $\delta$ and $\lt\|       |D_x|^\frac13 \nabla^2  d_{\mathrm{in}} \rt\|_{Y_{m, \epsilon}}$
		such that if $A>\bar{A}_2$ and
		\begin{align*}
			A^{  \delta} \lt\|    |D_x|^\frac13  d_{\mathrm{in}} \rt\|_{Y_{m, \epsilon}}   \leq 1, 
		\end{align*}
		the solutions to \eqref{eq:main2} are global in time satisfying
		\begin{align*}
			\lt \|e^{a A^{-\frac{1}{3}}\left|D_x\right|^{\frac{2}{3} }t} \nabla d\rt\|_{L^\infty L^\infty} \leq C,
		\end{align*}
		for all $t\geq0$.
	\end{Cor}
	
	
	\begin{Rem} When the system \eqref{eq:main1} is not coupled with Navier--Stokes equations, the solutions with the initial energy $\| d_{\text{in}}\|_{L^2(\mathbb{R}^2)}^2>8\pi$ may blow up (see Chang--Ding--Ye's counterexample in the 1-corotational class in a disc \cite{CDY1992}). Davila--del Pino--Wei \cite{DDW2020} also constructed finite time blow-up solutions of the two dimensional harmonic map flow precisely at those given
points. 
For the 2D nematic liquid crystal flow, Lai--Lin--Wang--Wei--Zhou \cite{LLWWZ2022} constructed solutions that blow up exactly at those given $k$ points by
the inner-outer gluing method. Theorem \ref{thm:Ericksen--Leslie system} and Corollary \ref{thm:harmonic map heat flow} show that the solutions to the two dimensional nematic liquid crystal or the harmonic map flow will not blow up around Couette flow for a large class initial data. To the best of our knowledge, this seems to be the first result concerning the suppression of blow-up behavior in high-dimensional liquid crystals. It is worth noting that  Chen--Huang--Liu \cite{CHL2020} considered the Cauchy problem of the Poiseuille flow of the full Ericksen–Leslie model for nematic liquid crystals recently, where the liquid crystal equations reduced to a one dimensional system
\begin{align*}
    \partial_t u=\partial_x(\partial_x u+\partial_t\theta),\quad \partial_{tt}\theta+2\partial_t\theta=c(\theta)\partial_x \left(c(\theta)\partial_x\theta \right)-\partial_x u.
\end{align*}
(See also the recent similar result in \cite{CHXZ2026}).


\end{Rem}

\begin{Rem}[Physical mechanism of blow-up suppression]
Intuitively, the mechanism by which the Couette flow suppresses the potential blow-up of the nematic liquid crystal director field can be understood as a competition between the \textit{nonlinear focusing effect} and the \textit{shear-induced mixing}.
In the absence of the background shear flow, the harmonic map heat flow  may develop singularities in finite time due to energy concentration.
However, the large shear flow $U=(Ay, 0)$ acts as a transport mechanism that rapidly mixes the fluid and the director field. This shearing effect stretches the spatial scales of the director field in the streamwise direction ($x$-direction), effectively transferring energy from low frequencies to high frequencies (large $k$).
Since the dissipation term $-\Delta d$ is significantly more efficient at high frequencies, this cascade of energy to smaller scales results in a much faster decay rate---known as \textit{enhanced dissipation}---than mere diffusion would provide.
If the shear amplitude $A$ is sufficiently large, this enhanced dissipation happens fast enough to disperse the energy concentration before a singularity can form.
\end{Rem}

\begin{Rem}[Motivation for the $|D_x|^{1/3}$ index]
The choice of the fractional derivative index $1/3$ in the norms (e.g., $\| |D_x|^{1/3} d \|$) is intrinsic to the enhanced dissipation properties of the Couette flow.
The multiplier method (see \eqref{eq:m} below) allows us to recover a fractional dissipation term with a coefficient scaling as $A^{-1/3}$:
\begin{equation*}
    \int_{\mathbb{R}^2} k \partial_\xi \mathcal{M}(k, \xi) |\hat{f}|^2 d\xi dk \gtrsim \frac{1}{A^{1/3}} \| |D_x|^{1/3} f \|_{L^2}^2.
\end{equation*}
Thus, the $|D_x|^{1/3}$ norm represents the maximal gain of regularity in the $x$-direction provided by the shear flow, which is crucial for closing the nonlinear energy estimates against the potential blow-up.
\end{Rem}

\begin{Rem}[The anisotropic norm]
		Theorem \ref{thm:Ericksen--Leslie system} and Corollary \ref{thm:harmonic map heat flow} contain an anisotropic norm
		\[
		\|f\|_{Y_{m,\epsilon}}
		=  \lt\|\langle D_x\rangle^{m}\langle \frac{1}{D_x}\rangle^{\epsilon} f \rt\|_{L^2(\mathbb{R}^2)},
		\quad  \frac13  <\epsilon<\frac{1}{2}<m,
		\]
		whose weight function in Fourier variables is defined as
		\[
	\Lambda_{m,\epsilon} (k) = (1+|k|^2)^{m}(1+|k|^{-2})^{\epsilon},
		\]
which plays an important role in dealing with the velocity by Biot-Savart's law. For example, see the term of 
\begin{align*}
    \lt\|e^{a A^{-\frac{1}{3}}|l|^{\frac{2}{3} } t } |l|  \lt \|  \Delta_l^{-1} \omega_l \rt\|_{L_y^{\infty}}\rt\|_{L_l^1}
\end{align*}
       in Lemma \ref{lem:est of py^2 u}, \eqref{41}.

        At high frequencies ($|k|\gg1$), $\Lambda_{m,\epsilon}(k)\sim |k|^{2m}$, enforcing $m$-th order regularity in the $x$-direction, while at low frequencies ($|k|\ll1$), $\Lambda_{m,\epsilon}(k)\sim |k|^{-2\epsilon}$, which penalizes concentration near $k=0$. Since $2\epsilon<1$, this singularity is integrable and does not exclude typical $L^2$ functions. In contrast to the isotropic Sobolev norm with weight $(1+|k|^2+|\xi|^2)^s$, the $Y_{m,\epsilon}$-norm is anisotropic and acts only in $x$. In particular,
\[
\|f\|_{H_x^s L_y^2} \lesssim \|f\|_{Y_{m,\epsilon}}, \qquad -\epsilon \le s \le m,
\]
so $Y_{m,\epsilon}$ controls a whole interval of $x$-regularities. Typical examples include $\mathcal{S}(\mathbb{R}^2)$, $C_c^\infty(\mathbb{R}^2)$, Gaussian profiles, functions $f=\partial_x g$ with $g\in L^2$, and $L^2$ functions with zero $x$-average, while nonzero constants are excluded.
	\end{Rem}

\begin{Rem}[Examples of initial data with arbitrarily large energy]
    	Let $\phi, \varphi \in \mathcal{S}(\mathbb{R})$ be Schwartz functions satisfying:
		\begin{itemize}
			\item $\operatorname{supp} \widehat{\phi} \subseteq [1, 2]$;
			\item $\|\phi\|_{L^2(\mathbb{R})} = \|\varphi\|_{L^2(\mathbb{R})} = 1$.
		\end{itemize}
		Define the perturbation $d_{\rm in}$ ($u_{\rm in} = 0$) depending on a scaling parameter $\lambda \in (0,1)$, a frequency parameter $N \gg 1$, and an exponent $\theta$ (to be determined) as:
		\begin{align*}
			d^{\lambda, N}_{\rm in}(x,y) :=  \lambda^\theta \phi(\lambda x) \varphi(y) \cos(Ny)e_1.
		\end{align*}
		
		\underline{\bf Norm Estimates in $Y_{m, \epsilon}$.}
		Note that the Fourier transform of $\phi(\lambda x)$ is supported in $k \in [\lambda, 2\lambda]$. For $\lambda \ll 1$, within this support we have:
		\begin{itemize}
			\item $\langle k \rangle \approx 1$.
			\item $\langle k^{-1} \rangle = (1 + |k|^{-2})^{1/2} \approx \lambda^{-1}$.
		\end{itemize}
		Also, the $L^2$ scaling of the $x$-component is $\|\phi(\lambda x)\|_{L^2_x} = \lambda^{-1/2}$.
		
		\noindent \textbf{(i) Estimate of the Low-Order Term ($L$):}
		Consider the term $|D_x|^{1/3} d_{\text{in}}$. The operator multiplier is $|k|^{1/3} \approx \lambda^{1/3}$.
		\begin{align}
			L(\lambda) &:= \lt\| |D_x|^{1/3} d_{\text{in}} \rt\|_{Y_{m, \epsilon}} \notag \\
			&\approx \lambda^\theta \cdot \underbrace{\lambda^{-\epsilon}}_{\text{weight}} \cdot \underbrace{\lambda^{1/3}}_{\text{derivative}} \cdot \underbrace{\lambda^{-1/2}}_{L^2 \text{ mass}} \cdot \|\varphi \cos(Ny)\|_{L^2_y} \notag \\
			&\leq C \lambda^{\theta - \epsilon - 1/6}. \label{eq:L_est}
		\end{align}
		(Note: $\|\varphi \cos(Ny)\|_{L^2_y} \approx 1/\sqrt{2}$ for large $N$).
		To ensure $L(\lambda)$ remains small (or bounded) as $\lambda \to 0$, we require $\theta > \epsilon + 1/6$.
		
		\noindent \textbf{(ii) Estimate of the High-Order Term ($H$):}
		Consider $\lt(\partial_x^2, \partial_y^2\rt) |D_x|^{1/3} d_{\text{in}}$. For large $N$, the $\partial_y^2$ term dominates ($\partial_y^2 \sim N^2$).
		\begin{align}
			H(\lambda, N) &:= \| \lt(\partial_x^2, \partial_y^2\rt) |D_x|^{1/3} d_{\text{in}} \|_{Y_{m, \epsilon}} \notag \\
			&\approx N^2 \lt\| |D_x|^{1/3} d_{\text{in}} \rt\|_{Y_{m, \epsilon}} \leq C N^2 \lambda^{\theta - \epsilon - 1/6}. \label{eq:H_est}
		\end{align}
		
		\underline{\bf Energy Estimate.}
		The Dirichlet energy is dominated by the $y$-derivative:
		\begin{align}
			E(\lambda,N)&:=\|\nabla d_{\text{in}}\|_{L^2}^2 \ge \|\partial_y d_{\mathrm{in}}\|_{L^2}^2 \notag \\
			&= \lambda^{2\theta} \|\phi(\lambda x)\|_{L^2_x}^2 \|\partial_y (\varphi(y) \cos(Ny))\|_{L^2_y}^2 \notag \\
			&\approx \lambda^{2\theta} \cdot \lambda^{-1} \cdot \frac{N^2}{2} \geq C \lambda^{2\theta - 1} N^2. \label{eq:Energy}
		\end{align}
		To achieve $\|\nabla d_{\text{in}}\|_{L^2}^2 > 8\pi$, we need $\lambda^{2\theta - 1} N^2 > 16\pi$.
		
		\underline{\bf Parameter Selection and Verification.}
		According to Theorem \ref{thm:Ericksen--Leslie system}, a valid parameter $A$ exists if the lower bound $\bar{A}_1$ is strictly smaller than the upper bound imposed by the smallness condition $A^\delta L \le 1$.
		\begin{align*}
			\underline{\text{Condition: } \quad C(m,\epsilon,\delta) \lt(H(\lambda, N) + 1\rt)^\kappa < (L(\lambda))^{-1/\delta}},
		\end{align*}
		where $\kappa = \max\{\frac{2}{\delta-1}, 24\}$. It follows from  \eqref{eq:L_est} and \eqref{eq:H_est} that
		\begin{align}
			C(m,\epsilon,\delta) (N^2 \lambda^{\theta - \epsilon - 1/6})^\kappa < (\lambda^{\theta - \epsilon - 1/6})^{-1/\delta}. \label{eq:Gap}
		\end{align}
		Let $\mu = \theta - \epsilon - 1/6$. We choose $\theta$ so that $\mu > 0$ is a small constant. The condition becomes:
		\begin{align*}
			N^{2\kappa} \lambda^{\kappa \mu} \ll \lambda^{-\mu/\delta} \implies N \ll \lambda^{-\frac{\mu}{2} (\frac{1}{\delta} + \kappa)}.
		\end{align*}
		Since the exponent on the right hand side is negative (power of $1/\lambda$), as $\lambda \to 0$, the allowed upper bound for $N$ grows to infinity. Thus, for any sufficiently small fixed $\lambda$, there exists a large range $N$ that satisfies the gap condition.
Similarly
, the energy scales as $E \approx \lambda^{2\theta-1} N^2$. Since $N$ can be chosen arbitrarily large (within the gap limit defined by $\lambda$), we can select $N$ large enough such that:
		\begin{align*}
			\lambda^{2\theta - 1} N^2 > 16\pi, 
		\end{align*}
        which is reasonable, since there holds 
        \begin{align*}
			\lambda^{-\theta +\frac12 } \ll \lambda^{-4}\ll \lambda^{-\frac{\mu}{2} (\frac{1}{\delta} + \kappa)}\
		\end{align*}
for small $\lambda$ by choosing suitable $\theta$. For example, take $\theta=1$, then  $\mu=\frac56-\epsilon>\frac13$ with $\frac13<\epsilon<\frac12$, and 
$-\frac12>-\frac{\mu}{2} (\frac{1}{\delta} + \kappa)$ holds obviously.
Hence, one can choose $N=\lambda^{-3}$ satisfying \eqref{eq:Gap} and by \eqref{eq:Energy}
\begin{align*}
			E(\lambda,N)&:=\|\nabla d_{\text{in}}\|_{L^2}^2 \ge \|\partial_y d_{\text{in}}\|_{L^2}^2 \geq C \lambda^{-5}, 
		\end{align*}
which goes to $\infty$ as $\lambda\rightarrow 0.$
        
\end{Rem}

	\subsection{Stability mechanisms}
    The above results show that the Couette flow plays an important role in the stability theory. In fact,
	the stability of Couette flow governed by the Navier--Stokes equations has been a prominent topic in fluid mechanics since the pioneering works of Rayleigh \cite{R1879}, Kelvin \cite{K1887},  Orr \cite{O1907}, and Sommerfeld \cite{S1908}. Consider the vorticity formulation of the 2D linearized Navier–Stokes equations around the Couette flow:
	\begin{equation*}\begin{aligned} 
			\left\{\begin{array}{l}
				\partial_t \omega+y \partial_x \omega-\nu \Delta \omega=0,  \\
				\left. \omega\right|_{t=0}=\omega_{\mathrm {in}}.
			\end{array}\right.	
	\end{aligned}\end{equation*}
	By Fourier transform, the Kelvin's solution is
	$$
	\hat{\omega}(t, k, \xi)=\hat{\omega}_{\mathrm{in}}(k, \xi+k t) e^{-\nu \int_0^t|k|^2+|\xi+k(t-s)|^2 d s},
	$$
	which satisfies the following two linear estimates:
	\begin{equation}\begin{aligned} \label{eq:enhanced dissipation}
			|\hat{\omega}(t, k, \xi)| \leq C\left|\hat{\omega}_{\mathrm{in}}(k, \xi+k t)\right| e^{-c \nu^{\frac{1}{3}}|k|^{\frac{2}{3}} t},
	\end{aligned}\end{equation}
	\begin{equation}\begin{aligned}  \label{eq:inviscid damping}
			|\hat{\phi_0}(t, k, \xi)| \leq C\langle t\rangle^{-2} \frac{1+|k|^2+|\xi+k t|^2}{|k|^4}\left|\hat{\omega}_{\mathrm{in }}(k, \xi+k t)\right| e^{-c \nu^{\frac{1}{3}}|k|^{\frac{2}{3}} t},
	\end{aligned}\end{equation}
	for some $c>0$, where $\phi_0=\Delta^{-1} \omega$ is the stream function.
	
	\subsubsection{Enhanced dissipation}
	Inequality \eqref{eq:enhanced dissipation} represents the enhanced dissipation estimate. Its associated time scale is $O\left(\nu^{-1/3}\right)$, which is significantly shorter than the standard heat dissipation time scale $O\left(\nu^{-1}\right)$ for small $\nu$, implying a much faster decay. The phenomenon of enhanced dissipation has been widely observed and studied in the physics literature (see, e.g., \cite{T1888, RY1983, BL1994,LB2001}). 
    Recently, it has garnered considerable attention on the  stability of the 2D Couette flow for the Navier--Stokes equation. For example, if the perturbation for vorticity is in Gevrey-$s$ with $s < 2$, Bedrossian--Masmoudi--Vicol \cite{BMV2016} showed that the solution is stable for $\gamma =0$ on a periodic domain, and more developments can be found in \cite{BVW2018,CLWZ2020,MZ2022,WZ2023}.  When $\Omega=\mathbb{R} \times \mathbb{R}$, Arbon--Bedrossian \cite{AB2025}  and  Li--Liu--Zhao \cite{LLZ2025} obtained almost sharp stability threshold. For further references,
    we refer the reader to \cite{DWZ2021,WZ2021,CWZ2024,CWY2025a} and the references therein.
	In addition, enhanced dissipation also plays a crucial role in the suppression of blow-up in the 2D Keller--Segel(--Navier--Stokes) system \cite{KX2016, BH2017, ZZZ2021, CW2024, LXX2025}. For the high dimensional or supercritical cases we refer the reader to \cite{DSW2025, CWWWY2025,CWWW2025,CWYZ2025} and the references therein for some recent developments.
	
	\subsubsection{Inviscid damping}
	Inequality \eqref{eq:inviscid damping} represents the so-called inviscid damping estimate, which is attributed to the mixing of vorticity induced by the shear flow. This phenomenon is analogous to Landau damping in plasma physics, first discovered by Landau \cite{L1946}. For general shear flows, establishing (non)linear inviscid damping remains a challenging problem. However, in a series of works \cite{WZZ2018,WZZ2019,WZZ2020a}, Wei--Zhang--Zhao established linear inviscid damping for both monotone and certain non-monotone flows, including the Poiseuille and Kolmogorov flows. We refer the reader to \cite{MV2011, BM2015, Z2016, BCV2019,IJ2020,J2020} and the references therein for some recent developments.
	

    \subsection{The key idea in the proof}

The proof of the main theorem relies on a bootstrap argument involving carefully constructed space-time norms. To close the energy estimates, particularly for the strong coupling between the Navier--Stokes equations and the director field equation, we introduce several technical innovations.

\subsubsection{Construction of the multipliers}
The first key idea is the construction of time-dependent Fourier multipliers to capture the stabilizing effects of the Couette flow. Motivated by the works on the stability of shear flows \cite{BGM2017, DWZ2021, WZ2023}, we introduce two self-adjoint Fourier multipliers acting as ``ghost weights" to provide additional dissipation properties.

For $k, \xi \in \mathbb{R}$ and the shear amplitude $A > 0$, we define:
\begin{align*}
    \mathcal{M}_1(k, \xi) &:= \arctan \left( A^{-\frac{1}{3}} |k|^{-\frac{1}{3}} \operatorname{sgn}(k)\xi \right) + \frac{\pi}{2}, \\
    \mathcal{M}_2(k, \xi) &:= \arctan \left( \frac{\xi}{k} \right) + \frac{\pi}{2}.  
\end{align*}
The total multiplier is defined as $\mathcal{M} := \mathcal{M}_1 + \mathcal{M}_2 + 1$, satisfying the boundedness condition 
\begin{align} \label{eq:bound of M}
    1 \le \mathcal{M} \le 1 + 2\pi.
\end{align}

\begin{itemize}
    \item $\mathcal{M}_1$ is the \textbf{enhanced dissipation multiplier}. It captures the enhanced dissipation time scale $O(\nu^{-1/3})$, which is significantly shorter than the standard heat dissipation time scale. This is crucial for suppressing the blow-up of the director field.
    \item $\mathcal{M}_2$ is the \textbf{inviscid damping multiplier}. It tracks the mixing of vorticity induced by the shear flow, analogous to Landau damping.
\end{itemize}

A crucial feature of these multipliers is their commutator with the transport operator $\partial_t + y\partial_x$: 
	\begin{equation} \label{eq:crucial M}
		2 \Re\big(\left(\partial_t+y \partial_x\right) f \mid \mathcal{M}_i f\big)=\frac{d}{d t}\lt\|\sqrt{\mathcal{M}_i} f\rt\|_{L^2}^2+\int_{\mathbb{R}^2}  k \partial_{\xi}  \mathcal{M}_i(k, \xi)|\hat{f}|^2 d k d \xi,
	\end{equation}
	for a sufficiently smooth function $f= f(t, x ,y)$ on $(0, +\infty) \times \mathbb{R}^2$. As shown in \cite{CWY2025b,LLZ2025}, we have the coercivity property:
\begin{equation} \label{eq:m}
    \int_{\mathbb{R}^2} k\partial_{\xi} \mathcal{M}(k, \xi) |\hat{f}|^2 d k d \xi \ge \frac{1}{4A^{\frac{1}{3}}} \lt\| |D_x|^{\frac{1}{3}} f \rt\|_{L^2}^2 - \frac{1}{2A} \lt\|\partial_y f\rt\|_{L^2}^2 + \lt\|\partial_x \nabla \Delta^{-1} f\rt\|_{L^2}^2.
\end{equation}
This allows us to recover a fractional derivative term $\lt\| |D_x|^{ 1/3} f \rt\|_{L^2}^2$ with a positive sign, which is essential for controlling the high-frequency interactions in the bootstrap argument.

\subsubsection{Frequency-based two-layer decomposition} \label{subsubsec132}
To address the complex nonlinear terms, particularly the geometric constraint term $|\nabla d|^2 d$ in the director equation and the stress tensor $\nabla \cdot (\nabla d \odot \nabla d)$ in the fluid equation, which involve convolutions in the frequency space that are difficult to estimate using standard paraproducts, we introduce the following frequency-based two-layer decomposition.
Motivated by \cite{LLZ2025}, for a fixed frequency $k$ and a convolution variable $l$, we divide the integration domain into three distinct regions based on the relative sizes of frequencies:
\begin{itemize}
    \item \textbf{Region I (Near-resonant interactions):} $\mathcal{R}_{\text{res}}(k, l) := \mathbb{R}\times \{l : \frac{1}{2}|k-l| \le |k| \le 2|k-l|\}$. In this region, the interacting frequencies are comparable, i.e., $|k| \sim |k-l|$.
    \item \textbf{Region II (High-Low interactions):} $\mathcal{R}_{\text{HL}}(k, l) := \mathbb{R}\times \{ l : |k| > 2|k-l|\}$. Here, the output frequency $k$ is much larger than the difference $k-l$, implying $|k| \sim |l|$.
    \item \textbf{Region III (Low-High interactions):} $\mathcal{R}_{\text{LH}}(k, l) :=\mathbb{R}\times \{ l: 2|k| < |k-l|\}$. Here, the output frequency is small compared to the inputs, implying $|k-l| \sim |l|$.
\end{itemize}
{\bf One new observation} is the following {\bf  frequency-based two-layer decomposition method}. Consider the form in Region I of like
\begin{align*} 
\int_{\mathcal{R}_{\text{res}} }\left(\int_{\mathbb{R}}\cdots d\eta\right)dkdl 
&=\int_{\mathcal{R}_{\text{res}} }\left(\int_{ \{\eta : \,\frac{1}{2}|k-l-\eta| \le |k-l| \le 2|k-l-\eta|\} }\cdots d\eta\right)dkdl\\
& \quad +\int_{\mathcal{R}_{\text{res}} }\left(\int_{ \{\eta :\, |k-l|>2|k-l-\eta|\} } \cdots d\eta\right)dkdl+\int_{\mathcal{R}_{\text{res}} }\left(\int_{ \{\eta :\, 2|k-l| < |k-l-\eta|\} } \cdots d\eta\right)dkdl \alabel{eq:twolayer}
\end{align*} 
which allows us to balance the weights $\langle k \rangle^m \langle 1/k \rangle^{\epsilon}$ and the derivative losses across different frequency modes effectively.
For example, in the proof of Lemma \ref{lem:est of dx13 nabla d2 d}, 
		using Fourier transform  we obtain
		\begin{align*} 
			& \quad   \left|\Re\left( |D_x|^{\frac13}  \left(  |\partial_x d|^2 d \right) \left\lvert\, \mathcal{M} e^{2 a A^{-\frac{1}{3}}\left|D_x\right|^{\frac{2}{3} }t }\langle D_x\rangle^{2 m  }\langle\frac{1}{D_x}\rangle^{2 \epsilon} |D_x|^{\frac13} d\right.\right)\right|\\
			&\lesssim  \int_{\mathbb{R}^3} e^{2 a A^{-\frac{1}{3}}|k|^{\frac{2}{3} } t }\langle k\rangle^{2 m   }\langle\frac{1}{k}\rangle^{2 \epsilon} |k|^\frac23 \\
			&\qqquad \quad \times \left|\int_{\mathbb{R}} \mathcal{M}\left(k, D_y\right)   d_k(y) \cdot   d_l(y)  \lt(k-l-\eta\rt) \eta  d_{k-l-\eta}(y) \cdot d_{\eta}(y)   d y\right| d k d ld\eta
		\end{align*}
		By \eqref{eq:trick1} and Corollary \ref{cor:L infty} ($\alpha=\beta=\frac13$), the term on the right hand in $\mathcal{R}_{\text{res}}(k, l)$ is controlled by
		\begin{align*}  
			& C \lt\|e^{a A^{-\frac{1}{3}}|k|^{\frac{2}{3} }t}\langle k\rangle^{m  }\langle\frac{1}{k}\rangle^\epsilon |k|^\frac23  \|d_k\|_{L_y^2}\rt\|_{L_k^2}     \lt\|e^{a A^{-\frac{1}{3}}|l|^{\frac{2}{3} } t } \| d_l \|_{L_y^{\infty}}\rt\|_{L_l^1} \\
			&\quad \times\left\|e^{a A^{-\frac{1}{3}}|k-l|^{\frac{2}{3} } t }\langle k-l\rangle^{m   } \langle\frac{1}{k-l}\rangle^\epsilon    \int_{\mathbb{R}}|k-l-\eta| |\eta| \| d_{k-l-\eta}\|_{L_y^2} \|d_{\eta}\|_{L_y^\infty} d\eta \right\|_{L_{k-l}^2},
		\end{align*}
		where the last term on the right side can be estimated using the idea from  (\ref{eq:twolayer}) again ( more details can be found in Lemma \ref{lem:Preliminary estimates}). We remark that the method  for handling the convolution of multiple functions could potentially be applied to many more general physical models, such as general Ericksen--Leslie systems with Leslies coefficients, the Landau--Lifshitz--Gilbert equation, and others.

\subsubsection{Energy transfer mechanism and blow-up suppression}

Through the energy estimates derived in Sections \ref{sec.3}--\ref{sec.6}, we demonstrate that the enhanced dissipation effect, quantified by the term $A^{1/6}\lt\| |D_x|^{1/3} d \rt\|$, dominates the nonlinear focusing effects when $A$ is sufficiently large. Specifically, the strong shear flow mixes the director field $d$ rapidly in the $x$-direction, transferring energy to high frequencies where dissipation is efficient. This mechanism prevents the concentration of energy required for the formation of singularities, thereby extending the local solution to a global one under the condition $A \gg 1$.

The final key idea is utilizing the large shear parameter $A$ to suppress finite-time blow-up. The coupling between the fluid and the crystal induces a complex energy transfer mechanism, which we quantify through the following hierarchy of estimates:

\begin{itemize}
    \item \textbf{Horizontal Regularity (Propositions \ref{lem:est of |D_x|13 d}):}
    The proposition control the horizontal derivative $|D_x|^{1/3} d$. The enhanced dissipation provides strong decay factors with respect to $A$. Schematically we have:
    \begin{equation*}
       \lt \|  |D_x|^{1/3} d \rt\|_{X}^2 \lesssim \|\text{data}\|^2 + A^{-1/2}\|\omega\|_{X} \lt\| |D_x|^{1/3} d\rt\|_{X} + A^{-1/6} \|\text{nonlinear terms}\|,
    \end{equation*}
where the factor $A^{-1/6}$ is crucial for absorbing the nonlinear growth.

    \item \textbf{Mixed Regularity (Proposition \ref{lem:est of pypy dx13 d}):}
    This proposition handles the mixed derivative $\nabla^2 |D_x|^{1/3}   d$. Due to the shear transport term $y\partial_x$, a derivative loss of order $A$ typically occurs when estimating vertical derivatives. We obtain:
    \begin{equation*}
        \lt\| \partial_y |D_x|^{1/3} \nabla d \rt\|_{X}^2 \lesssim \|\text{data}\|^2 + A \lt\| |D_x|^{1/3} d \rt\|_{X} \lt\| \partial_y |D_x|^{1/3} \nabla d \rt\|_{X} + \dots,
    \end{equation*}
    where the second item from the right is {\bf the most troublesome one}.
    Although a factor of $A$ appears, it is multiplied by the lower-order norm $\lt\| |D_x|^{1/3} d \rt\|_{X}$, which is small (controlled by Proposition  \ref{lem:est of |D_x|13 d}). This allows us to close the estimate despite the shear-induced growth.

    \item \textbf{Fluid Vorticity (Proposition \ref{lem:est of omega}):}
    The fluid vorticity $\omega$ is driven by the Leslie stress $\nabla \cdot (\nabla d \odot \nabla d)$. The estimate takes the form:
    \begin{equation*}
        \| \omega \|_{X}^2 \lesssim \|\text{data}\|^2 + A^{-1/2} \|\omega\|_{X}^3 + A^{-1/3} \|\omega\|_{X} \left(\lt\| |D_x|^{1/3} \nabla d \rt\|_{X}^2 + \lt\| |D_x|^{1/3} \lt(\px^2, \py^2\rt) d \rt\|_{X}^2\right) ,
    \end{equation*}
   which shows that the feedback from the director field to the fluid is suppressed by $A^{-1/2}$ or $A^{-1/3}$.
\end{itemize}

The system exhibits a cyclic dependency: the director field estimates  rely on the smallness of the velocity field $\|\omega\|_X$, while the vorticity estimate (Prop. \ref{lem:est of omega}) relies on the control of the director field stress $\lt\|\nabla^2 |D_x|^{1/3} d\rt\|_X$. 
By choosing $A$ sufficiently large, the decay factors ($A^{-1/6}, A^{-1/2}$) in the nonlinearities and coupling terms dominate the destabilizing factors (such as the linear growth $A$), thereby closing the bootstrap argument in Proposition \ref{main prop2}.

	\subsection{Notations and Outlines}
	
	Here are some notations used in this paper.
	
	\noindent\textbf{Notations}:
	\begin{itemize}
		\item For a given function $f(x,y)$ on $\mathbb{R}^2$, its $k$-th horizontal Fourier modes can be defined by
		\begin{equation} \label{eq:def of fk}
			f_k(y)=\mathcal{F}_{x \rightarrow k}(f)(k, y) = (2\pi)^{-\frac12}\int_{\mathbb{R}} f(x, y) \mathrm{e}^{- ikx} dx,
		\end{equation}
		In addition, denote 
		$$
		\hat{f}_k(\xi)=\hat{f}(k, \xi) = \mathcal{F}_{y \rightarrow \xi}(f_k)(y) = (2\pi)^{-1}  \int_{\mathbb{R}^2} f(x, y) \mathrm{e}^{- i(kx + \xi y)} dxdy.
		$$
		For another given function $g(x,y)$ on $\mathbb{R}^2$, note that
		\begin{equation*}
			\mathcal{F}_{x \rightarrow k}(f\cdot g)(k, y) = (2\pi)^{-\frac12}  f_k(y)*g_k(y).
		\end{equation*}
		\item Denote $C$ by  a positive constant independent of $A$, $t$ and the initial data, and it may be different from line to line. $B \lesssim D$ means there exists an absolute constant $C>0$, such that $B \leq C D$. $B \sim D$ means that there exists a constant $C>0$ such that
        $
         C^{-1} B \le D \le C B .
        $

		\item The space norm $\|f\|_{L^{p}}$ is defined by	
		$\|f\|_{L^{p}(\mathbb{R}^2)}=\left(\int_{\mathbb{R}^2}|f|^p dxdy\right)^{{1}/{p}}  $. For $t\geq0$, the time-space norm $\|f\|_{L^{q}L^{p}}$ is defined by	
		$\|f\|_{L^qL^p}=\lt\|\|f\|_{L^p(\mathbb{R}^2)}\rt\|_{L^q(0,t)}$.
		For simplicity, we write $\| (f,g)\|_{L^p} = \sqrt{\|f\|_{L^p}^2 + \|g\|_{L^p}^2 }$. 
		\item Denote $(f | g)$ by the $L^2\left(\mathbb{R}^2\right)$ inner product of $f$ and $g$.
        \item Denote $\nabla_k :=(k,\, \py)$ and $\Delta_k := \py^2 - k^2 $.
        \item We introduce two notations to make our proofs more concise.
\[
\p = 
\begin{cases}
\px, & i=1, \\
\py,  & i=2,
\end{cases}
\qquad 
\q = 
\begin{cases}
\hat{\px}, &i=1, \\
\py,        & i=2,
\end{cases}\alabel{eq:ph-fr}
\]
where $\hat{\px}$ denotes $\px$ on the Fourier side (Specifically, it may be $k$, $k-l$, $k-l-\eta$ and etc in different situations, and it varies from case to case).
	\end{itemize}

	 \noindent\textbf{Outlines}:
The rest of the paper is organized as follows. In Sec. \ref{sec.2}, we present some preliminaries and give the proof of the main theorem under Prop. \ref{main prop2}. In Sec. \ref{sec.3}, we derive the energy estimate for $|D_x|^{1/3} d$ (Prop. \ref{lem:est of |D_x|13 d}), including estimates for the elastic nonlinearity and the nonlinear energy. Sec. \ref{sec.4} is devoted to the energy estimate for the higher-order derivative terms $(\partial_x^2, \partial_y^2)|D_x|^{1/3} d$ (Prop. \ref{lem:est of pypy dx13 d}). In Sec. \ref{sec.5}, we establish the energy estimate for the vorticity $\omega$ (Prop. \ref{lem:est of omega}). Finally, Sec. \ref{sec.6} completes the proof of Prop. \ref{main prop2}.

	\section{Preliminaries and proof of the main theorem} \label{sec.2}

    In this section we introduce some useful lemmas and complete our main theorem under an important proposition. 
    
    \subsection{Space–time estimates }
	
	Let $u$ be determined by $\eqref{eq:main}_2$, and $f$ satisfy
	\begin{equation}\begin{aligned} \label{eq:f}
			\left\{\begin{array}{l}
				\partial_t f+y \partial_x f-\frac{1}{A} \Delta f=-\frac{1}{A}u\cdot\nabla f+g,\quad (t,x,y) \in (0, T)\times \mathbb{R}^2, \\
				\left. f\right|_{t=0} = f_{\rm in}, \quad (x,y) \in  \mathbb{R}^2,
			\end{array}\right.
	\end{aligned}\end{equation}
	where $g$ represents an additional term. We recall the space-time estimate for \eqref{eq:f}.
	\begin{Prop}[Proposition 2.3 in \cite{CWY2025b}] \label{prop:est of f}
		Let $u$ be determined by $\eqref{eq:main}_2$ and $0<a<\tfrac{1}{16(1+2 \pi)}$.
		Then, for $0<\epsilon<\frac12<m$ and $0\leq t\leq T$,
		\begin{equation*}\begin{aligned}
				\lt \|    f \rt \|_{X_{a,m,\epsilon}}^2 
				&\lesssim    
				\lt\|   f_{\mathrm{in}} \rt\|_{L^2}^2 +   \int_{0}^{t} \left| \Re\left(g \left\lvert\, \mathcal{M} e^{2 a A^{-\frac{1}{3}}\left|D_x\right|^{\frac{2}{3} }t }\langle D_x\rangle^{2 m}\langle\frac{1}{D_x}\rangle^{2 \epsilon} f\right.\right) \right| dt  \\
				&\quad + \frac1{A^\frac12}\|         \omega \|_{X_{a,m,\epsilon}}   \|         f \|_{X_{a,m,\epsilon}}^2 .
		\end{aligned}\end{equation*}
Especially,
\begin{align*}
 &\quad\frac{1}{A}\int_{0}^{t} \left| \Re\left( u\cdot\nabla f\left\lvert\, \mathcal{M} e^{2 a A^{-\frac{1}{3}}\left|D_x\right|^{\frac{2}{3} }t }\langle D_x\rangle^{2 m}\langle\frac{1}{D_x}\rangle^{2 \epsilon} f\right.\right) \right| dt\nonumber\\
 &\lesssim \frac{1}{A}\int_{0}^{t} \int_{\mathbb{R}^2} e^{2 a A^{-\frac{1}{3}}|k|^{\frac{2}{3} } t }\langle k\rangle^{2 m}\langle\frac{1}{k}\rangle^{2 \epsilon}\left|\int_{\mathbb{R}} \mathcal{M}\left(k, D_y\right) f_k(y) \partial_y \Delta_l^{-1} \omega_l(y)  (k-l) f_{k-l}(y) d y\right| d k d l dt\nonumber\\
 &\quad + \frac{1}{A}\int_{0}^{t} \int_{\mathbb{R}^2} e^{2 a A^{-\frac{1}{3}}|k|^{\frac{2}{3} } t }\langle k\rangle^{2 m}\langle\frac{1}{k}\rangle^{2 \epsilon}\left| \int_{\mathbb{R}} \mathcal{M}\left(k, D_y\right) f_k(y) l \Delta_l^{-1} \omega_l(y)  \py f_{k-l}(y) d y  \right|d k d l dt\nonumber\\
 &\lesssim \frac1{A^\frac12}\|         \omega \|_{X_{a,m,\epsilon}}   \|         f \|_{X_{a,m,\epsilon}}^2, \alabel{eq:uf-estimate}
 \end{align*}
 where $f_k(y)$ and $X_{a,m,\epsilon}$ are defined in \eqref{eq:def of fk} and  \eqref{eq:X norm}, respectively.
	\end{Prop}

Next we give a direct estimate for nonlinear terms of the form $(|D_x|^{1/3}u) \cdot \nabla f$.
	
	\begin{Lem}[Lemma A.3 in \cite{CWY2025b}] \label{lem:est of dx u}
		Let $u$ be determined by $\eqref{eq:main}_2$ and $0<a<\tfrac{1}{16(1+2 \pi)}$. For $0<\epsilon<\frac12<m$, $0\leq t\leq T$ and certain smooth enough $f$,  
		\begin{align*}    
				&\quad\frac1A \int_{0}^{t} \left|\Re\left(|D_x|^{\frac13}u \cdot \nabla f \left\lvert\, \mathcal{M} e^{2 a A^{-\frac{1}{3}}\left|D_x\right|^{\frac{2}{3} }t }\langle D_x\rangle^{2 m   }\langle\frac{1}{D_x}\rangle^{2 \epsilon} |D_x|^\frac13  f  \right.\right)\right|   dt\\
                &\lesssim \frac{1}{A}\int_{0}^{t} \int_{\mathbb{R}^2} e^{2 a A^{-\frac{1}{3}}|k|^{\frac{2}{3} } t }\langle k\rangle^{2 m}\langle\frac{1}{k}\rangle^{2 \epsilon}|k|^\frac13 |l|^{\frac13}\left|\int_{\mathbb{R}} \mathcal{M}\left(k, D_y\right) f_k(y) \partial_y \Delta_l^{-1} \omega_l(y)  (k-l) f_{k-l}(y) d y\right| d k d l dt\\
 &\quad + \frac{1}{A}\int_{0}^{t} \int_{\mathbb{R}^2} e^{2 a A^{-\frac{1}{3}}|k|^{\frac{2}{3} } t }\langle k\rangle^{2 m}\langle\frac{1}{k}\rangle^{2 \epsilon}|k|^\frac13 |l|^{\frac13}\left| \int_{\mathbb{R}} \mathcal{M}\left(k, D_y\right) f_k(y) l \Delta_l^{-1} \omega_l(y)  \py f_{k-l}(y) d y  \right|d k d l dt\\
				&\lesssim \frac1{A^\frac12}\|    \omega \|_{X_{a,m,\epsilon}}    \lt\|    |D_x|^\frac13 f \rt\|_{X_{a,m,\epsilon}}^2   . \alabel{eq:dx13 uf-estimate}
		\end{align*}
	\end{Lem} 
    \subsection{$L^\infty$ embedding}

	The following is a series of lemmas about the $L^\infty$ embedding, which will be frequently used.
	\begin{Lem} \label{lem:l infty embedding}
		For any $0< \delta_1 < \delta_2$ and for certain sufficiently smooth function $f= f(x, y)$ that vanishes at infinity, there exists a positive constant $C$ depending on $\delta_1,\, \delta_2$, such that
		$$
		\|f \|_{L^{\infty}} \leq C \lt\|\left|D_x\right|^{\delta_1} \langle \frac{1}{D_x} \rangle^{\delta_2} \nabla f \rt\|_{L^2},
		$$
		where $\lt\|\left|D_x\right|^{\delta_1} \langle \frac{1}{D_x} \rangle^{\delta_2} \nabla f \rt\|_{L^2}^2:=\lt\|\left|k\right|^{\delta_1} \langle \frac{1}{k} \rangle^{\delta_2} (k,\xi) \hat{f} \rt\|_{L^2}^2$.
	\end{Lem}
	\begin{proof}
		From Fourier inversion formula and Plancherel identity, it follows that
		$$
		\begin{aligned}
			\|f\|_{L^{\infty}} & \leq   \int_{\mathbb{R}^2 }\left|\widehat{f}_k(\xi)\right| dk d\xi  \leq   \int_{\mathbb{R}^2 } \frac{1}{\sqrt{k^2+\xi^2}}\left|(k, \xi)\widehat{f}_k(\xi)\right| d k d\xi \\
			&\lesssim \int_{  \mathbb{R} } \frac{1}{|k|^{\frac{1}{2}}}\lt\| (k, \xi)\widehat{  f }_k( \xi)\rt\|_{L_\xi^2} dk \lesssim \int_{  \mathbb{R} } \frac{1}{|k|^{\frac{1}{2} +  {\delta_1} } \langle \frac1{k}\rangle^{ {\delta_2} }} \lt\| |k|^{  {\delta_1} } \langle \frac1{k}\rangle^{ {\delta_2} } (k, \xi)\widehat{  f }_k( \xi) \rt\|_{L_\xi^2} dk \\
			& \lesssim  \lt\|\left|D_x\right|^{\delta_1} \langle \frac{1}{D_x} \rangle^{\delta_2} \nabla f \rt\|_{L^2},
		\end{aligned}
		$$
        where we used 
        $$\int_{  \mathbb{R} } \frac{1}{|k|^{1 +  {2\delta_1} } \langle \frac1{k}\rangle^{ {2\delta_2} }}dk \lesssim 1 $$
        due to $0< \delta_1 < \delta_2$.
        
	\end{proof}
	\begin{Cor} \label{cor:l infty}
		Let $u \in H^1(\mathbb{R}^2)$ be determined by $\eqref{eq:main}_2$. Then, there exist positive constants $C$ depending on $\epsilon$ and $m$, such that
		\begin{equation*}\begin{aligned}
				\| u \|_{L^\infty} &\leq C \|   \omega \|_{Y_{m, \epsilon}}, \quad 0< \epsilon \leq 2m, \\
				\| \nabla d \|_{ L^\infty } \leq   &    C
                \lt\|    \lt(\px^2,   \py^2\rt) |D_x|^\frac13  d \rt\|_{{Y_{m, \epsilon}}}, \quad \frac13 <\epsilon \leq 2m,
		\end{aligned}\end{equation*}
        where $Y_{m, \epsilon}$ is defined in \eqref{eq:def of Y}.
	\end{Cor}
	\begin{proof}
		By Lemma \ref{lem:l infty embedding} and $\Delta u=\nabla^{\perp}\omega$, 
		we have
		\begin{equation*}\begin{aligned}
				\| u \|_{L^\infty} \lesssim \lt\| |D_x |^{\frac{\epsilon}{2}} \langle \frac{1}{D_x} \rangle^{\epsilon} \nabla u \rt\|_{L^2} \lesssim \|  \omega \|_{Y_{m, \epsilon}}, 
		\end{aligned}\end{equation*}
		provided with $0 < \frac \epsilon 2 \leq m$.
		Moreover, if $ \frac16< \frac{\epsilon}{2} \leq m $, then it follows from Lemma \ref{lem:l infty embedding} that
		\begin{equation*}\begin{aligned}
				\| \px d \|_{L^\infty} &\lesssim  \lt\| |D_x |^{\frac{\epsilon}{2} - \frac16} \langle \frac{1}{D_x} \rangle^{\epsilon - \frac13 } \px \nabla d \rt\|_{L^2} \\
				&\lesssim \lt\|   \px^2 |D_x|^\frac13 d \rt\|_{Y_{m, \epsilon}} +\lt\|  \px \py |D_x|^\frac13  d \rt\|_{Y_{m, \epsilon}}
		\end{aligned}\end{equation*}
		and
		\begin{equation*}\begin{aligned}
				\| \py d \|_{L^\infty} &\lesssim   \lt\| |D_x |^{\frac{\epsilon}{2} - \frac16} \langle \frac{1}{D_x} \rangle^{\epsilon - \frac13 }   \nabla \py d \rt\|_{L^2} \\
				&\lesssim \lt\|    \py^2 |D_x|^\frac13  d \rt\|_{Y_{m, \epsilon}} + \lt\|   \px \py |D_x|^\frac13  d \rt\|_{Y_{m, \epsilon}}.
		\end{aligned}\end{equation*}
	\end{proof}

	\begin{Lem}[see, for example,  Lemma A.2 in \cite{CWY2025b}]
		For $f_l= f_l(y) \in H^1(\mathbb{R})$ (defined in \eqref{eq:def of fk}), it holds that
		\begin{align}
			\| f_l(\cdot)\|_{L^\infty} &\leq    \|   f_l(\cdot)\|_{L^2}^\frac12  \|   \py f_l(\cdot)\|_{L^2}^\frac12, \label{eq:GN1}\\
			\| f_l(\cdot)\|_{L^\infty} &\leq  |l|^{-\frac12} \| \nabla_l f_l(\cdot)\|_{L^2}. \label{eq:GN2}
		\end{align}
	\end{Lem}

    Using \eqref{eq:GN1}, we get the following corollary immediately.
	\begin{Cor} \label{cor:L infty}
		For certain sufficiently smooth function $f_l= f_l(y)$, $\alpha,\beta\geq0$ and $-m<\frac{\alpha+\beta-1}{2}<\epsilon$, we have
		\begin{align*}
			\lt\|e^{a A^{-\frac{1}{3}}|l|^{\frac{2}{3} } t }     \| f_l \|_{L_y^{\infty}}\rt\|_{L_l^1} 
			& \leq \lt\|e^{a A^{-\frac{1}{3}}|l|^{\frac{2}{3} }t}\langle l\rangle^m\langle\frac{1}{l}\rangle^\epsilon     \lt\|  |l|^\alpha  f_l \rt\|_{L_y^2}^\frac12 \|  |l|^\beta \py  f_l \|_{L_y^2}^\frac12 \rt\|_{L_l^2}\lt\| |l|^{-\frac{\alpha + \beta}{2}} \langle l\rangle^{-m}\langle\frac{1}{l}\rangle^{-\epsilon} \rt\|_{L_l^2}  \\
			&\lesssim  \lt\|e^{a A^{-\frac{1}{3}}\left|D_x\right|^{\frac{2}{3} }t}   |D_x|^\alpha f\rt\|_{Y_{m,\epsilon}}^\frac12 \lt\|e^{a A^{-\frac{1}{3}}\left|D_x\right|^{\frac{2}{3} }t}   |D_x|^\beta \py f\rt\|_{Y_{m,\epsilon}}^\frac12  .
		\end{align*}
	\end{Cor}

	\subsection{Inequalities  in different frequency regions}

    Recall the definitions of $\mathcal{R}_{\text{res}}(k,l)$, $\mathcal{R}_{\text{HL}}(k,l) $ and $\mathcal{R}_{\text{LH}}(k,l) $ in the last Section \ref{subsubsec132}, and we have the following inequality estimates between different frequencies. 
\begin{Lem}[Estimates in frequency regions] \label{lem:freq_estimates}
 		The following inequalities hold in the respective regions:
		
		\begin{enumerate}
			\item \textbf{In $\mathcal{R}_{\text{res}}(k,l)$ :} For any $s, s_1,s_2 \ge 0$, 
			\begin{gather*}
				|l|^{s} \lesssim |k|^{s}\sim |k-l|^{s}\quad \text{and}\quad
				\langle k\rangle^{s_1 }\langle\frac{1}{k}\rangle^{s_2} \lesssim \langle k-l\rangle^{s_1 }\langle\frac{1}{k-l}\rangle^{s_2}; \alabel{eq:trick1}
			\end{gather*}
			
			\item \textbf{In $\mathcal{R}_{\text{HL}}(k,l) $:} For any $s, s_1,s_2 \ge 0$, 
			\begin{gather*}
				|k-l|^{s} \lesssim |k|^{s}\sim |l|^{s} \quad \text{and}\quad
				\langle k\rangle^{s_1 }\langle\frac{1}{k}\rangle^{s_2} \lesssim \langle l\rangle^{s_1 }\langle\frac{1}{l}\rangle^{s_2}; \alabel{eq:trick2}
			\end{gather*}
			
			\item \textbf{In $\mathcal{R}_{\text{LH}}(k,l) $:} For any $s_1, s_2, s_3 \ge 0$ and $s = s_1 + s_2$, 
			\begin{gather*}
				|k|^{s} \lesssim |k-l|^{s}\sim |l|^{s},\quad
				\langle k \rangle^{2s } \lesssim  \langle l \rangle^{s} \langle k-l \rangle^{s}, \alabel{eq:trick3}\\
				1 \leq \langle \frac{1}{k-l} \rangle^{s} |k-l|^{s} \lesssim \langle\frac{1}{l}\rangle^{s_3}\langle\frac{1}{k-l}\rangle^{s_3}\left(1+\langle\frac{1}{k}\rangle^{s-2 s_3} \right)|k-l|^{s_1 }|l|^{s_2}. \alabel{eq:trick3.5}
			\end{gather*}
		\end{enumerate}
	\end{Lem}

    \begin{proof}
It suffices to verify the inequality of \eqref{eq:trick3.5}, and the others follows by direct computations. In fact, it is sufficient to prove
$$
\langle \frac{1}{k-l} \rangle^{s-2s_3}\lesssim  1+\langle\frac{1}{k}\rangle^{s-2 s_3}
$$
due to \eqref{eq:trick3}. One can verify this fact by considering the two cases: $s-2s_3$ is greater than zero or $s-2s_3$ is less than or equal to zero.
    \end{proof}
	
	\subsection{Proof of the main theorem under Proposition \ref{main prop2}} 
	
	In this paper, we use the standard bootstrap arguments to prove the main theorem under Proposition \ref{main prop2}.
	
	
	Let us define $T$ to be the end-point of the largest interval $[0, T]$ such that the following hypotheses hold for all $0 \leq t \leq T$ :
    	\begin{gather*}
		E(t):=       A^{\delta}\||D_x|^{\frac13} d\|_{X_{a,m,\epsilon}}  +   \lt\|\lt(\px^2,\py^2\rt)|D_x|^{\frac13}d \rt\|_{X_{a,m,\epsilon}}   + \| \omega \|_{X_{a,m,\epsilon}}  \leq 2K, \alabel{eq:bootstap2}
	\end{gather*}
	where $a, m, \epsilon, \delta \geq 0$ and  $K\geq 1$  will be determined in the proof. We shall show the following proposition to improve the hypotheses \eqref{eq:bootstap2}.
	\begin{Prop} \label{main prop2}
		Under the same assumptions as in Theorem~\ref{thm:Ericksen--Leslie system}, there exists a positive constant $\bar{A}_1$ satisfying
        $$\bar{A}_1= C(m,\epsilon,\delta)\left(\lt\|    \lt(\px^2,\py^2\rt) |D_x|^\frac13  d_{\mathrm{in}} \rt\|_{Y_{m,\epsilon}} +\|        \omega_{\mathrm{in}} \|_{Y_{m,\epsilon}} 	+1\right) ^{\max\{\frac{2}{\delta-1},24 \}}$$
		such that if $A>\bar{A}_1$ and \eqref{eq:smallness2} holds, then
		$$
		 E(t) \leq K \quad \text{for all } 0 < t < T.
		$$
	\end{Prop}
	\begin{proof}[\bf Proof of Theorem \ref{thm:Ericksen--Leslie system}]
		Proposition \ref{main prop2} with the local well-posedness (see, e.g., \cite{H2011}, \cite{HLX2014} or \cite{HLW2014}) of the system \eqref{eq:main} implies that $T=+\infty$, and thus completes the proof by using Corollary \ref{cor:l infty}.
	\end{proof}


\section{The energy estimate for $|D_x|^{1/3} d$}	 \label{sec.3}

This section is devoted to the proof of the main energy estimate for $|D_x|^{1/3} d$.
The argument is based on a set of preparatory inequalities for weighted Fourier multipliers and their application to the nonlinear terms. The main proposition is as follows.

	\begin{Prop}
	     \label{lem:est of |D_x|13 d}
		Assume that $0\leq t \leq T$. For $0<a<\frac{1}{16(1+2 \pi)}$ and $ \frac16 <\epsilon< \frac12 <m$, 
		\begin{align*} 
			 \lt\|      |D_x|^\frac13 d  \rt\|_{X_{a,m,\epsilon}}^2  
			&\leq  C\lt\|      |D_x|^\frac13  d_{\mathrm{in}} \rt\|_{Y_{m,\epsilon}}^2 	+ \frac{C} {A^\frac12} \lt\|   \omega \rt\|_{X_{a,m,\epsilon}}   \lt \|   |D_x|^\frac13 d\rt \|_{X_{a,m,\epsilon}}^2
			\\
			&\quad + \frac{C}{A^{\frac{1}{3}}}\lt \|    |D_x|^\frac13 d\rt \|_{X_{a,m,\epsilon}}^{2}  \lt  \|   \nabla |D_x|^\frac13 d \rt\|_{X_{a,m,\epsilon}}^2  \\
			&\quad + \frac{C}{A^{\frac{1}{6}}}\lt \|  |D_x|^\frac13 d^1 \rt\|_{X_{a,m,\epsilon}} \lt\|    |D_x|^\frac13 d \rt\|_{X_{a,m,\epsilon}} \lt \|   \nabla |D_x|^\frac13 d 
            \rt\|_{X_{a,m,\epsilon}}  .
		\end{align*}
	\end{Prop}

    Before proving Proposition \ref{lem:est of |D_x|13 d}, we establish two lemmas concerning the nonlinear terms.

	\begin{Lem}    
		\label{lem:est of dx13 nabla d2 d}
		For $0<a<\frac{1}{16(1+2 \pi)}$ and $\frac16<\epsilon< \frac12<m$,
		\begin{align*}
			& \quad   \frac1A \int_{0}^{t} \left|\Re\left( |D_x|^{\frac13}  \left(  |\nabla d|^2 d \right) \left\lvert\, \mathcal{M} e^{2 a A^{-\frac{1}{3}}\left|D_x\right|^{\frac{2}{3} }t }\langle D_x\rangle^{2 m  }\langle\frac{1}{D_x}\rangle^{2 \epsilon} |D_x|^{\frac13} d\right.\right)\right| dt\\
			&\lesssim \frac{1}{A^{\frac{1}{3}}} \lt\|   |D_x|^\frac13 d \rt\|_{X_{a,m,\epsilon}}^2    \lt\|  \nabla |D_x|^\frac13 d \rt\|_{X_{a,m,\epsilon}}^2  .
		\end{align*}
	\end{Lem}

	\begin{Lem}
		\label{lem:est of dx13 nabla d2}
		For $0<a<\frac{1}{16(1+2 \pi)}$ and $0<\epsilon< \frac12 <m$, 
		\begin{align*}
			& \quad  \frac1A \int_{0}^{t} \left|\Re\left( |D_x|^\frac13  \left( |\nabla d|^2 e_1\right) \left\lvert\, \mathcal{M} e^{2 a A^{-\frac{1}{3}}\left|D_x\right|^{\frac{2}{3} }t }\langle D_x\rangle^{2 m}\langle\frac{1}{D_x}\rangle^{2 \epsilon} |D_x|^\frac13  d\right.\right)\right| dt\\
			&\lesssim \frac{1}{A^{\frac{1}{6}}} \lt\|   |D_x|^\frac13 d^1 \rt\|_{X_{a,m,\epsilon}} \lt\|    |D_x|^\frac13 d \rt\|_{X_{a,m,\epsilon}} \lt \| \nabla |D_x|^\frac13 d \rt\|_{X_{a,m,\epsilon}}   .
		\end{align*}
	\end{Lem}
	
	With Lemmas \ref{lem:est of dx13 nabla d2 d} and \ref{lem:est of dx13 nabla d2} in hand, we now prove Proposition \ref{lem:est of |D_x|13 d}.
    
	\begin{proof}[\bf Proof of Proposition \ref{lem:est of |D_x|13 d}.]
		It follows from \eqref{eq:main} that $|D_x|^\frac13 d$ satisfies
		\begin{equation}\begin{aligned} \label{eq:dx 13 d}
				\partial_t |D_x|^\frac13  d +  y \partial_x |D_x|^\frac13  d -  \frac{1}{A} \Delta |D_x|^\frac13  d=  -\frac1 A  |D_x|^\frac13 (u \cdot \nabla d)       +  \frac1A   |D_x|^\frac13  \lt(|\nabla d|^2 (d + e_1)\rt).
		\end{aligned}\end{equation}
	Multiplying \eqref{eq:dx 13 d} by $2 \mathcal{M} e^{2 a A^{-\frac{1}{3}}\left|D_x\right|^{\frac{2}{3}}  t } \langle D_x \rangle^{2 m} \langle\frac{1}{D_x} \rangle^{2 \epsilon} |D_x|^{\frac13}  d$, and using \eqref{eq:bound of M}--\eqref{eq:m}, for $0<a<\frac{1}{16(1+2 \pi)}$ we have
		\begin{align*}
			&\quad \frac{d}{d t}\lt\|\sqrt{\mathcal{M}} e^{a A^{-\frac{1}{3}}\left|D_x\right|^{\frac{2}{3} }t}   |D_x|^\frac13 d\rt\|_{Y_{m,\epsilon}}^2   +    \frac1A \lt\|e^{a A^{-\frac{1}{3}}\left|D_x\right|^{\frac{2}{3} }t}   \nabla|D_x|^\frac13 d\rt\|_{Y_{m,\epsilon}}^2 \\
			& \quad\quad +\frac{1}{16 A^{\frac{1}{3}}}\lt\|e^{a A^{-\frac{1}{3}}\left|D_x\right|^{\frac{2}{3} } t }   |D_x|^\frac23 d\rt\|_{Y_{m,\epsilon}}^2   + \lt \|e^{a A^{-\frac{1}{3}}\left|D_x\right|^{\frac{2}{3} }t }   \partial_x \nabla  \Delta^{-1}  |D_x|^\frac13 d  \rt\|_{Y_{m,\epsilon}}^2 \\
			&\leq  \frac2A  \left| \Re\left(|D_x|^{\frac13}    \lt( u \cdot \nabla d\rt) \left\lvert\, \mathcal{M} e^{2 a A^{-\frac{1}{3}}\left|D_x\right|^{\frac{2}{3} } t}\langle D_x\rangle^{2 m}\langle\frac{1}{D_x}\rangle^{2 \epsilon}  |D_x|^\frac13 d\right.\right)  \right|\\
			&\quad + \frac2A  \left| \Re\left( |D_x|^{\frac13}    \lt(|\nabla d|^2 (d + e_1)\rt) \left\lvert\, \mathcal{M} e^{2 a A^{-\frac{1}{3}}\left|D_x\right|^{\frac{2}{3} } t}\langle D_x\rangle^{2 m}\langle\frac{1}{D_x}\rangle^{2 \epsilon}   |D_x|^\frac13 d\right.\right)  \right|.
            \alabel{eq:dx 13 d'}
		\end{align*}
       Here we utilize a key estimate in Fourier space:
\begin{align*}
(fg \bigm||D_x|^\frac23 h) &= \bigl( |k|^\frac13 \widehat{fg} \bigm| |k|^\frac13 \hat{h} \bigr) \\
&\lesssim \int_{\mathbb{R}^2} |k|^\frac13 |k-l|^\frac13 |\hat{f}(k-l) \hat{g}(l) \hat{h}(k)|  dkdl \\
&\quad + \int_{\mathbb{R}^2} |k|^\frac13 |l|^\frac13 |\hat{f}(k-l) \hat{g}(l) \hat{h}(k)| dkdl.
\end{align*}
This inequality holds for suitable functions $f$, $g$ and $h$. Formally decomposing $|D_x|^{\frac13}(u\cdot\nabla d)$ into two corresponding terms, we can control it via $|D_x|^{\frac13} u\cdot \nabla d$ and $u\cdot\nabla|D_x|^{\frac13}d$ when performing the $L^2$ energy estimate for \eqref{eq:dx 13 d}.
Then, applying \eqref{eq:uf-estimate} in   Proposition \ref{prop:est of f} ($f=|D_x|^{\frac13}d$) to the term $u\cdot\nabla |D_x|^{\frac13}d$ in \eqref{eq:dx 13 d'} and using \eqref{eq:dx13 uf-estimate} in Lemma \ref{lem:est of dx u} ($f=d$) to the term $ |D_x|^{\frac13}u\cdot\nabla d$, 
        we have
		\begin{align*} 
			&\quad\lt \|      |D_x|^\frac13  d \rt \|_{X_{a,m,\epsilon}}^2 \\ 
			&\lesssim   \lt \|      |D_x|^\frac13  d_{\mathrm{in}}\rt \|_{Y_{m,\epsilon}}^2    + \frac1{A^\frac12}\|      \omega \|_{X_{a,m,\epsilon}}  \lt \|      |D_x|^\frac13 d\rt \|_{X_{a,m,\epsilon}}^2\\
			&\quad + \frac1A \int_{0}^{t} \left|\Re\left( |D_x|^\frac13  \left(  |\nabla d|^2 d \right) \left\lvert\, \mathcal{M} e^{2 a A^{-\frac{1}{3}}\left|D_x\right|^{\frac{2}{3} }t }\langle D_x\rangle^{2 m}\langle\frac{1}{D_x}\rangle^{2 \epsilon} |D_x|^\frac13  d\right.\right)\right| dt\\
			&\quad + \frac1A \int_{0}^{t} \left|\Re\left( |D_x|^\frac13  \left( |\nabla d|^2 e_1\right) \left\lvert\, \mathcal{M} e^{2 a A^{-\frac{1}{3}}\left|D_x\right|^{\frac{2}{3} }t }\langle D_x\rangle^{2 m}\langle\frac{1}{D_x}\rangle^{2 \epsilon} |D_x|^\frac13  d\right.\right)\right| dt.
		\end{align*}
		Consequently, recalling Lemmas \ref{lem:est of dx13 nabla d2 d} and \ref{lem:est of dx13 nabla d2},   the proof is complete.
	\end{proof}
	The remain argument is split into two subsections. In the first one, we establish a general lemma, which provides the basic tools for estimating nonlinear terms. In the second one, we use these results to prove the nonlinear estimates showed in Lemmas \ref{lem:est of dx13 nabla d2 d} and \ref{lem:est of dx13 nabla d2}.
	
	\subsection{Estimates for the elastic nonlinearity}
    We first show a general estimate that will be repeatedly used in the nonlinear analysis. Recalling \eqref{eq:def of fk}, we have:
	\begin{Lem}\label{lem:Preliminary estimates}
	Let $i\in\{0,1\}$, $0<a<\frac{1}{16(1+2 \pi)}$ and $\frac{1-i}6<\epsilon< \frac12 <m$. For certain smooth functions $f = f(x,y)$ and $g = g(x,y)$, we have
	\begin{align*}
		I&:=\lt\|e^{a A^{-\frac{1}{3}}|k-l|^{\frac{2}{3} } t }\langle k-l\rangle^{m   } \langle\frac{1}{k-l}\rangle^\epsilon  |k-l|^\frac i3  \int_{\mathbb{R}}   \|  f_{k-l-\eta}(y) \|_{L_y^2} \| g_{\eta}(y)  \|_{L_y^\infty} d\eta \rt\|_{L_{k-l}^2} \\
		&\lesssim  \lt \|  e^{a A^{-\frac{1}{3}}|D_x|^{\frac{2}{3} } t }    |D_x|^\frac13 f   \rt \|_{Y_{m,\epsilon}}\lt  \|  e^{a A^{-\frac{1}{3}}|D_x|^{\frac{2}{3} } t }   |D_x|^\frac{i+1}3 g  \rt \|_{Y_{m,\epsilon}}^{\frac12}\lt \|  e^{a A^{-\frac{1}{3}}|D_x|^{\frac{2}{3} } t }     \py |D_x|^\frac13 g   \rt \|_{Y_{m,\epsilon}}^\frac12. \alabel{eq:est of |nabla d|^2}
	\end{align*}
\end{Lem}

  This result will be frequently used in the next subsection and next section to control the nonlinear terms appearing in the energy estimate. And to complete this subsection, we present the proof of Lemma \ref{lem:Preliminary estimates}.

\begin{proof}[\bf Proof of Lemma \ref{lem:Preliminary estimates}] According to Lemma \ref{lem:freq_estimates}, we estimate the term $I$ in frequency regions.
	\begin{align*}
		I&\leq \lt\|e^{a A^{-\frac{1}{3}}|k-l|^{\frac{2}{3} } t }\langle k-l\rangle^{m} \langle\frac{1}{k-l}\rangle^\epsilon  |k-l|^\frac i3      \|  f_{k-l-\eta}(y)\|_{L_y^2} \| g_{\eta}(y)  \|_{L_y^\infty}   \rt\|_{L_{k-l}^2 L_\eta^1(\mathcal{R}_{\text{res}}(k-l, \eta))} \\
		&\quad +  \lt\|e^{a A^{-\frac{1}{3}}|k-l|^{\frac{2}{3} } t }\langle k-l\rangle^{m   } \langle\frac{1}{k-l}\rangle^\epsilon   |k-l|^\frac i3    \|  f_{k-l-\eta}(y)\|_{L_y^2} \| g_{\eta}(y)  \|_{L_y^\infty}   \rt\|_{L_{k-l}^2 L_\eta^1(  \mathcal{R}_{\text{HL}}(k-l, \eta) )}\\
		&\quad + \lt\|e^{a A^{-\frac{1}{3}}|k-l|^{\frac{2}{3} } t }\langle k-l\rangle^{m   } \langle\frac{1}{k-l}\rangle^\epsilon   |k-l|^\frac i3   \|  f_{k-l-\eta}(y)\|_{L_y^2} \| g_{\eta}(y)  \|_{L_y^\infty}  \rt\|_{L_{k-l}^2 L_\eta^1(  \mathcal{R}_{\text{LH}}(k-l, \eta) )}\\
		&=: I_{1} + I_{2} +I_{3}.  
	\end{align*}

    \underline{\bf Estimate of $I_1$.} In the domain of  $\mathcal{R}_{\text{res}}(k-l, \eta)$, using \eqref{eq:trick1} we have
    $$ |\eta| \lesssim |k-l-\eta| \sim |k-l|.$$
    Then by $|k_1+k_2|^\frac23\leq |k_1|^\frac23+|k_2|^\frac23$ and Young inequality for convolutions we estimate $I_1$ as follows:
	\begin{align*}
		I_{1} &\lesssim \lt\|e^{a A^{-\frac{1}{3}}|k-l|^{\frac{2}{3} } t }\langle k-l-\eta\rangle^{m }\langle\frac{1}{k-l-\eta}\rangle^\epsilon  |\eta|^{-\frac{1-i}3} |k-l-\eta|^{\frac13}    \|  f_{k-l-\eta}(y)\|_{L_y^2} \| g_{\eta}(y)  \|_{L_y^\infty}  \rt\|_{L_{k-l}^2 L_\eta^1}\\
		&\lesssim \lt\|  e^{a A^{-\frac{1}{3}}|k-l-\eta|^{\frac{2}{3} } t } \langle k-l-\eta\rangle^{m  }\langle\frac{1}{k-l-\eta}\rangle^\epsilon   |k-l-\eta|^{\frac13}  \|  f_{k-l-\eta}(y)\|_{L_y^2}    \rt\|_{L_{k-l-\eta}^2 }    \\
		&\quad\times  \lt \|  e^{a A^{-\frac{1}{3}}|\eta|^{\frac{2}{3} } t }   |\eta|^{-\frac{1-i}3}  \| g_{\eta}(y)  \|_{L_y^\infty}     \rt\|_{L_{\eta}^1 },
    \end{align*}
    which can be controlled due to  H$\text{\"{o}}$lder inequality and \eqref{eq:GN1} by 
    \begin{align*}
		&C\lt\|  e^{a A^{-\frac{1}{3}}|D_x|^{\frac{2}{3} } t } \langle D_x\rangle^{m }\langle\frac{1}{D_x}\rangle^\epsilon   |D_x|^\frac13   f   \rt\|_{L^2 }   \lt\||\eta|^{-\frac{4-i}6}\langle \eta\rangle^{-m}\langle\frac{1}{\eta}\rangle^{-\epsilon}\rt\|_{L_{\eta}^2 }\\
		&\quad\times \lt\|  e^{a A^{-\frac{1}{3}}|\eta|^{\frac{2}{3} } t }    \langle \eta\rangle^m\langle\frac{1}{\eta}\rangle^\epsilon \lt \||\eta|^{\frac{i+1}3} g_\eta (y)\rt \|_{L_y^2}^\frac12 \lt \| |\eta|^{ \frac13} \py    g_{\eta}(y)\rt \|_{L_y^2}^\frac12   \rt\|_{L_{\eta}^2 }\\
		&\lesssim  \lt\|  e^{a A^{-\frac{1}{3}}|D_x|^{\frac{2}{3} } t }   |D_x|^\frac13 f   \rt\|_{Y_{m,\epsilon}} \lt \|  e^{a A^{-\frac{1}{3}}|D_x|^{\frac{2}{3} } t }   |D_x|^{\frac{i+1}3} g  \rt  \|_{Y_{m,\epsilon}}^{\frac12} \lt \|  e^{a A^{-\frac{1}{3}}|D_x|^{\frac{2}{3} } t } \nabla  |D_x|^\frac13 g   \rt \|_{Y_{m,\epsilon}}^\frac12,
	\end{align*}
    where we used
$$\lt \||\eta|^{-\frac{4-i}6}\langle \eta\rangle^{-m}\langle\frac{1}{\eta}\rangle^{-\epsilon}\rt \|_{L_{\eta}^2 }\lesssim 1$$
for $\epsilon>\frac16$ when $i=0$ and $\epsilon>0$ when $i=1.$

    \underline{\bf Estimate of $I_2$.}
	In $\mathcal{R}_{\text{HL}}(k-l, \eta)$, by \eqref{eq:trick2} we get
    $$|k-l-\eta| \lesssim |k-l| \sim |\eta|$$
    and using Young inequality for convolutions again, we have
	\begin{align*}
		I_{2} &\lesssim   \lt\|e^{a A^{-\frac{1}{3}}|k-l|^{\frac{2}{3} } t }  \langle \eta\rangle^{m  }\langle\frac{1}{\eta}\rangle^\epsilon |k-l-\eta|^{\frac{i-2}6}  |\eta|^{  \frac{i+2}6}   \| f_{k-l-\eta}(y) \|_{L_y^2} \|  g_{\eta}(y) \|_{L_y^\infty}  \rt\|_{L_{k-l}^2 L_\eta^1} \\
		&\lesssim \lt\|  e^{a A^{-\frac{1}{3}}|k-l-\eta|^{\frac{2}{3} } t }  |k-l-\eta|^{\frac{i-2}6} \| f_{k-l-\eta}(y) \|_{L_y^2}     \rt\|_{L_{k-l-\eta}^1 }\lt\|  e^{a A^{-\frac{1}{3}}|\eta|^{\frac{2}{3} } t } \langle  \eta\rangle^{m  }\langle\frac{1}{ \eta}\rangle^\epsilon  |\eta|^{  \frac{i+2}6} \| g_{\eta}(y)  \|_{L_y^\infty}   \rt\|_{L_{\eta}^2 }
\end{align*}
which, along with \eqref{eq:GN1}, implies
        \begin{align*}
		I_2&\lesssim \lt \|  e^{a A^{-\frac{1}{3}}|D_x|^{\frac{2}{3} } t } \langle D_x\rangle^{m}\langle\frac{1}{D_x}\rangle^\epsilon   |D_x|^\frac{i+1}3 g   \rt  \|_{L^2 }^\frac12 \lt \|  e^{a A^{-\frac{1}{3}}|D_x|^{\frac{2}{3} } t } \langle D_x\rangle^{m}\langle\frac{1}{D_x}\rangle^\epsilon   \py  |D_x|^\frac13  g    \rt \|_{L^2 }^\frac12 \\
		&\quad\times \lt\|  e^{a A^{-\frac{1}{3}}|k-l-\eta|^{\frac{2}{3} } t } \langle k-l-\eta\rangle^m\langle\frac{1}{k-l-\eta}\rangle^\epsilon  |k-l-\eta|^{\frac13} \|  f_{k-l-\eta}(y)\|_{L_y^2}     \rt\|_{L_{k-l-\eta}^2 } \\
		&\quad\times \lt\||k-l-\eta|^{-\frac{4-i}{6}} \langle k-l-\eta\rangle^{-m}\langle\frac{1}{k-l-\eta}\rangle^{-\epsilon}  \rt\|_{L_{k-l-\eta}^2 } \\
		&\lesssim \lt \|  e^{a A^{-\frac{1}{3}}|D_x|^{\frac{2}{3} } t }  |D_x|^\frac13 f   \rt \|_{Y_{m,\epsilon}}\lt \|  e^{a A^{-\frac{1}{3}}|D_x|^{\frac{2}{3} } t }  |D_x|^\frac{i+1}3 g  \rt   \|_{Y_{m,\epsilon}}^{\frac12} \lt \|  e^{a A^{-\frac{1}{3}}|D_x|^{\frac{2}{3} } t }  \nabla  |D_x|^\frac13 g \rt   \|_{Y_{m,\epsilon}}^\frac12.  
	\end{align*}
    
	\underline{\bf Estimate of $I_3$.} For $I_3$, we infer from \eqref{eq:trick3} and \eqref{eq:trick3.5} that
    \begin{gather*}
    |k-l| \lesssim |k-l-\eta|\sim |\eta|,\quad \langle k-l \rangle^2\lesssim  \langle \eta \rangle \langle k-l-\eta \rangle,\\
				1 \leq \langle \frac{1}{k-l-\eta} \rangle^{s} |k-l-\eta|^{s} \lesssim \langle\frac{1}{\eta}\rangle^{s_3}\langle\frac{1}{k-l-\eta}\rangle^{s_3}\left(1+\langle\frac{1}{k-l}\rangle^{s-2 s_3} \right)|k-l-\eta|^{s_1 }|\eta|^{s_2}.
    \end{gather*}
    Then, choosing $s=\frac{4-i}{6}$, $s_1=\frac{1-i}3$, $s_2=\frac{i+2}{6}$ and $s_3=\epsilon$,  we have
	\begin{align*}
		I_{3} &\lesssim \lt\|  e^{a A^{-\frac{1}{3}}|k-l|^{\frac{2}{3} } t }\langle k-l -\eta \rangle^{m  }\langle \eta\rangle^m\langle k-l\rangle^{-m}\rt.\\
		&\lt.\quad\times  \langle\frac{1}{\eta}\rangle^\epsilon\langle\frac{1}{k-l-\eta}\rangle^\epsilon\left(\langle\frac{1}{k-l}\rangle^{\epsilon} +\langle\frac{1}{k-l}\rangle^{\frac{4-i}{6}- \epsilon} \right)
		|k-l-\eta|^\frac13  |\eta|^{ \frac{i+2}6} \| f_{k-l-\eta}(y) \|_{L_y^2}  \|  g_{\eta} (y) \|_{L_y^\infty}   \rt\|_{L_{k-l}^2 L_\eta^1 }\\
		&\lesssim \lt\|  e^{a A^{-\frac{1}{3}}|k-l-\eta|^{\frac{2}{3} } t }  \langle k-l -\eta \rangle^{m} \langle\frac{1}{k-l-\eta}\rangle^\epsilon |k-l-\eta|^{\frac13} \| f_{k-l-\eta}(y) \|_{L_y^2} \rt\|_{L_{k-l-\eta}^2 }    \\
		&\quad\times \lt\|  e^{a A^{-\frac{1}{3}}|\eta|^{\frac{2}{3} } t } \langle  \eta\rangle^m\langle\frac{1}{ \eta}\rangle^\epsilon  |\eta|^\frac{i+2}6 \| g_{\eta}(y) \|_{L_y^\infty}   \rt\|_{L_{\eta}^2 } \lt\| \left(\langle\frac{1}{k-l}\rangle^{\epsilon} +\langle\frac{1}{k-l}\rangle^{\frac{4-i}{6}- \epsilon} \right) \langle k-l\rangle^{-m}\rt\|_{L_{k-l}^2}\\
		&\lesssim  \lt \|  e^{a A^{-\frac{1}{3}}|D_x|^{\frac{2}{3} } t }  |D_x|^\frac13 f  \rt \|_{Y_{m,\epsilon}} \lt \|  e^{a A^{-\frac{1}{3}}|D_x|^{\frac{2}{3} } t }   |D_x|^\frac{i+1}3 g  \rt \|_{Y_{m,\epsilon}}^{\frac12} \lt \|  e^{a A^{-\frac{1}{3}}|D_x|^{\frac{2}{3} } t }\nabla  |D_x|^\frac13 g  \rt  \|_{Y_{m,\epsilon}}^\frac12.
	\end{align*}
	Hence, the estimate \eqref{eq:est of |nabla d|^2} holds by summing up $I_1, I_2, I_3$.
\end{proof}

	\subsection{Nonlinear energy estimates}
	
	Using the result established in the previous subsection, we prove the two nonlinear estimates, i.e., lemmas \ref{lem:est of dx13 nabla d2 d} and \ref{lem:est of dx13 nabla d2}.
    
	\begin{proof}[Proof of Lemma \ref{lem:est of dx13 nabla d2 d}]
		Thanks to the Fourier transform, we obtain
		\begin{align*} 
			& \quad   \frac1A \left|\Re\left( |D_x|^{\frac13}  \left(  |\nabla d|^2 d \right) \left\lvert\, \mathcal{M} e^{2 a A^{-\frac{1}{3}}\left|D_x\right|^{\frac{2}{3} }t }\langle D_x\rangle^{2 m  }\langle\frac{1}{D_x}\rangle^{2 \epsilon} |D_x|^{\frac13} d\right.\right)\right|\\
			&\lesssim    \frac1A \int_{\mathbb{R}^3} e^{2 a A^{-\frac{1}{3}}|k|^{\frac{2}{3} } t }\langle k\rangle^{2 m   }\langle\frac{1}{k}\rangle^{2 \epsilon} |k|^\frac23 \\
            &\qqquad \qqquad \times \left|\int_{\mathbb{R}} \mathcal{M}\left(k, D_y\right)   d_k(y) \cdot   d_l(y)  \nabla_{k-l-\eta} d_{k-l-\eta}(y) \cdot \nabla_{\eta}  d_{\eta}(y)   d y\right| d k d ld\eta , \alabel{eq:Fourier}
		\end{align*}
        where the last term can be estimated by three different domains as follows.

		$\bullet$~\textbf{In $\mathcal{R}_{\text{res}}(k,l)$.}  By \eqref{eq:trick1} and Young inequality, we deduce that
		\begin{align*}  
			&\quad \int_{\mathcal{R}_{\text{res}(k,l)}\times\mathbb{R}} e^{2 a A^{-\frac{1}{3}}|k|^{\frac{2}{3} } t }\langle k\rangle^{2 m   }\langle\frac{1}{k}\rangle^{2 \epsilon} |k|^\frac23 \\
			&\qqquad \qqquad \qquad \times \left|\int_{\mathbb{R}} \mathcal{M}\left(k, D_y\right)   d_k(y) \cdot   d_l(y)  \nabla_{k-l-\eta} d_{k-l-\eta}(y) \cdot \nabla_{\eta}  d_{\eta}(y)   d y\right| d k d ld\eta \\
			&\lesssim \left\|e^{a A^{-\frac{1}{3}}|k|^{\frac{2}{3} }t}\langle k\rangle^{m  }\langle\frac{1}{k}\rangle^\epsilon |k|^\frac23  \|d_k\|_{L_y^2}\right\|_{L_k^2}     \left\|e^{a A^{-\frac{1}{3}}|l|^{\frac{2}{3} } t } \| d_l \|_{L_y^{\infty}}\right\|_{L_l^1} \\
			&\quad \times\left\|e^{a A^{-\frac{1}{3}}|k-l|^{\frac{2}{3} } t }\langle k-l\rangle^{m   } \langle\frac{1}{k-l}\rangle^\epsilon    \int_{\mathbb{R}}  \lt \| \nabla_{k-l-\eta} d_{k-l-\eta}\rt \|_{L_y^2} \lt \|\nabla_\eta d_{\eta}\rt \|_{L_y^\infty} d\eta \right\|_{L_{k-l}^2},
\end{align*}
        and by  Corollary \ref{cor:L infty} ($\alpha=\beta=\frac13$) and Lemma \ref{lem:Preliminary estimates} ($i=0$, $f=g=\nabla d$)  this can be bounded by 
        \begin{align*}  
			& C\lt \|e^{a A^{-\frac{1}{3}}\left|D_x\right|^{\frac{2}{3} } t }    |D_x|^\frac23 d \rt \|_{Y_{m,\epsilon}} \lt  \|e^{a A^{-\frac{1}{3}}\left|D_x\right|^{\frac{2}{3} } t }    |D_x|^\frac13 d\rt \|_{Y_{m,\epsilon}}^{\frac12}  \lt  \|  e^{a A^{-\frac{1}{3}}|D_x|^{\frac{2}{3} } t }     \py |D_x|^\frac13 d    \rt \|_{Y_{m,\epsilon}}^{\frac12} \\
			&\quad \times \lt   \|  e^{a A^{-\frac{1}{3}}|D_x|^{\frac{2}{3} } t }     \nabla |D_x|^\frac13 d    \|_{Y_{m,\epsilon}}^{\frac32} \|  e^{a A^{-\frac{1}{3}}|D_x|^{\frac{2}{3} } t }     \nabla \py |D_x|^\frac13 d    \rt \|_{Y_{m,\epsilon} }^\frac12 . \alabel{case1'}
		\end{align*}

        $\bullet$~\textbf{In $\mathcal{R}_{\text{HL}}(k,l) $.} Using Lemma \ref{lem:Preliminary estimates} ($i=0$, $f=g=\nabla d$)  again, along with \eqref{eq:GN1} and \eqref{eq:trick2}, we have
		\begin{align*}   
			&\quad\int_{\mathcal{R}_{\text{HL}(k,l)}\times\mathbb{R}}   e^{2 a A^{-\frac{1}{3}}|k|^{\frac{2}{3} } t }\langle k\rangle^{2 m   }\langle\frac{1}{k}\rangle^{2 \epsilon} |k|^\frac23 \\
			&\qqquad \qqquad \qquad \times \left|\int_{\mathbb{R}} \mathcal{M}\left(k, D_y\right)   d_k(y) \cdot   d_l(y)  \nabla_{k-l-\eta} d_{k-l-\eta}(y) \cdot \nabla_{\eta}  d_{\eta}(y)   d y\right| d k d ld\eta  \\
			&\lesssim  \lt\|e^{a A^{-\frac{1}{3}}|k|^{\frac{2}{3} }t}\langle k\rangle^{m   }\langle\frac{1}{k}\rangle^\epsilon |k|^\frac23 \| d_k\|_{L_y^2}
            \rt\|_{L_k^2}     \lt\|e^{a A^{-\frac{1}{3}}|l|^{\frac{2}{3} } t } \langle l\rangle^{m   }\langle\frac{1}{l}\rangle^\epsilon |l|^\frac13\| d_l \|_{L_y^{\infty}}\rt\|_{L_l^2} \\
			&\quad \times\left\|e^{a A^{-\frac{1}{3}}|k-l|^{\frac{2}{3} } t }\langle k-l\rangle^{m   } \langle\frac{1}{k-l}\rangle^\epsilon    \int_{\mathbb{R}}  \|\nabla_{k-l-\eta}  d_{k-l-\eta}\|_{L_y^2} \|\nabla_\eta d_{\eta}\|_{L_y^\infty} d\eta \right\|_{L_{k-l}^2} \\
			&\quad\times \lt \||k-l|^{-\frac13}\langle k-l\rangle^{-m}\langle\frac{1}{k-l}\rangle^{-\epsilon}\rt \|_{L_{k-l}^2}\\
			&\lesssim \lt  \|e^{a A^{-\frac{1}{3}}\left|D_x\right|^{\frac{2}{3} } t }    |D_x|^\frac23 d \rt \|_{Y_{m,\epsilon} }  \lt \|e^{a A^{-\frac{1}{3}}\left|D_x\right|^{\frac{2}{3} } t }    |D_x|^\frac13 d \rt \|_{Y_{m,\epsilon} }^{\frac12}    
			\lt \|  e^{a A^{-\frac{1}{3}}|D_x|^{\frac{2}{3} } t }     \py |D_x|^\frac13 d    \rt \|_{Y_{m,\epsilon}}^{\frac12} \\
			&\quad \times   \lt \|  e^{a A^{-\frac{1}{3}}|D_x|^{\frac{2}{3} } t }     \nabla |D_x|^\frac13 d   \rt \|_{Y_{m,\epsilon} }^{\frac32}\lt \|  e^{a A^{-\frac{1}{3}}|D_x|^{\frac{2}{3} } t }     \py \nabla |D_x|^\frac13 d   \rt \|_{Y_{m,\epsilon} }^\frac12. \alabel{case2'}
		\end{align*}

		$\bullet$~\textbf{In $\mathcal{R}_{\text{LH}}(k,l) $.} As in the domain ${\mathcal{R}_{\text{HL}(k,l)}}$, by \eqref{eq:trick3} and  \eqref{eq:trick3.5} we have 
		\begin{align*}  
			&\quad\int_{\mathcal{R}_{\text{LH}(k,l)}\times\mathbb{R}} e^{2 a A^{-\frac{1}{3}}|k|^{\frac{2}{3} } t }\langle k\rangle^{2 m   }\langle\frac{1}{k}\rangle^{2 \epsilon} |k|^\frac23 \\
			&\qqquad \qqquad \qquad \times \left|\int_{\mathbb{R}} \mathcal{M}\left(k, D_y\right)   d_k(y) \cdot   d_l(y) \nabla_{k-l-\eta} d_{k-l-\eta}(y) \cdot \nabla_{\eta}  d_{\eta}(y)   d y\right| d k d ld\eta \\
			& \lesssim \lt\|e^{a A^{-\frac{1}{3}}|k|^{\frac{2}{3} } t }\langle k\rangle^{m }\langle\frac{1}{k}\rangle^\epsilon  |k|^\frac23 \lt\|d_k\rt\|_{L_y^2}\rt\|_{L_k^2} \lt \|e^{a A^{-\frac{1}{3}}|l|^{\frac{2}{3} } t }\langle l\rangle^{m   }\langle\frac{1}{l}\rangle^\epsilon  |l|^{\frac13}  \lt\|d_l\rt\|_{L_y^\infty}\rt\|_{L_l^2} \\
			& \quad \times  \lt \|e^{a A^{-\frac{1}{3}}|k-l|^{\frac{2}{3} } t }\langle k-l\rangle^{m   } \langle\frac{1}{k-l}\rangle^\epsilon    \int_{\mathbb{R}} \lt \| \nabla_{k-l-\eta}   d_{k-l-\eta}\rt\|_{L_y^2} \lt\|\nabla_\eta d_{\eta}\rt\|_{L_y^\infty} d\eta \rt\|_{L_{k-l}^2} \\
			&\quad \times \lt\|  \left(\langle\frac{1}{k}\rangle^\epsilon  +\langle\frac{1}{k}\rangle^{\frac13 -\epsilon}  \right)\langle k\rangle^{-m}\rt\|_{L_k^2},
\end{align*}
            which is controlled by (\ref{case1'}) by using \eqref{eq:GN1} and Lemma \ref{lem:Preliminary estimates} again.
        
		Hence, one collects \eqref{case1'}--\eqref{case2'}, and deduces from \eqref{eq:X norm} and \eqref{eq:Fourier} that
		\begin{align*}
			&\quad   \frac1A \int_{0}^{t} \left|\Re\left(    |\nabla d|^2 d  \left\lvert\, \mathcal{M} e^{2 a A^{-\frac{1}{3}}\left|D_x\right|^{\frac{2}{3} }t }\langle D_x\rangle^{2 m  }\langle\frac{1}{D_x}\rangle^{2 \epsilon} |D_x|^{\frac23} d\right.\right)\right| dt \\
            &\lesssim \frac{1}{A^{\frac13}}\left(A^{\frac16}\lt  \|e^{a A^{-\frac{1}{3}}\left|D_x\right|^{\frac{2}{3} } t }    |D_x|^\frac23 d \rt \|_{L_t^2Y_{m,\epsilon} } \right) \lt \|e^{a A^{-\frac{1}{3}}\left|D_x\right|^{\frac{2}{3} } t }    |D_x|^\frac13 d \rt \|_{L_t^\infty Y_{m,\epsilon} }^{\frac12}    
			\left(A^{\frac14}\lt \|  e^{a A^{-\frac{1}{3}}|D_x|^{\frac{2}{3} } t }     \py |D_x|^\frac13 d    \rt \|_{L_t^2 Y_{m,\epsilon}}^{\frac12}\right) \\
			&\quad \times   \lt \|  e^{a A^{-\frac{1}{3}}|D_x|^{\frac{2}{3} } t }     \nabla |D_x|^\frac13 d   \rt \|_{L_t^\infty Y_{m,\epsilon} }^{\frac32}\left(A^{\frac14}\lt \|  e^{a A^{-\frac{1}{3}}|D_x|^{\frac{2}{3} } t }     \py \nabla |D_x|^\frac13 d   \rt \|_{L_t^2Y_{m,\epsilon} }^\frac12\right)\\
            &\lesssim   \frac{1}{A^{\frac{1}{3}}} \lt\|     |D_x|^\frac13 d \rt\|_{X_{a,m,\epsilon}}^2  \lt\|   \nabla |D_x|^\frac13 d \rt\|_{X_{a,m,\epsilon}}^2, 
		\end{align*}
        which completes the proof.
	\end{proof}

	\begin{proof}[Proof of Lemma \ref{lem:est of dx13 nabla d2}]
		Thanks to the Fourier transform, we have
		\begin{align*} 
			& \quad   \frac1A   \left|\Re\left( |D_x|^\frac13  \left( |\nabla d|^2 e_1\right) \left\lvert\, \mathcal{M} e^{2 a A^{-\frac{1}{3}}\left|D_x\right|^{\frac{2}{3} }t }\langle D_x\rangle^{2 m}\langle\frac{1}{D_x}\rangle^{2 \epsilon} |D_x|^\frac13  d\right.\right)\right| \\
			&\lesssim   \frac1A \int_{\mathbb{R}^2} e^{2 a A^{-\frac{1}{3}}|k|^{\frac{2}{3} } t }\langle k\rangle^{2 m   }\langle\frac{1}{k}\rangle^{2 \epsilon} |k|^\frac23 \left|\int_{\mathbb{R}} \mathcal{M}\left(k, D_y\right)   d_k^1(y)  \nabla_l d_l(y)    \cdot \nabla_{k-l} d_{k-l}(y)  d y\right| d k d l,  \alabel{eq:Fourier2}
		\end{align*}
        which will still be estimated in three different frequency domains.

		$\bullet$~\textbf{In $\mathcal{R}_{\text{res}}(k,l)$.}
  From \eqref{eq:trick1} and Corollary \ref{cor:L infty} ($\alpha=\frac23, \beta=\frac13$),
		we deduce that
		\begin{align*}  
			&\quad \int_{\mathcal{R}_{\text{res}}(k,l)} e^{2 a A^{-\frac{1}{3}}|k|^{\frac{2}{3} } t }\langle k\rangle^{2 m   }\langle\frac{1}{k}\rangle^{2 \epsilon} |k|^\frac23 \left|\int_{\mathbb{R}} \mathcal{M}\left(k, D_y\right)   d_k^1(y)  \nabla_l d_l(y)    \cdot \nabla_{k-l} d_{k-l}(y)  d y\right| d k d l \\
			&\lesssim \lt\|e^{a A^{-\frac{1}{3}}|k|^{\frac{2}{3} }t}\langle k\rangle^{m  }\langle\frac{1}{k}\rangle^\epsilon |k|^\frac13  \lt\|d_k^1\rt\|_{L_y^2}\rt\|_{L_k^2}    \lt \|e^{a A^{-\frac{1}{3}}|l|^{\frac{2}{3} } t }  \lt \| \nabla_l d_l \rt\|_{L_y^{\infty}}\rt\|_{L_l^1} \\
			&\quad \times\lt\|e^{a A^{-\frac{1}{3}}|k-l|^{\frac{2}{3} } t }\langle k-l\rangle^{m   } \langle\frac{1}{k-l}\rangle^\epsilon   |k-l|^\frac13   \lt\| \nabla_{k-l} d_{k-l}\rt\|_{L_y^2}   \rt\|_{L_{k-l}^2} \\
			&\lesssim \lt\|e^{a A^{-\frac{1}{3}}\left|D_x\right|^{\frac{2}{3} } t }    |D_x|^\frac13 d^1 \rt\|_{Y_{m,\epsilon}} \lt\|e^{a A^{-\frac{1}{3}}\left|D_x\right|^{\frac{2}{3} }t}      |D_x|^\frac13  \nabla |D_x|^\frac13 d\rt\|_{Y_{m,\epsilon}}^\frac12       \\
			&\quad \times \lt\|e^{a A^{-\frac{1}{3}}\left|D_x\right|^{\frac{2}{3} }t}    \py\nabla |D_x|^\frac13 d\rt\|_{Y_{m,\epsilon}}^\frac12 \lt\|e^{a A^{-\frac{1}{3}}\left|D_x\right|^{\frac{2}{3} } t }   \nabla |D_x|^\frac13  d\rt \|_{Y_{m,\epsilon}} . \alabel{eq:case11'}
		\end{align*}
		
        $\bullet$~\textbf{In $\mathcal{R}_{\text{HL}}(k,l) $.}
 Based on \eqref{eq:GN1}, it follows from \eqref{eq:trick2} that
		\begin{align*}   
			&\quad\int_{\mathcal{R}_{\text{HL}}(k,l)}   e^{2 a A^{-\frac{1}{3}}|k|^{\frac{2}{3} } t }\langle k\rangle^{2 m   }\langle\frac{1}{k}\rangle^{2 \epsilon} |k|^\frac23 \left|\int_{\mathbb{R}} \mathcal{M}\left(k, D_y\right)   d_k^1(y)  \nabla_l d_l(y)    \cdot \nabla_{k-l} d_{k-l}(y)  d y\right| d k d l \\
			&\lesssim  \lt\|e^{a A^{-\frac{1}{3}}|k|^{\frac{2}{3} }t}\langle k\rangle^{m   }\langle\frac{1}{k}\rangle^\epsilon |k|^\frac13\lt \| d_k^1\rt\|_{L_y^2}\rt\|_{L_k^2}    \lt \|e^{a A^{-\frac{1}{3}}|l|^{\frac{2}{3} } t } \langle l\rangle^{m   }\langle\frac{1}{l}\rangle^\epsilon |l|^\frac12\lt \| \nabla_l d_l\rt \|_{L_y^{\infty}}\rt\|_{L_l^2} \\
			&\quad \times\lt\|e^{a A^{-\frac{1}{3}}|k-l|^{\frac{2}{3} } t }\langle k-l\rangle^{m   } \langle\frac{1}{k-l}\rangle^\epsilon    |k-l|^\frac13 \lt\|\nabla_{k-l} d_{k-l }\rt\|_{L_y^2}\rt \|_{L_{k-l}^2} 
			\lt\||k-l|^{-\frac12}\langle k-l\rangle^{-m}\langle\frac{1}{k-l}\rangle^{-\epsilon}\rt\|_{L_{k-l}^2}\\
			&\lesssim \lt \|e^{a A^{-\frac{1}{3}}\left|D_x\right|^{\frac{2}{3} } t }    |D_x|^\frac13 d^1 \rt\|_{Y_{m,\epsilon}}\lt \|e^{a A^{-\frac{1}{3}}\left|D_x\right|^{\frac{2}{3} }t}    |D_x|^\frac13  \nabla |D_x|^\frac13 d\rt\|_{Y_{m,\epsilon}}^\frac12       \\
			&\quad \times \lt\|e^{a A^{-\frac{1}{3}}\left|D_x\right|^{\frac{2}{3} }t}    \py \nabla |D_x|^\frac13 d\rt\|_{Y_{m,\epsilon}}^\frac12\lt \|e^{a A^{-\frac{1}{3}}\left|D_x\right|^{\frac{2}{3} } t }   \nabla  |D_x|^\frac13  d \rt\|_{Y_{m,\epsilon}}.  \alabel{eq:case12'}
		\end{align*}

$\bullet$~\textbf{In $\mathcal{R}_{\text{LH}}(k,l) $.}
        Using \eqref{eq:trick3} and \eqref{eq:trick3.5} ($s=s_2=\frac12$, $s_1=0$, $s_3=\epsilon$), we have
		\begin{align*}  
			&\quad\int_{\mathcal{R}_{\text{LH}}(k,l)} e^{2 a A^{-\frac{1}{3}}|k|^{\frac{2}{3} } t }\langle k\rangle^{2 m   }\langle\frac{1}{k}\rangle^{2 \epsilon} |k|^\frac23 \left|\int_{\mathbb{R}} \mathcal{M}\left(k, D_y\right)   d_k^1(y)  \nabla_l d_l(y)    \cdot \nabla_{k-l} d_{k-l}(y)  d y\right| d k d l \\
			& \lesssim \lt\|e^{a A^{-\frac{1}{3}}|k|^{\frac{2}{3} } t }\langle k\rangle^{m }\langle\frac{1}{k}\rangle^\epsilon  |k|^\frac13 \lt\|d_k^1\rt\|_{L_y^2}\rt\|_{L_k^2}  \lt\|e^{a A^{-\frac{1}{3}}|l|^{\frac{2}{3} } t }\langle l\rangle^{m   }\langle\frac{1}{l}\rangle^\epsilon  |l|^{\frac12}  \lt\|\nabla_l d_l\rt\|_{L_y^\infty}\rt\|_{L_l^2} \\
			& \quad \times \lt  \|e^{a A^{-\frac{1}{3}}|k-l|^{\frac{2}{3} } t }\langle k-l\rangle^{m   } \langle\frac{1}{k-l}\rangle^\epsilon    |k-l|^\frac13 \lt\| \nabla_{k-l} d_{k-l }\rt\|_{L_y^2}\rt  \|_{L_{k-l}^2} 
			\lt\|  \left(\langle\frac{1}{k}\rangle^\epsilon  +\langle\frac{1}{k}\rangle^{\frac12 -\epsilon}  \right)\langle k\rangle^{-m}\rt\|_{L_k^2},
		\end{align*}
		which can be bounded by the last term in \eqref{eq:case11'} by using \eqref{eq:GN1} again.
		
		Note that the estimates \eqref{eq:Fourier2}--\eqref{eq:case12'} shows that
		\begin{align*}
			&\quad  \left|\Re\left(    |\nabla d|^2   \left\lvert\, \mathcal{M} e^{2 a A^{-\frac{1}{3}}\left|D_x\right|^{\frac{2}{3} }t }\langle D_x\rangle^{2 m  }\langle\frac{1}{D_x}\rangle^{2 \epsilon} |D_x|^{\frac23} d^1\right.\right)\right|\\
			&\lesssim \lt \|e^{a A^{-\frac{1}{3}}\left|D_x\right|^{\frac{2}{3} } t }    |D_x|^\frac13 d^1 \rt\|_{Y_{m,\epsilon}} \lt\|e^{a A^{-\frac{1}{3}}\left|D_x\right|^{\frac{2}{3} }t}    |D_x|^\frac13  \nabla |D_x|^\frac13 d\rt\|_{Y_{m,\epsilon}}^\frac12       \\
			&\quad \times \lt\|e^{a A^{-\frac{1}{3}}\left|D_x\right|^{\frac{2}{3} }t}    \py \nabla |D_x|^\frac13 d\rt\|_{Y_{m,\epsilon}}^\frac12 \lt\|e^{a A^{-\frac{1}{3}}\left|D_x\right|^{\frac{2}{3} } t }   \nabla |D_x|^\frac13  d \rt\|_{Y_{m,\epsilon}},  
		\end{align*}
		and from which, combined with \eqref{eq:X norm}, we deduce that
		\begin{align*}
			&\quad \frac1A \int_{0}^{t} \left|\Re\left(    |\nabla d|^2  \left\lvert\, \mathcal{M} e^{2 a A^{-\frac{1}{3}}\left|D_x\right|^{\frac{2}{3} }t }\langle D_x\rangle^{2 m  }\langle\frac{1}{D_x}\rangle^{2 \epsilon} |D_x|^{\frac23} d^1\right.\right)\right| dt\\
            &\lesssim \frac{1}{A}\lt \|e^{a A^{-\frac{1}{3}}\left|D_x\right|^{\frac{2}{3} } t }    |D_x|^\frac13 d^1 \rt\|_{L_t^\infty Y_{m,\epsilon}} \lt\|e^{a A^{-\frac{1}{3}}\left|D_x\right|^{\frac{2}{3} }t}    |D_x|^\frac13  \nabla |D_x|^\frac13 d\rt\|_{L_t^2 Y_{m,\epsilon}}^\frac12       \\
			&\quad \times \lt\|e^{a A^{-\frac{1}{3}}\left|D_x\right|^{\frac{2}{3} }t}    \py \nabla |D_x|^\frac13 d\rt\|_{L_t^2 Y_{m,\epsilon}}^\frac12 \lt\|e^{a A^{-\frac{1}{3}}\left|D_x\right|^{\frac{2}{3} } t }   \nabla |D_x|^\frac13  d \rt\|_{L_t^2 Y_{m,\epsilon}}\\
			&\lesssim \frac{1}{A^{\frac{1}{6}}} \lt\|     |D_x|^\frac13 d^1 \rt\|_{X_{a,m,\epsilon}} \lt\|     |D_x|^\frac13 d \rt\|_{X_{a,m,\epsilon}} \lt \|    \nabla |D_x|^\frac13 d \rt\|_{X_{a,m,\epsilon}}. 
		\end{align*}
        This completes the proof.
	\end{proof}

	\section{The energy estimate for the higher-order derivative terms $\lt(\px^2, \py^2\rt) |D_x|^{1/3} d$} \label{sec.4}

	In this section, we devote to the proof of the main energy estimate for $\lt(\px^2, \py^2\rt) |D_x|^{1/3} d$, which implies the boundedness of $\nabla d$.

	\begin{Prop} \label{lem:est of pypy dx13 d}
   
		Assume that $0\leq t \leq T$. Then, for $0<a<\frac{1}{16(1+2 \pi)}$ and $\frac16<\epsilon< \frac12<m$, 
		\begin{align*} 
			&\quad\lt\|\lt(\px^2,\py^2\rt) |D_x|^\frac13 d  \rt\|_{X_{a,m,\epsilon}}^2 \\
			&\leq C\lt\|   \lt (\px^2,\py^2\rt) |D_x|^\frac13  d_{\mathrm{in}} \rt\|_{Y_{m,\epsilon}}^2  +C A \lt \|     |D_x|^\frac13 d\rt\|_{X_{a,m,\epsilon}}\lt \|    \lt(\px^2,\py^2\rt) |D_x|^\frac13 d\rt\|_{X_{a,m,\epsilon}}\\
			&\quad + \frac{C}{A^{\frac16}} \lt \|     \lt(\px^2,\py^2\rt)|D_x|^\frac13 d \rt\|_{X_{a,m,\epsilon}}^{\frac32}\lt\|     \nabla|D_x|^{\frac13} d \rt\|_{X_{a,m,\epsilon}}^{\frac52} \\
            &\quad+ \frac{C}{A^{\frac{1}{6}}}  \lt \|     \lt(\px^2,\py^2\rt) |D_x|^\frac13 d \rt\|_{X_{a,m,\epsilon}}^2\lt\| |D_x|^\frac13 d \rt\|_{X_{a,m,\epsilon}}^\frac12\lt\|    \nabla |D_x|^\frac13 d \rt\|_{X_{a,m,\epsilon}}^\frac32 \\
			&\quad+   \frac{C}{A^{\frac{1}{6}}}  \lt\|     \lt(\px^2,\py^2\rt)|D_x|^\frac13 d^1 \rt\|_{X_{a,m,\epsilon}} \lt\|\lt (\px^2,\py^2\rt) |D_x|^\frac13 d \rt\|_{X_{a,m,\epsilon}}\lt \|     \nabla|D_x|^\frac13 d \rt\|_{X_{a,m,\epsilon}}\\
			&\quad+\frac{C}{A^\frac13} \lt\|      \lt(\px^2,\py^2\rt)|D_x|^\frac13 d \rt\|_{X_{a,m,\epsilon}} \lt\|    \omega \rt\|_{X_{a,m,\epsilon}}\lt(\lt\|      \lt(\px^2,\py^2\rt)|D_x|^\frac13 d \rt\|_{X_{a,m,\epsilon}} +\lt\|     \nabla|D_x|^\frac13 d \rt\|_{X_{a,m,\epsilon}} \rt).
		\end{align*}
	\end{Prop}
	
	To deal with the energy estimates involving higher-order derivatives, we establish three lemmas that will be used to prove the above proposition.
	\begin{Lem} \label{lem:est of pypy nabla d2 d}
		Let $0<a<\frac{1}{16(1+2 \pi)}$ and $\frac16 <\epsilon< \frac12 <m$. Then,
		\begin{align*}
			& \quad   \frac1A \int_{0}^{t} \left|\Re\left( \p^2  \left(  |\nabla d|^2 d \right) \left\lvert\, \mathcal{M} e^{2 a A^{-\frac{1}{3}}\left|D_x\right|^{\frac{2}{3} }t }\langle D_x\rangle^{2 m  }\langle\frac{1}{D_x}\rangle^{2 \epsilon}  \p^2 |D_x|^\frac23 d\right.\right)\right| dt\\
			&\lesssim \frac{1}{A^{\frac16}}  \lt\|     \lt(\px^2,\py^2\rt)|D_x|^\frac13 d \rt\|_{X_{a,m,\epsilon}}^{\frac32}\lt\|     \nabla|D_x|^{\frac13} d \rt\|_{X_{a,m,\epsilon}}^{\frac52} \\
            &\quad+ \frac{1}{A^{\frac{1}{6}}}  \lt \|     \lt(\px^2,\py^2\rt) |D_x|^\frac13 d \rt\|_{X_{a,m,\epsilon}}^2\lt\| |D_x|^\frac13 d \rt\|_{X_{a,m,\epsilon}}^\frac12\lt\|    \nabla |D_x|^\frac13 d\rt \|_{X_{a,m,\epsilon}}^\frac32
		\end{align*}
        holds for $i=1,2$, where $\p$ is defined in \eqref{eq:ph-fr}.
	\end{Lem}

	\begin{Lem} \label{lem:est of pypy nabla d2 e1}
		For $0<a<\frac{1}{16(1+2 \pi)}$, $0<\epsilon< \frac12 <m$ and $i=1,2$, the following estimate holds:
		\begin{align*}
			& \quad  \frac1A \int_{0}^{t} \left|\Re\left( \p^2  \left( |\nabla d|^2 e_1\right) \left\lvert\, \mathcal{M} e^{2 a A^{-\frac{1}{3}}\left|D_x\right|^{\frac{2}{3} }t }\langle D_x\rangle^{2 m}\langle\frac{1}{D_x}\rangle^{2 \epsilon} \p^2 |D_x|^\frac23  d\right.\right)\right| dt\\
			&\lesssim  \frac{1}{A^{\frac{1}{6}}}\lt  \|    \lt(\px^2,\py^2\rt) |D_x|^\frac13 d^1 \rt\|_{X_{a,m,\epsilon}} \lt\|     \lt(\px^2,\py^2\rt) |D_x|^\frac13 d \rt\|_{X_{a,m,\epsilon}} \lt\|     \nabla|D_x|^\frac13 d \rt\|_{X_{a,m,\epsilon}}.
		\end{align*}
	\end{Lem}
	\begin{Lem} \label{lem:est of py^2 u}
		For $0<a<\frac{1}{16(1+2 \pi)}$, $\frac16<\epsilon<\frac12<m$ and $i=1,2$, we have
		\begin{align*}
			&\quad \frac1A \int_0^t \left| \Re\left( \p^2   \lt( u \cdot \nabla d\rt) \left\lvert\, \mathcal{M} e^{2 a A^{-\frac{1}{3}}\left|D_x\right|^{\frac{2}{3} } t}\langle D_x\rangle^{2 m}\langle\frac{1}{D_x}\rangle^{2 \epsilon} \p^2  |D_x|^\frac23 d\right.\right)  \right|dt \\
			&\lesssim \frac{1}{A^\frac13}\lt \|      \lt(\px^2,\py^2\rt)|D_x|^\frac13 d \rt\|_{X_{a,m,\epsilon}}\lt \|    \omega \rt\|_{X_{a,m,\epsilon}}\lt(\lt\|      \lt(\px^2,\py^2\rt)|D_x|^\frac13 d \rt\|_{X_{a,m,\epsilon}} +\lt\|     \nabla|D_x|^\frac13 d \rt\|_{X_{a,m,\epsilon}} \rt) .
            \end{align*}
	\end{Lem}

	\begin{proof}[\bf Proof of Proposition \ref{lem:est of pypy dx13 d}.]
		For $i=1,2$, one applies $\p^2$ to $\eqref{eq:main}_3$ and has that
		\begin{equation}\begin{aligned} \label{eq:pypy d}
				\partial_t \p^2  d +  y \partial_x \p^2  d -  \frac{1}{A} \Delta \p^2 d= -\frac1A \p^2 (u\cdot \nabla d)+  \frac1A   \p^2  \lt(|\nabla d|^2 (d + e_1)\rt)    -  2(i-1)\px\py d.
		\end{aligned}\end{equation}
		By taking the inner product of \eqref{eq:pypy d} with $2 \mathcal{M} e^{2 a A^{-\frac{1}{3}}\left|D_x\right|^{\frac{2}{3}}  t } \langle D_x \rangle^{2 m} \langle\frac{1}{D_x} \rangle^{2 \epsilon} \p^2 |D_x|^\frac23 d$, and using \eqref{eq:bound of M}, \eqref{eq:crucial M} and \eqref{eq:m}, for $0<a<\frac{1}{16(1+2 \pi)}$ we have
		\begin{align*}
			&\quad \frac{d}{d t}\lt\|\sqrt{\mathcal{M}} e^{a A^{-\frac{1}{3}}\left|D_x\right|^{\frac{2}{3} }t}   \p^2 |D_x|^\frac13 d\rt\|_{Y_{m,\epsilon}}^2   +    \frac1A \lt\|e^{a A^{-\frac{1}{3}}\left|D_x\right|^{\frac{2}{3} }t}   \nabla \p^2 |D_x|^\frac13 d\rt\|_{Y_{m,\epsilon}}^2 \\
			& \quad\quad +\frac{1}{16 A^{\frac{1}{3}}}\lt\|e^{a A^{-\frac{1}{3}}\left|D_x\right|^{\frac{2}{3} } t }  \left|D_x\right|^{\frac{1}{3}} \p^2 |D_x|^\frac13 d\rt\|_{Y_{m,\epsilon}}^2  +  \lt\|e^{a A^{-\frac{1}{3}}\left|D_x\right|^{\frac{2}{3} }t }   \partial_x \nabla  \Delta^{-1} \p^2 |D_x|^\frac13 d  \rt\|_{L^2}^2 \\
			&\lesssim \frac2A  \left| \Re\left( \p^2   \lt( u \cdot \nabla d\rt) \left\lvert\, \mathcal{M} e^{2 a A^{-\frac{1}{3}}\left|D_x\right|^{\frac{2}{3} } t}\langle D_x\rangle^{2 m}\langle\frac{1}{D_x}\rangle^{2 \epsilon} \p^2  |D_x|^\frac23 d\right.\right)  \right|\\
			&\quad + \frac2A  \left| \Re\left( \p^2   \lt(|\nabla d|^2 (d + e_1)\rt) \left\lvert\, \mathcal{M} e^{2 a A^{-\frac{1}{3}}\left|D_x\right|^{\frac{2}{3} } t}\langle D_x\rangle^{2 m}\langle\frac{1}{D_x}\rangle^{2 \epsilon} \p^2  |D_x|^\frac23 d\right.\right)  \right|\\
			&\quad + 4(i-1)\left|   \Re\left( \px\py d \left\lvert\, \mathcal{M} e^{2 a A^{-\frac{1}{3}}\left|D_x\right|^{\frac{2}{3} } t}\langle D_x\rangle^{2 m}\langle\frac{1}{D_x}\rangle^{2 \epsilon} \p^2 |D_x|^\frac23 d \right.\right) \right| .
		\end{align*}
		Notice that when $i=2$, 
		\begin{align*}
			&  \quad \left|   \Re\left( \px\py d \left\lvert\, \mathcal{M} e^{2 a A^{-\frac{1}{3}}\left|D_x\right|^{\frac{2}{3} } t}\langle D_x\rangle^{2 m}\langle\frac{1}{D_x}\rangle^{2 \epsilon} \py^2 |D_x|^\frac23 d \right.\right) \right|  \\
			&\lesssim \lt\|e^{a A^{-\frac{1}{3}}\left|D_x\right|^{\frac{2}{3} } t }   \nabla |D_x|^\frac13 d\rt\|_{Y_{m,\epsilon}}\lt \|e^{a A^{-\frac{1}{3}}\left|D_x\right|^{\frac{2}{3} } t }   \nabla \px  \py |D_x|^\frac13 d\rt\|_{Y_{m,\epsilon}} 
		\end{align*}
		implies 
		\begin{align*}
			&   \int_{0}^{t}\left|   \Re\left( \px \py d \left\lvert\, \mathcal{M} e^{2 a A^{-\frac{1}{3}}\left|D_x\right|^{\frac{2}{3} } t}\langle D_x\rangle^{2 m}\langle\frac{1}{D_x}\rangle^{2 \epsilon} \py^2 |D_x|^\frac23 d \right.\right) \right| dt \\
			&\lesssim A  \lt\|     |D_x|^\frac13 d\rt\|_{X_{a,m,\epsilon}} \lt\|    \px\py |D_x|^\frac13 d\rt\|_{X_{a,m,\epsilon}} ,
		\end{align*}
		which, along with Lemmas \ref{lem:est of pypy nabla d2 d}, \ref{lem:est of pypy nabla d2 e1} and \ref{lem:est of py^2 u}, completes the proof.
	\end{proof}

	In the remainder of this section, we provide the proofs of the three lemmas one by one.

	\begin{proof}[\bf Proof of Lemma \ref{lem:est of pypy nabla d2 d}.]
		Using integration by parts and the Fourier transform, we have
		\begin{align*} 
			& \quad   \frac1A \left|\Re\left( \p^2  \left(  |\nabla d|^2 d \right) \left\lvert\, \mathcal{M} e^{2 a A^{-\frac{1}{3}}\left|D_x\right|^{\frac{2}{3} }t }\langle D_x\rangle^{2 m  }\langle\frac{1}{D_x}\rangle^{2 \epsilon} \p^2 |D_x|^\frac23  d\right.\right)\right|\\
			&\lesssim \frac1A \int_{\mathbb{R}^3} e^{2 a A^{-\frac{1}{3}}|k|^{\frac{2}{3} } t }\langle k\rangle^{2 m   }\langle\frac{1}{k}\rangle^{2 \epsilon} |k|^\frac23 \\
			&\qqquad \quad \times \left|\int_{\mathbb{R}} \mathcal{M}\left(k, D_y\right)   \q^3 d_k(y) \cdot   \q d_l(y)     \nabla_{k-l-\eta} d_{k-l-\eta}(y) \cdot \nabla_\eta d_{\eta}(y)   d y\right| d k d ld\eta\\
			&\quad + \frac1A \int_{\mathbb{R}^3} e^{2 a A^{-\frac{1}{3}}|k|^{\frac{2}{3} } t }\langle k\rangle^{2 m   }\langle\frac{1}{k}\rangle^{2 \epsilon} |k|^\frac23 \\
			&\qqquad \quad \times \left|\int_{\mathbb{R}} \mathcal{M}\left(k, D_y\right)   \q^3 d_k(y) \cdot    d_l(y)    \nabla_{k-l-\eta} d_{k-l-\eta}(y) \cdot \nabla_\eta \q d_{\eta}(y)   d y\right| d k d ld\eta \\
            &=: II_1+II_2. \alabel{eq:nabla^2 d}
		\end{align*}

		$\bullet$~\textbf{In $\mathcal{R}_{\text{res}}(k,l)\times\mathbb{R}$.} With help  of \eqref{eq:trick1} and Corollary \ref{cor:L infty} ($\alpha=\beta=\frac13$), for the term of $II_1$ we arrive at
		\begin{align*}  
			&\quad \int_{\mathcal{R}_{\text{res}(k,l)}\times\mathbb{R}} e^{2 a A^{-\frac{1}{3}}|k|^{\frac{2}{3} } t }\langle k\rangle^{2 m   }\langle\frac{1}{k}\rangle^{2 \epsilon} |k|^\frac23 \\
			&\qqquad \qqquad \qquad \times \left|\int_{\mathbb{R}} \mathcal{M}\left(k, D_y\right)  \q^3 d_k(y) \cdot   \q d_l(y)    \nabla_{k-l-\eta}   d_{k-l-\eta}(y) \cdot \nabla_{\eta}  d_{\eta}(y)   d y\right| d k d ld\eta \\
			&\lesssim \lt\|e^{a A^{-\frac{1}{3}}|k|^{\frac{2}{3} }t}\langle k\rangle^{m  }\langle\frac{1}{k}\rangle^\epsilon |k|^\frac13   \lt\|\q^3 d_k\rt\|_{L_y^2}\rt\|_{L_k^2}     \lt\|e^{a A^{-\frac{1}{3}}|l|^{\frac{2}{3} } t } \lt \| \q d_l \rt\|_{L_y^{\infty}}\rt\|_{L_l^1} \\
			&\quad \times\lt\|e^{a A^{-\frac{1}{3}}|k-l|^{\frac{2}{3} } t }\langle k-l\rangle^{m   } \langle\frac{1}{k-l}\rangle^\epsilon |k-l|^\frac13 \int_{\mathbb{R}}   \lt\| \nabla_{k-l-\eta} d_{k-l-\eta} \cdot \nabla_{\eta} d_{\eta}\rt\|_{L_y^2} d\eta \rt\|_{L_{k-l}^2} \\
            &\lesssim\lt \|e^{a A^{-\frac{1}{3}}\left|D_x\right|^{\frac{2}{3} } t }    \nabla \lt(\nabla^2 |D_x|^\frac13 d\rt) \rt\|_{Y_{m,\epsilon}}  \lt\|e^{a A^{-\frac{1}{3}}\left|D_x\right|^{\frac{2}{3} }t}    \nabla|D_x|^\frac13 d\rt\|_{Y_{m,\epsilon}}^\frac12 \lt\|e^{a A^{-\frac{1}{3}}\left|D_x\right|^{\frac{2}{3} }t}    \py \nabla |D_x|^\frac13 d\rt\|_{Y_{m,\epsilon}}^\frac12      \\
			&\quad \times\lt\|e^{a A^{-\frac{1}{3}}|k-l|^{\frac{2}{3} } t }\langle k-l\rangle^{m   } \langle\frac{1}{k-l}\rangle^\epsilon |k-l|^\frac13 \int_{\mathbb{R}}   \lt\| \nabla_{k-l-\eta} d_{k-l-\eta} \rt\|_{L_y^2}   \lt\| \nabla_{\eta} d_{\eta}\rt\|_{L_y^\infty} d\eta \rt\|_{L_{k-l}^2}
            \end{align*}
and using Lemma \ref{lem:Preliminary estimates} ($i=1$, $f=g=\nabla d$), it can be controlled by
            \begin{align*}
			&C\lt \|e^{a A^{-\frac{1}{3}}\left|D_x\right|^{\frac{2}{3} } t }    \nabla \lt(\nabla^2 |D_x|^\frac13 d\rt) \rt\|_{Y_{m,\epsilon}}  \lt\|e^{a A^{-\frac{1}{3}}\left|D_x\right|^{\frac{2}{3} }t}    \nabla|D_x|^\frac13 d\rt\|_{Y_{m,\epsilon}}^\frac32 \lt\|e^{a A^{-\frac{1}{3}}\left|D_x\right|^{\frac{2}{3} }t}    \nabla^2 |D_x|^\frac13 d\rt\|_{Y_{m,\epsilon}}^\frac12      \\
			&\quad\times\lt\|  e^{a A^{-\frac{1}{3}}|D_x|^{\frac{2}{3} } t }     |D_x|^{\frac13} \nabla |D_x|^\frac13 d    \rt\|_{Y_{m,\epsilon}}^\frac12\lt\|  e^{a A^{-\frac{1}{3}}|D_x|^{\frac{2}{3} } t }     \nabla \lt(\nabla |D_x|^\frac13 d\rt)   \rt \|_{Y_{m,\epsilon}}^\frac12. \alabel{eq:est of |px d|^2 py d 1}  
		\end{align*}
		Similarly,  for the term of $II_2$, using \eqref{eq:trick1} and Corollary \ref{cor:L infty} ($\alpha=\beta=\frac13$) again, we have
		\begin{align*}  
			&\quad \int_{\mathcal{R}_{\text{res}(k,l)}\times\mathbb{R}} e^{2 a A^{-\frac{1}{3}}|k|^{\frac{2}{3} } t }\langle k\rangle^{2 m   }\langle\frac{1}{k}\rangle^{2 \epsilon} |k|^\frac23 \\
			&\qqquad \qqquad \qquad \times \left|\int_{\mathbb{R}} \mathcal{M}\left(k, D_y\right)  \q^3 d_k(y) \cdot    d_l(y)  \nabla_{k-l-\eta} d_{k-l-\eta}(y) \cdot \nabla_\eta \q d_{\eta}(y)   d y\right| d k d ld\eta \\
			&\lesssim\lt \|e^{a A^{-\frac{1}{3}}|k|^{\frac{2}{3} }t}\langle k\rangle^{m  }\langle\frac{1}{k}\rangle^\epsilon |k|^\frac13  \lt \|\q^3 d_k\rt\|_{L_y^2}\rt\|_{L_k^2}     \lt\|e^{a A^{-\frac{1}{3}}|l|^{\frac{2}{3} } t }\lt \|   d_l \rt\|_{L_y^{\infty}}\rt\|_{L_l^1} \\
			&\quad \times\lt\|e^{a A^{-\frac{1}{3}}|k-l|^{\frac{2}{3} } t }\langle k-l\rangle^{m   } \langle\frac{1}{k-l}\rangle^\epsilon |k-l|^\frac13 \int_{\mathbb{R}}  \lt\| \nabla_{k-l-\eta} d_{k-l-\eta} \cdot \nabla_{\eta} \q d_{\eta}\rt\|_{L_y^2} d\eta \rt\|_{L_{k-l}^2}\\ 
            &\lesssim \lt\|  e^{a A^{-\frac{1}{3}}|D_x|^{\frac{2}{3} } t }      \nabla\lt(\nabla^2|D_x|^{\frac13} d\rt)   \rt \|_{Y_{m,\epsilon}}\lt\|e^{a A^{-\frac{1}{3}}\left|D_x\right|^{\frac{2}{3} }t}      |D_x|^\frac13 d\rt\|_{Y_{m,\epsilon}}^\frac12 \lt\|e^{a A^{-\frac{1}{3}}\left|D_x\right|^{\frac{2}{3} }t}    \nabla |D_x|^\frac13 d\rt\|_{Y_{m,\epsilon}}^\frac12      \\
			&\quad \times\lt\|e^{a A^{-\frac{1}{3}}|k-l|^{\frac{2}{3} } t }\langle k-l\rangle^{m   } \langle\frac{1}{k-l}\rangle^\epsilon |k-l|^\frac13 \int_{\mathbb{R}}  \lt\| \nabla_{k-l-\eta} d_{k-l-\eta}\rt\|_{L_y^2}  \lt\| \nabla_{\eta} \q d_{\eta}\rt\|_{L_y^\infty} d\eta \rt\|_{L_{k-l}^2}
\end{align*}
and using Lemma \ref{lem:Preliminary estimates} ($i=1$, $f=\nabla d$, $g=\nabla^2 d$) it can be bounded by
            \begin{align*}
			&C \lt\|  e^{a A^{-\frac{1}{3}}|D_x|^{\frac{2}{3} } t }      \nabla\lt(\nabla^2|D_x|^{\frac13} d\rt)   \rt \|_{Y_{m,\epsilon}}^{\frac32}\lt\|e^{a A^{-\frac{1}{3}}\left|D_x\right|^{\frac{2}{3} }t}      |D_x|^\frac13 d\rt\|_{Y_{m,\epsilon}}^\frac12 \lt\|e^{a A^{-\frac{1}{3}}\left|D_x\right|^{\frac{2}{3} }t}    \nabla |D_x|^\frac13 d\rt\|_{Y_{m,\epsilon}}^\frac32      \\
			&\quad\times \lt\|  e^{a A^{-\frac{1}{3}}|D_x|^{\frac{2}{3} } t }      |D_x|^{\frac13}\nabla^2 |D_x|^{\frac13} d    \rt\|_{Y_{m,\epsilon}}^{\frac12}.\alabel{eq:est of px d px py d d 1}   
		\end{align*}

		$\bullet$~\textbf{In $\mathcal{R}_{\text{HL}}(k,l)\times\mathbb{R}$.}
		By \eqref{eq:GN1} and \eqref{eq:trick2}, on one hand, we have
		\begin{align*}   
			&\quad\int_{\mathcal{R}_{\text{HL}}(k,l)\times\mathbb{R}}   e^{2 a A^{-\frac{1}{3}}|k|^{\frac{2}{3} } t }\langle k\rangle^{2 m   }\langle\frac{1}{k}\rangle^{2 \epsilon} |k|^\frac23 \\
				&\qqquad \qqquad \qquad \times \left|\int_{\mathbb{R}} \mathcal{M}\left(k, D_y\right)  \q^3 d_k(y) \cdot   \q d_l(y)    \nabla_{k-l-\eta}   d_{k-l-\eta}(y) \cdot \nabla_{\eta}  d_{\eta}(y)   d y\right| d k d ld\eta \\
			&\lesssim  \|e^{a A^{-\frac{1}{3}}|k|^{\frac{2}{3} }t}\langle k\rangle^{m   }\langle\frac{1}{k}\rangle^\epsilon |k|^\frac13 \|\q^3  d_k\|_{L_y^2}\|_{L_k^2}     \|e^{a A^{-\frac{1}{3}}|l|^{\frac{2}{3} } t } \langle l\rangle^{m   }\langle\frac{1}{l}\rangle^\epsilon |l|^\frac12 \| \q d_l \|_{L_y^{\infty}}\|_{L_l^2} \\
			&\quad \times\|e^{a A^{-\frac{1}{3}}|k-l|^{\frac{2}{3} } t }\langle k-l\rangle^{m   } \langle\frac{1}{k-l}\rangle^\epsilon    \int_{\mathbb{R}} \| \nabla_{k-l-\eta} d_{k-l-\eta} \|_{L_y^2} \| \nabla_{\eta} d_{\eta}\|_{L_y^\infty} d\eta \|_{L_{k-l}^2} \\
			&\quad\times\| |k-l|^{-\frac16} \langle k-l\rangle^{-m}\langle\frac{1}{k-l}\rangle^{-\epsilon}\|_{L_{k-l}^2}\\
			&\lesssim  \lt\|e^{a A^{-\frac{1}{3}}\left|D_x\right|^{\frac{2}{3} } t }    \nabla (\nabla^2 |D_x|^\frac13 d )\rt\|_{Y_{m,\epsilon}}  \lt\|e^{a A^{-\frac{1}{3}}\left|D_x\right|^{\frac{2}{3} }t}    |D_x|^\frac13 \nabla |D_x|^\frac13 d\rt|_{Y_{m,\epsilon}}^\frac12 \lt\|e^{a A^{-\frac{1}{3}}\left|D_x\right|^{\frac{2}{3} }t}    \nabla^2|D_x|^\frac13 d\rt\|_{Y_{m,\epsilon}}^\frac12     \\
			&\quad \times   \lt\|  e^{a A^{-\frac{1}{3}}|D_x|^{\frac{2}{3} } t }      \nabla|D_x|^{\frac13}d    \rt\|_{Y_{m,\epsilon}}^{\frac32}\lt\|  e^{a A^{-\frac{1}{3}}|D_x|^{\frac{2}{3} } t }     \nabla(\nabla|D_x|^{\frac13} d )   \rt\|_{Y_{m,\epsilon} }^\frac12,
			\\
		  \alabel{eq:est of |px d|^2 py d 2}
		\end{align*}
        where we applied Lemma \ref{lem:Preliminary estimates} ($i=0$, $f=g=\nabla d$) to the second line.
		On the other hand, for the term of $II_2$, using Lemma \ref{lem:Preliminary estimates} ($i=0$, $f=\nabla d$ and $g=\nabla^2 d$) we obtain
		\begin{align*}   
			&\quad\int_{\mathcal{R}_{\text{HL}}(k,l)\times\mathbb{R}}   e^{2 a A^{-\frac{1}{3}}|k|^{\frac{2}{3} } t }\langle k\rangle^{2 m   }\langle\frac{1}{k}\rangle^{2 \epsilon} |k|^\frac23 \\
			&\qqquad \qqquad \qquad \times \left|\int_{\mathbb{R}} \mathcal{M}\left(k, D_y\right)  \q^3 d_k(y) \cdot    d_l(y)  \nabla_{k-l-\eta} d_{k-l-\eta}(y) \cdot \nabla_\eta \q d_{\eta}(y)   d y\right| d k d ld\eta \\
			&\lesssim \lt \|e^{a A^{-\frac{1}{3}}|k|^{\frac{2}{3} }t}\langle k\rangle^{m   }\langle\frac{1}{k}\rangle^\epsilon |k|^\frac13 \lt\|\q^3 d_k\rt\|_{L_y^2}\rt\|_{L_k^2}   \lt  \|e^{a A^{-\frac{1}{3}}|l|^{\frac{2}{3} } t } \langle l\rangle^{m   }\langle\frac{1}{l}\rangle^\epsilon |l|^\frac12\lt \|  d_l \rt\|_{L_y^{\infty}}\rt\|_{L_l^2} \\
			&\quad \times\lt\|e^{a A^{-\frac{1}{3}}|k-l|^{\frac{2}{3} } t }\langle k-l\rangle^{m   } \langle\frac{1}{k-l}\rangle^\epsilon  \int_{\mathbb{R}}  \lt\| \nabla_{k-l-\eta} d_{k-l-\eta} \rt\|_{L_y^2}  \lt\| \nabla_{\eta} \q d_{\eta}\rt\|_{L_y^\infty} d\eta\rt \|_{L_{k-l}^2} \\
			&\quad\times\lt \| |k-l|^{-\frac16}\langle k-l\rangle^{-m}\langle\frac{1}{k-l}\rangle^{-\epsilon}\rt\|_{L_{k-l}^2}\\
			&\lesssim \lt\|e^{a A^{-\frac{1}{3}}\left|D_x\right|^{\frac{2}{3} } t }    \nabla \lt(\nabla^2 |D_x|^\frac13 d\rt )\rt\|_{Y_{m,\epsilon}}^{\frac32}  \lt\|e^{a A^{-\frac{1}{3}}\left|D_x\right|^{\frac{2}{3} }t}   |D_x|^\frac23   d \rt\|_{Y_{m,\epsilon}}^\frac12\lt\|e^{a A^{-\frac{1}{3}}\left|D_x\right|^{\frac{2}{3} }t}    \nabla |D_x|^\frac13 d\rt\|_{Y_{m,\epsilon}}^\frac32      \\
			&\quad\times\lt \|  e^{a A^{-\frac{1}{3}}|D_x|^{\frac{2}{3} } t }     \nabla^2 |D_x|^\frac13 d    \rt\|_{Y_{m,\epsilon}}^\frac12.
		 \alabel{eq:est of py d py^2 d d 2}
		\end{align*}
		
		$\bullet$~~\textbf{In $\mathcal{R}_{\text{LH}}\times\mathbb{R}$.}
		In virtue of \eqref{eq:trick3} and \eqref{eq:trick3.5}, we have
		\begin{align*}  
			&\quad\int_{\mathcal{R}_{\text{LH}}(k,l)\times\mathbb{R}} e^{2 a A^{-\frac{1}{3}}|k|^{\frac{2}{3} } t }\langle k\rangle^{2 m   }\langle\frac{1}{k}\rangle^{2 \epsilon} |k|^\frac23 \\
		&\qqquad \qqquad \qquad \times \left|\int_{\mathbb{R}} \mathcal{M}\left(k, D_y\right)  \q^3 d_k(y) \cdot   \q d_l(y)    \nabla_{k-l-\eta}   d_{k-l-\eta}(y) \cdot \nabla_{\eta}  d_{\eta}(y)   d y\right| d k d ld\eta \\
			&\lesssim \lt \|e^{a A^{-\frac{1}{3}}|k|^{\frac{2}{3} }t}\langle k\rangle^{m   }\langle\frac{1}{k}\rangle^\epsilon |k|^\frac13 \lt\|\q^3  d_k\rt\|_{L_y^2}\rt\|_{L_k^2}   \lt  \|e^{a A^{-\frac{1}{3}}|l|^{\frac{2}{3} } t } \langle l\rangle^{m   }\langle\frac{1}{l}\rangle^\epsilon |l|^\frac12 \lt\| \q d_l \rt\|_{L_y^{\infty}}\rt\|_{L_l^2} \\
		&\quad \times\lt\|e^{a A^{-\frac{1}{3}}|k-l|^{\frac{2}{3} } t }\langle k-l\rangle^{m   } \langle\frac{1}{k-l}\rangle^\epsilon    \int_{\mathbb{R}} \lt\| \nabla_{k-l-\eta} d_{k-l-\eta} \cdot \nabla_{\eta} d_{\eta}\rt\|_{L_y^2} d\eta \rt\|_{L_{k-l}^2} \\
			&\quad\times
			\lt\|  \left(\langle\frac{1}{k}\rangle^\epsilon  +\langle\frac{1}{k}\rangle^{\frac16  -\epsilon}  \right)\langle k\rangle^{-m}\rt\|_{L_k^2}
		 \alabel{eq:est of |px d|^2 py d 3}
		\end{align*}
		and
		\begin{align*}  
			&\quad\int_{\mathcal{R}_{\text{LH}}(k,l)\times\mathbb{R}} e^{2 a A^{-\frac{1}{3}}|k|^{\frac{2}{3} } t }\langle k\rangle^{2 m   }\langle\frac{1}{k}\rangle^{2 \epsilon} |k|^\frac23 \\
				&\qqquad \qqquad \qquad \times \left|\int_{\mathbb{R}} \mathcal{M}\left(k, D_y\right)  \q^3 d_k(y) \cdot    d_l(y)  \nabla_{k-l-\eta} d_{k-l-\eta}(y) \cdot \nabla_\eta \q d_{\eta}(y)   d y\right| d k d ld\eta \\
			& \lesssim \lt\|e^{a A^{-\frac{1}{3}}|k|^{\frac{2}{3} } t }\langle k\rangle^{m }\langle\frac{1}{k}\rangle^\epsilon  |k|^\frac13 \|\py^3 d_k\|_{L_y^2}\rt\|_{L_k^2}  \lt\|e^{a A^{-\frac{1}{3}}|l|^{\frac{2}{3} } t }\langle l\rangle^{m   }\langle\frac{1}{l}\rangle^\epsilon  |l|^{\frac12}  \| d_l\|_{L_y^\infty}\rt\|_{L_l^2} \\
			& \quad \times   \lt\|e^{a A^{-\frac{1}{3}}|k-l|^{\frac{2}{3} } t }\langle k-l\rangle^{m   } \langle\frac{1}{k-l}\rangle^\epsilon    \int_{\mathbb{R}}  \| \nabla_{k-l-\eta} d_{k-l-\eta}\|_{L_y^2} \|\nabla_\eta \q d_{\eta}\|_{L_y^\infty} d\eta \rt\|_{L_{k-l}^2} 
			\\
            &\quad \times \lt\|  \left(\langle\frac{1}{k}\rangle^\epsilon  +\langle\frac{1}{k}\rangle^{\frac16 -\epsilon}  \right)\langle k\rangle^{-m}\rt\|_{L_k^2}.
\alabel{eq:est of py d py^2 d d 3 1}
		\end{align*}
        Then, using \eqref{eq:GN1} and Lemma \ref{lem:Preliminary estimates}, \eqref{eq:est of |px d|^2 py d 3} and \eqref{eq:est of py d py^2 d d 3 1} are controlled by the last term of \eqref{eq:est of |px d|^2 py d 2} and \eqref{eq:est of py d py^2 d d 2}, respectively.

	Collecting \eqref{eq:est of |px d|^2 py d 1}--\eqref{eq:est of py d py^2 d d 2}, we arrive
		\begin{align*} 
		& \quad   II_1+II_2 \\ 
	&\lesssim \frac{1}{A}\sum_{j\in\{0,1\}}\lt\|e^{a A^{-\frac{1}{3}}\left|D_x\right|^{\frac{2}{3} } t }    \nabla \lt(\nabla^2 |D_x|^\frac13 d\rt)\rt \|_{Y_{m,\epsilon}}  \lt\|e^{a A^{-\frac{1}{3}}\left|D_x\right|^{\frac{2}{3} }t}    |D_x|^{\frac{j}{3}}\nabla |D_x|^\frac13 d\rt\|_{Y_{m,\epsilon}}^\frac12 \\&\quad\times\lt\|e^{a A^{-\frac{1}{3}}\left|D_x\right|^{\frac{2}{3} }t}    \nabla^2 |D_x|^\frac13 d\rt\|_{Y_{m,\epsilon}}^\frac12      
\lt\|  e^{a A^{-\frac{1}{3}}|D_x|^{\frac{2}{3} } t }     \nabla |D_x|^\frac13 d    \rt\|_{Y_{m,\epsilon}}\lt\|  e^{a A^{-\frac{1}{3}}|D_x|^{\frac{2}{3} } t }     |D_x|^{\frac{1-j}{3}} \nabla |D_x|^\frac13 d    \rt\|_{Y_{m,\epsilon}}^\frac12\\&\qquad\times\lt\|  e^{a A^{-\frac{1}{3}}|D_x|^{\frac{2}{3} } t }     \nabla\lt(\nabla|D_x|^\frac13 d \rt)   \rt\|_{Y_{m,\epsilon}}^\frac12\\
&\quad+\frac{1}{A}\sum_{j\in\{0,1\}}\lt\|e^{a A^{-\frac{1}{3}}\left|D_x\right|^{\frac{2}{3} } t }    \nabla \lt(\nabla^2 |D_x|^\frac13 d\rt) \rt\|_{Y_{m,\epsilon}}^{\frac32} \lt \|e^{a A^{-\frac{1}{3}}\left|D_x\right|^{\frac{2}{3} }t}    |D_x|^\frac{j+1}{3} d\rt\|_{Y_{m,\epsilon}}^\frac12\\
&\qquad\times \lt\|e^{a A^{-\frac{1}{3}}\left|D_x\right|^{\frac{2}{3} }t}    \nabla |D_x|^\frac13 d\rt\|_{Y_{m,\epsilon}}^\frac32\lt\|  e^{a A^{-\frac{1}{3}}|D_x|^{\frac{2}{3} } t }     |D_x|^{\frac{1-j}{3}}\nabla^2 |D_x|^\frac13 d    \rt\|_{Y_{m,\epsilon}}^\frac12     \\
&\lesssim \frac{1}{A^{\frac16}} \lt( \lt\|     \nabla^2|D_x|^\frac13 d\rt \|_{X_{a,m,\epsilon}}^{\frac32}\lt\|     \nabla|D_x|^{\frac13} d \rt\|_{X_{a,m,\epsilon}}^{\frac52}+\lt\|    \nabla^2 |D_x|^\frac13 d \rt\|_{X_{a,m,\epsilon}}^2\lt\| |D_x|^\frac13 d \rt\|_{X_{a,m,\epsilon}}^\frac12  \lt\|  \nabla|D_x|^\frac13 d \rt\|_{X_{a,m,\epsilon}}^{\frac32}\rt),
	\end{align*}
	which completes the proof by recalling \eqref{eq:nabla^2 d}.
	\end{proof}

	\begin{proof}[\bf Proof of Lemma \ref{lem:est of pypy nabla d2 e1}.]
		Note that
		\begin{align*} 
			& \quad   \frac1A   \left|\Re\left( \p^2  \left( |\nabla d|^2 e_1\right) \left\lvert\, \mathcal{M} e^{2 a A^{-\frac{1}{3}}\left|D_x\right|^{\frac{2}{3} }t }\langle D_x\rangle^{2 m}\langle\frac{1}{D_x}\rangle^{2 \epsilon} \p^2 |D_x|^\frac23  d\right.\right)\right| \\
			&\lesssim \frac1A \int_{\mathbb{R}^2} e^{2 a A^{-\frac{1}{3}}|k|^{\frac{2}{3} } t }\langle k\rangle^{2 m   }\langle\frac{1}{k}\rangle^{2 \epsilon} |k|^\frac23 \left|\int_{\mathbb{R}} \mathcal{M}\left(k, D_y\right)   \q^3 d_k^1(y)  l \lt(k-l \rt)  \q d_l(y)  \cdot d_{k-l }(y)  d y\right| d k d l \\
			&\quad + \frac1A \int_{\mathbb{R}^2} e^{2 a A^{-\frac{1}{3}}|k|^{\frac{2}{3} } t }\langle k\rangle^{2 m   }\langle\frac{1}{k}\rangle^{2 \epsilon} |k|^\frac23 \left|\int_{\mathbb{R}} \mathcal{M}\left(k, D_y\right)   \q^3 d_k^1(y)  \q\py d_l(y)    \cdot \py d_{k-l}(y)  d y\right| d k d l  .  
		\end{align*}

		$\bullet$~\textbf{In ${\mathcal{R}_{\text{res}}(k,l)}$.} It follows from \eqref{eq:trick1} 
		 that
		\begin{align*}  
			&\quad \int_{\mathcal{R}_{\text{res}}(k,l)} e^{2 a A^{-\frac{1}{3}}|k|^{\frac{2}{3} } t }\langle k\rangle^{2 m   }\langle\frac{1}{k}\rangle^{2 \epsilon} |k|^\frac23 \left|\int_{\mathbb{R}} \mathcal{M}\left(k, D_y\right)  \q^3 d_k^1(y)  l \lt(k-l \rt)  \q d_l(y)  \cdot d_{k-l }(y)  d y\right| d k d l \\
			&\lesssim \lt\|e^{a A^{-\frac{1}{3}}|k|^{\frac{2}{3} }t}\langle k\rangle^{m  }\langle\frac{1}{k}\rangle^\epsilon |k|^\frac13 \lt \|\q^3 d_k^1\rt\|_{L_y^2}\rt\|_{L_k^2}    \lt \|e^{a A^{-\frac{1}{3}}|l|^{\frac{2}{3} } t }  |l|^\frac23  \lt\|\q d_l \rt\|_{L_y^{\infty}}\rt\|_{L_l^1} \\
			&\quad \times\lt\|e^{a A^{-\frac{1}{3}}|k-l|^{\frac{2}{3} } t }\langle k-l\rangle^{m   } \langle\frac{1}{k-l}\rangle^\epsilon      |k-l |^\frac53     \lt\|d_{k-l} \rt\|_{L_y^2} d\eta \rt\|_{L_{k-l}^2} \\
			&\lesssim \lt\|e^{a A^{-\frac{1}{3}}\left|D_x\right|^{\frac{2}{3} } t }   \nabla \lt(\nabla^2|D_x|^\frac13 d^1\rt) \rt\|_{Y_{m,\epsilon}}\lt\|e^{a A^{-\frac{1}{3}}\left|D_x\right|^{\frac{2}{3} }t}     \nabla^2 |D_x|^\frac13 d\rt\|_{Y_{m,\epsilon}}\lt\|e^{a A^{-\frac{1}{3}}\left|D_x\right|^{\frac{2}{3} } t }   |D_x|^\frac13 \nabla|D_x|^{\frac13} d \rt\|_{Y_{m,\epsilon}}, \alabel{eq:est of px d px py d 1}
		\end{align*}
        where we used H$\text{\"{o}}$lder inequality and \eqref{eq:GN1} to get
        \begin{align*}
            &\lt\|e^{a A^{-\frac{1}{3}}|l|^{\frac{2}{3} } t }  |l|^\frac23  \lt\|\q d_l\rt \|_{L_y^{\infty}}\rt\|_{L_l^1}\\
            &\lesssim\lt\|e^{a A^{-\frac{1}{3}}|l|^{\frac{2}{3} } t }  |l|^\frac23  \lt\|\q d_l \rt\|_{L_y^{2}}^{\frac12}\lt\|\py\q d_l \rt\|_{L_y^{2}}^{\frac12}\rt\|_{L_l^1}\\
            &\lesssim\lt \|e^{a A^{-\frac{1}{3}}|l|^{\frac{2}{3} } t } \langle l \rangle^{m} \langle \frac{1}{l} \rangle^{\epsilon} \lt\| |l|\q |l|^\frac13 d_l \rt\|_{L_y^{2}}^{\frac12}\lt\|\py\q |l|^{\frac13}d_l \rt\|_{L_y^{2}}^{\frac12}\rt\|_{L_l^2}\lt\||l|^{-\frac16}\langle l \rangle^{-m} \langle \frac{1}{l} \rangle^{-\epsilon}\rt\|_{L^2_l}\\
            &\lesssim \lt\|e^{a A^{-\frac{1}{3}}\left|D_x\right|^{\frac{2}{3} }t}     \nabla^2 |D_x|^\frac13 d\rt\|_{Y_{m,\epsilon}}
        \end{align*}
		Moreover,  by Corollary \ref{cor:L infty} ($\alpha=\frac23,~\beta=\frac13$) we have 
		\begin{align*}  
			&\quad \int_{\mathcal{R}_{\text{res}}(k,l)} e^{2 a A^{-\frac{1}{3}}|k|^{\frac{2}{3} } t }\langle k\rangle^{2 m   }\langle\frac{1}{k}\rangle^{2 \epsilon} |k|^\frac23 \left|\int_{\mathbb{R}} \mathcal{M}\left(k, D_y\right) \q^3  d_k^1(y)  \q\py d_l(y)    \cdot \py d_{k-l}(y)  d y\right| d k d l \\
			&\lesssim \lt\|e^{a A^{-\frac{1}{3}}|k|^{\frac{2}{3} }t}\langle k\rangle^{m  }\langle\frac{1}{k}\rangle^\epsilon |k|^\frac13 \lt \|\nabla\q^2 d_k^1\rt\|_{L_y^2}\rt\|_{L_k^2}    \lt \|e^{a A^{-\frac{1}{3}}|l|^{\frac{2}{3} } t }   \lt\| \q\py d_l \rt\|_{L_y^{\infty}}\rt\|_{L_l^1} \\
			&\quad \times\lt\|e^{a A^{-\frac{1}{3}}|k-l|^{\frac{2}{3} } t }\langle k-l\rangle^{m   } \langle\frac{1}{k-l}\rangle^\epsilon   |k-l|^\frac13   \lt\| \py d_{k-l}\rt\|_{L_y^2}   \rt\|_{L_{k-l}^2} \\
			&\lesssim \lt\|e^{a A^{-\frac{1}{3}}\left|D_x\right|^{\frac{2}{3} } t }   \nabla \lt(\nabla^2 |D_x|^\frac13 d^1 \rt)\rt\|_{Y_{m,\epsilon}} \lt\|e^{a A^{-\frac{1}{3}}\left|D_x\right|^{\frac{2}{3} }t}    |D_x|^\frac13 \nabla^2 |D_x|^\frac13 d\rt\|_{Y_{m,\epsilon}}^\frac12       \\
			&\quad \times\lt\|e^{a A^{-\frac{1}{3}}\left|D_x\right|^{\frac{2}{3} }t}   \nabla \lt(\nabla^2 |D_x|^\frac13 d\rt)\rt\|_{Y_{m,\epsilon}}^\frac12 \lt\|e^{a A^{-\frac{1}{3}}\left|D_x\right|^{\frac{2}{3} } t }   \nabla |D_x|^\frac13  d \rt\|_{Y_{m,\epsilon}} .  \alabel{eq:est of py d py^2 d 1}
		\end{align*}

		$\bullet$~\textbf{In ${\mathcal{R}_{\text{HL}}(k,l)}$.}
		By \eqref{eq:GN1} and \eqref{eq:trick2},
		we get
		\begin{align*}   
			&\quad\int_{\mathcal{R}_{\text{HL}}(k,l)}   e^{2 a A^{-\frac{1}{3}}|k|^{\frac{2}{3} } t }\langle k\rangle^{2 m   }\langle\frac{1}{k}\rangle^{2 \epsilon} |k|^\frac23 \left|\int_{\mathbb{R}} \mathcal{M}\left(k, D_y\right)   \q^3 d_k^1(y)  l \lt(k-l \rt) \q  d_l(y)  \cdot d_{k-l }(y)  d y\right| d k d l \\
			&\lesssim  \lt\|e^{a A^{-\frac{1}{3}}|k|^{\frac{2}{3} }t}\langle k\rangle^{m   }\langle\frac{1}{k}\rangle^\epsilon |k|^\frac13\lt\| \q^3 d_k^1\rt\|_{L_y^2}\rt\|_{L_k^2}     \lt\|e^{a A^{-\frac{1}{3}}|l|^{\frac{2}{3} } t } \langle l\rangle^{m   }\langle\frac{1}{l}\rangle^\epsilon |l|^\frac32 \| \q d_l \|_{L_y^{\infty}}\rt\|_{L_l^2} \\
			&\quad \times\lt\|e^{a A^{-\frac{1}{3}}|k-l|^{\frac{2}{3} } t }\langle k-l\rangle^{m   } \langle\frac{1}{k-l}\rangle^\epsilon   |k-l|^{\frac43}\lt \| d_{k-l}\rt\|_{L_y^2}  \rt\|_{L_{k-l}^2} 
			\lt\| |k-l|^{-\frac12} \langle k-l\rangle^{-m}\langle\frac{1}{k-l}\rangle^{-\epsilon}\rt\|_{L_{k-l}^2}\\
			&\lesssim  \lt \|e^{a A^{-\frac{1}{3}}\left|D_x\right|^{\frac{2}{3} } t }   \nabla \lt(\nabla^2 |D_x|^\frac13 d^1 \rt)\rt\|_{Y_{m,\epsilon}}\lt \|e^{a A^{-\frac{1}{3}}\left|D_x\right|^{\frac{2}{3} }t}    |D_x|^\frac13 \nabla^2 |D_x|^\frac13 d
            \rt\|_{Y_{m,\epsilon}}^\frac12       \\
			&\quad \times \lt\|e^{a A^{-\frac{1}{3}}\left|D_x\right|^{\frac{2}{3} }t}   \nabla \lt(\nabla^2 |D_x|^\frac13 d\rt)\rt\|_{Y_{m,\epsilon}}^\frac12\lt \|e^{a A^{-\frac{1}{3}}\left|D_x\right|^{\frac{2}{3} } t }   \nabla |D_x|^\frac13  d \rt\|_{Y_{m,\epsilon}}.  \alabel{eq:est of px d px py d 2}
		\end{align*}
		Similarly, we have
		\begin{align*}   
			&\quad\int_{\mathcal{R}_{\text{HL}}(k,l)}   e^{2 a A^{-\frac{1}{3}}|k|^{\frac{2}{3} } t }\langle k\rangle^{2 m   }\langle\frac{1}{k}\rangle^{2 \epsilon} |k|^\frac23 \left|\int_{\mathbb{R}} \mathcal{M}\left(k, D_y\right)   \q^3 d_k^1(y)  \q\py d_l(y)    \cdot \py d_{k-l}(y)  d y\right| d k d l \\
			&\lesssim  \lt\|e^{a A^{-\frac{1}{3}}|k|^{\frac{2}{3} }t}\langle k\rangle^{m   }\langle\frac{1}{k}\rangle^\epsilon |k|^\frac13 \lt\| \q^3 d_k^1\rt\|_{L_y^2}\rt\|_{L_k^2}     \lt\|e^{a A^{-\frac{1}{3}}|l|^{\frac{2}{3} } t } \langle l\rangle^{m   }\langle\frac{1}{l}\rangle^\epsilon |l|^\frac12 \lt\| \q\py d_l \rt\|_{L_y^{\infty}}\rt\|_{L_l^2} \\
			&\quad \times\lt\|e^{a A^{-\frac{1}{3}}|k-l|^{\frac{2}{3} } t }\langle k-l\rangle^{m   } \langle\frac{1}{k-l}\rangle^\epsilon    |k-l|^\frac13 \lt\|\py d_{k-l }\rt\|_{L_y^2} \rt\|_{L_{k-l}^2} 
			\lt\||k-l|^{-\frac12}\langle k-l\rangle^{-m}\langle\frac{1}{k-l}\rangle^{-\epsilon}\rt\|_{L_{k-l}^2}\\
			&\lesssim  
           \lt \|e^{a A^{-\frac{1}{3}}\left|D_x\right|^{\frac{2}{3} } t }   \nabla \lt(\nabla^2 |D_x|^\frac13 d^1 \rt)\rt\|_{Y_{m,\epsilon}}\lt \|e^{a A^{-\frac{1}{3}}\left|D_x\right|^{\frac{2}{3} }t}    |D_x|^\frac13 \nabla^2 |D_x|^\frac13 d\rt\|_{Y_{m,\epsilon}}^\frac12       \\
			&\quad \times \lt\|e^{a A^{-\frac{1}{3}}\left|D_x\right|^{\frac{2}{3} }t}   \nabla \lt(\nabla^2 |D_x|^\frac13 d\rt)\rt\|_{Y_{m,\epsilon}}^\frac12 \lt\|e^{a A^{-\frac{1}{3}}\left|D_x\right|^{\frac{2}{3} } t }   \nabla |D_x|^\frac13  d \rt\|_{Y_{m,\epsilon}}
            \alabel{eq:est of py d py^2 d 2}
		\end{align*}
		
		$\bullet$~\textbf{In ${\mathcal{R}_{\text{LH}}(k,l)}$.} 
By \eqref{eq:GN1}, \eqref{eq:trick3} and \eqref{eq:trick3.5} ($s_1=0$, $s=s_2=\frac12$, $s_3=\epsilon$),
		we get
		\begin{align*}   
			&\quad\int_{\mathcal{R}_{\text{LH}}(k,l)}   e^{2 a A^{-\frac{1}{3}}|k|^{\frac{2}{3} } t }\langle k\rangle^{2 m   }\langle\frac{1}{k}\rangle^{2 \epsilon} |k|^\frac23 \left|\int_{\mathbb{R}} \mathcal{M}\left(k, D_y\right)   \q^3 d_k^1(y)  l \lt(k-l \rt) \q  d_l(y)  \cdot d_{k-l }(y)  d y\right| d k d l \\
			&\lesssim  \lt\|e^{a A^{-\frac{1}{3}}|k|^{\frac{2}{3} }t}\langle k\rangle^{m   }\langle\frac{1}{k}\rangle^\epsilon |k|^\frac13 \lt\| \q^3 d_k^1\rt\|_{L_y^2}\rt\|_{L_k^2}     \lt\|e^{a A^{-\frac{1}{3}}|l|^{\frac{2}{3} } t } \langle l\rangle^{m   }\langle\frac{1}{l}\rangle^\epsilon |l|^\frac32 \| \q d_l \|_{L_y^{\infty}}\rt\|_{L_l^2} \\
			&\quad \times\lt\|e^{a A^{-\frac{1}{3}}|k-l|^{\frac{2}{3} } t }\langle k-l\rangle^{m   } \langle\frac{1}{k-l}\rangle^\epsilon   |k-l|^{\frac43}\lt \| d_{k-l}\rt\|_{L_y^2}  \rt\|_{L_{k-l}^2} 
			\lt\|  \langle k\rangle^{-m}\lt(\langle\frac{1}{k}\rangle^{\epsilon}+\langle\frac{1}{k}\rangle^{\frac12-\epsilon}\rt)\rt\|_{L_{k}^2},
		\end{align*}
        which is bounded by the last term  of \eqref{eq:est of px d px py d 2}.
		Similarly,
		\begin{align*}   
			&\quad\int_{\mathcal{R}_{\text{LH}}(k,l)}   e^{2 a A^{-\frac{1}{3}}|k|^{\frac{2}{3} } t }\langle k\rangle^{2 m   }\langle\frac{1}{k}\rangle^{2 \epsilon} |k|^\frac23 \left|\int_{\mathbb{R}} \mathcal{M}\left(k, D_y\right)   \q^3 d_k^1(y)  \q\py d_l(y)    \cdot \py d_{k-l}(y)  d y\right| d k d l \\
			&\lesssim  \lt\|e^{a A^{-\frac{1}{3}}|k|^{\frac{2}{3} }t}\langle k\rangle^{m   }\langle\frac{1}{k}\rangle^\epsilon |k|^\frac13 \lt\| \q^3 d_k^1\rt\|_{L_y^2}\rt\|_{L_k^2}     \lt\|e^{a A^{-\frac{1}{3}}|l|^{\frac{2}{3} } t } \langle l\rangle^{m   }\langle\frac{1}{l}\rangle^\epsilon |l|^\frac12 \lt\| \q\py d_l \rt\|_{L_y^{\infty}}\rt\|_{L_l^2} \\
			&\quad \times\lt\|e^{a A^{-\frac{1}{3}}|k-l|^{\frac{2}{3} } t }\langle k-l\rangle^{m   } \langle\frac{1}{k-l}\rangle^\epsilon    |k-l|^\frac13 \lt\|\py d_{k-l }\rt\|_{L_y^2} \rt\|_{L_{k-l}^2}  \lt\|\langle k\rangle^{-m}\lt(\langle\frac{1}{k}\rangle^{\epsilon}+\langle\frac{1}{k}\rangle^{\frac12-\epsilon}\rt)\rt\|_{L_{k}^2}, 
		\end{align*}
        which has the same upper bound as \eqref{eq:est of py d py^2 d 2}.
		
		Combining \eqref{eq:est of px d px py d 1}--\eqref{eq:est of py d py^2 d 2}, we arrive
		\begin{align*}
			&\quad \frac1A   \left|\Re\left( \p^2  \left( |\nabla d|^2 e_1\right) \left\lvert\, \mathcal{M} e^{2 a A^{-\frac{1}{3}}\left|D_x\right|^{\frac{2}{3} }t }\langle D_x\rangle^{2 m}\langle\frac{1}{D_x}\rangle^{2 \epsilon} \p^2 |D_x|^\frac23  d\right.\right)\right|\\
			&\lesssim \frac1A \lt(\lt\|e^{a A^{-\frac{1}{3}}\left|D_x\right|^{\frac{2}{3} } t }   \nabla \lt(\nabla^2|D_x|^\frac13 d^1\rt) \rt\|_{Y_{m,\epsilon}}\lt\|e^{a A^{-\frac{1}{3}}\left|D_x\right|^{\frac{2}{3} }t}     \nabla^2 |D_x|^\frac13 d\rt\|_{Y_{m,\epsilon}}\lt\|e^{a A^{-\frac{1}{3}}\left|D_x\right|^{\frac{2}{3} } t }   |D_x|^\frac13 \nabla|D_x|^{\frac13} d \rt\|_{Y_{m,\epsilon}}\rt.\\
			&\lt.\qquad+\lt\|e^{a A^{-\frac{1}{3}}\left|D_x\right|^{\frac{2}{3} } t }   \nabla \lt(\nabla^2 |D_x|^\frac13 d^1\rt )\rt\|_{Y_{m,\epsilon}} \lt\|e^{a A^{-\frac{1}{3}}\left|D_x\right|^{\frac{2}{3} }t}    |D_x|^\frac13 \nabla^2 |D_x|^\frac13 d\rt\|_{Y_{m,\epsilon}}^\frac12   \rt.    \\
			&\lt.\qqquad \times \lt\|e^{a A^{-\frac{1}{3}}\left|D_x\right|^{\frac{2}{3} }t}   \nabla \lt(\nabla^2 |D_x|^\frac13 d\rt)\rt\|_{Y_{m,\epsilon}}^\frac12 \lt\|e^{a A^{-\frac{1}{3}}\left|D_x\right|^{\frac{2}{3} } t }   \nabla |D_x|^\frac13  d \rt\|_{Y_{m,\epsilon}}\rt)\\
            &\lesssim \frac1{A^{\frac16}}\lt\|e^{a A^{-\frac{1}{3}}\left|D_x\right|^{\frac{2}{3} } t }  \nabla^2 |D_x|^\frac13 d^1\rt \|_{X_{a,m,\epsilon}}  \lt\|e^{a A^{-\frac{1}{3}}\left|D_x\right|^{\frac{2}{3} } t }  \nabla^2|D_x|^\frac13 d \rt\|_{X_{a,m,\epsilon}}  \lt\|e^{a A^{-\frac{1}{3}}\left|D_x\right|^{\frac{2}{3} } t }  \nabla|D_x|^\frac13 d \rt\|_{X_{a,m,\epsilon}}.
		\end{align*}
		The proof is complete.
	\end{proof}

	\begin{proof}[\bf Proof of Lemma \ref{lem:est of py^2 u}.]
		Consider the following decomposition: 
		\begin{align*} 
			& \quad   \frac1A   \left|\Re\left( \p^2  \left(  u \cdot \nabla d \right) \left\lvert\, \mathcal{M} e^{2 a A^{-\frac{1}{3}}\left|D_x\right|^{\frac{2}{3} }t }\langle D_x\rangle^{2 m}\langle\frac{1}{D_x}\rangle^{2 \epsilon} \p^2 |D_x|^\frac23  d\right.\right)\right| \\
			&\lesssim \frac1A \int_{\mathbb{R}^2} e^{2 a A^{-\frac{1}{3}}|k|^{\frac{2}{3} } t }\langle k\rangle^{2 m   }\langle\frac{1}{k}\rangle^{2 \epsilon} |k|^\frac23 \left|\int_{\mathbb{R}} \mathcal{M}\left(k, D_y\right)   (k-l) \q\py \Delta_l^{-1} \omega_l \q^3 d_k(y) \cdot d_{k-l}(y)    d y\right| d k d l \\
			&\quad + \frac1A \int_{\mathbb{R}^2} e^{2 a A^{-\frac{1}{3}}|k|^{\frac{2}{3} } t }\langle k\rangle^{2 m   }\langle\frac{1}{k}\rangle^{2 \epsilon} |k|^\frac23 \left|\int_{\mathbb{R}} \mathcal{M}\left(k, D_y\right)   (k-l) \py \Delta_l^{-1} \omega_l \q^3 d_k(y) \cdot \q d_{k-l}(y)    d y\right| d k d l \\
			&\quad + \frac1A \int_{\mathbb{R}^2} e^{2 a A^{-\frac{1}{3}}|k|^{\frac{2}{3} } t }\langle k\rangle^{2 m   }\langle\frac{1}{k}\rangle^{2 \epsilon} |k|^\frac23 \left|\int_{\mathbb{R}} \mathcal{M}\left(k, D_y\right)    \q (l \Delta_l^{-1} \omega_l) \q^3 d_k(y) \cdot\py  d_{k-l}(y)    d y\right| d k d l \\
			&\quad + \frac1A \int_{\mathbb{R}^2} e^{2 a A^{-\frac{1}{3}}|k|^{\frac{2}{3} } t }\langle k\rangle^{2 m   }\langle\frac{1}{k}\rangle^{2 \epsilon} |k|^\frac23 \left|\int_{\mathbb{R}} \mathcal{M}\left(k, D_y\right)   l  \Delta_l^{-1} \omega_l \q^3 d_k(y) \cdot \q\py d_{k-l}(y)    d y\right| d k d l\\
            &=: \frac1A \int_{\mathbb{R}^2}  \Phi_1(k,l)dkdl+\cdots+ \frac1A \int_{\mathbb{R}^2}  \Phi_4(k,l)dkdl.
		\end{align*}
        Next, we estimate each term in four steps.
		
		\underline{\textbf{Step I: Estimates of $\Phi_1(k,l)$.}}
		By  \eqref{eq:GN1} and \eqref{eq:trick1}, we have 
		\begin{align*}   
         \int_{\mathcal{R}_{\text{res}}(k,l)}\Phi_1(k,l)dkdl
			&\lesssim\lt\|e^{a A^{-\frac{1}{3}}|k|^{\frac{2}{3} }t}\langle k\rangle^{m  }\langle\frac{1}{k}\rangle^\epsilon  |k|^\frac13  \lt\| \q^3 d_k\rt\|_{L_y^2}\rt\|_{L_k^2}    \lt  \|e^{a A^{-\frac{1}{3}}|l|^{\frac{2}{3} } t } |l|^{-\frac13} \lt\|\q\py \Delta_l^{-1} \omega_l \rt\|_{L_y^{2}}\rt\|_{L_l^1} \\
			&\qquad \times\lt\|e^{a A^{-\frac{1}{3}}|k-l|^{\frac{2}{3} } t }\langle k-l\rangle^{m   }\langle\frac{1}{k-l}\rangle^\epsilon  |k-l|^\frac53 \lt\|   d_{k-l}\rt\|_{L_y^\infty}\rt\|_{L_{k-l}^2} \\
			&\lesssim  \lt\|e^{a A^{-\frac{1}{3}}\left|D_x\right|^{\frac{2}{3} } t }   \nabla \lt(\nabla^2|D_x|^\frac13 d\rt)\rt\|_{Y_{m,\epsilon}}\lt\|e^{a A^{-\frac{1}{3}}\left|D_x\right|^{\frac{2}{3} } t }     \omega\rt\|_{Y_{m,\epsilon}} \\
			&\quad \times  \lt\|e^{a A^{-\frac{1}{3}}\left|D_x\right|^{\frac{2}{3} } t }   |D_x|^\frac13 \nabla|D_x|^{\frac13}   d\rt\|_{Y_{m,\epsilon}}^\frac12 \lt\|e^{a A^{-\frac{1}{3}}\left|D_x\right|^{\frac{2}{3} } t }   |D_x|^\frac13 \nabla^2 |D_x|^\frac13   d\rt\|_{Y_{m,\epsilon}}^\frac12,\alabel{11}
		\end{align*}
       where we used H$\text{\"{o}}$lder inequality in the first inequality above to get
       \begin{align*}
           \lt  \|e^{a A^{-\frac{1}{3}}|l|^{\frac{2}{3} } t } |l|^{-\frac13} \lt\|\q\py \Delta_l^{-1} \omega_l \rt\|_{L_y^{2}}\rt\|_{L_l^1}
           &\lesssim\lt  \|e^{a A^{-\frac{1}{3}}|l|^{\frac{2}{3} } t } \langle l\rangle^m\langle \frac1l\rangle^\epsilon\lt\|\q\py \Delta_l^{-1} \omega_l \rt\|_{L_y^{2}}\rt\|_{L_l^2}\lt  \| |l|^{-\frac13}\langle l\rangle^{-m}\langle \frac1l\rangle^{-\epsilon}\rt\|_{L_l^2}\\
          & \lesssim \lt\|e^{a A^{-\frac{1}{3}}\left|D_x\right|^{\frac{2}{3} } t }     \omega\rt\|_{Y_{m,\epsilon}}
       \end{align*}
       for $\epsilon>0.$
		 Moreover, \eqref{eq:trick2} implies
		\begin{align*}   
			&\quad \int_{\mathcal{R}_{\text{HL}}(k,l)}\Phi_1(k,l)dkdl\\ & \lesssim\lt\|e^{a A^{-\frac{1}{3}}|k|^{\frac{2}{3} } t }\langle k\rangle^{m  }\langle\frac{1}{k}\rangle^\epsilon  |k|^\frac13 \lt\| \q^3 d_k\rt\|_{L_y^2}\rt\|_{L_k^2}           \lt  \|e^{a A^{-\frac{1}{3}}|l|^{\frac{2}{3} } t }\langle l\rangle^{m  }\langle\frac{1}{l}\rangle^\epsilon|l|^{\frac13}\lt\| \q\py \Delta_l^{-1} \omega_l  \rt\|_{L_y^{2}}\rt\|_{L_l^2} \\
            & \quad \times\lt\|\langle k-l\rangle^{-m  }\langle\frac{1}{k-l}\rangle^{-\epsilon}   |k-l|^{-\frac13}\rt\|_{L_{k-l}^2} \lt\|e^{a A^{-\frac{1}{3}}|k-l|^{\frac{2}{3} } t }\langle k-l\rangle^{m  }\langle\frac{1}{k-l}\rangle^\epsilon   |k-l|^\frac43  \|   d_{k-l}\|_{L_y^\infty}\rt\|_{L_{k-l}^2} 
            \end{align*}
            and \eqref{eq:trick3}-\eqref{eq:trick3.5} imply
            \begin{align*}    
            &\quad \int_{\mathcal{R}_{\text{LH}}(k,l)}\Phi_1(k,l)dkdl\\ 
			& \lesssim\lt\|e^{a A^{-\frac{1}{3}}|k|^{\frac{2}{3} } t }\langle k\rangle^m\langle\frac{1}{k}\rangle^\epsilon  |k|^\frac13  \| \q^3 d_k\|_{L_y^2}\rt\|_{L_k^2}  \lt\|e^{a A^{-\frac{1}{3}}|l|^{\frac{2}{3} } t }\langle l\rangle^{ m  }\langle\frac{1}{l}\rangle^\epsilon|l|^\frac13 \|\q\py \Delta_l^{-1} \omega_l \|_{L_y^{2}}\rt\|_{L_l^2} \\
			& \quad \times\lt\|\left(\langle\frac{1}{k}\rangle^\epsilon+\langle\frac{1}{k}\rangle^{\frac{1}{3}-\epsilon}\right)\langle k\rangle^{-m}\rt\|_{L_k^2} \lt\|e^{a A^{-\frac{1}{3}}|k-l|^{\frac{2}{3} } t }\langle k-l\rangle^{m  }\langle\frac{1}{k-l}\rangle^\epsilon   |k-l|^\frac43  \|   d_{k-l}\|_{L_y^\infty}\rt\|_{L_{k-l}^2},
		\end{align*}
       where   both terms can be controlled by
        \begin{align*} 
			& C\lt\|e^{a A^{-\frac{1}{3}}\left|D_x\right|^{\frac{2}{3} } t }    \nabla \lt(\nabla^2 |D_x|^\frac13 d \rt)\rt\|_{Y_{m,\epsilon}}  \lt\|e^{a A^{-\frac{1}{3}}\left|D_x\right|^{\frac{2}{3} } t }   |D_x|^\frac13  \omega\rt\|_{Y_{m,\epsilon}}   \\
			& \times  \lt\|e^{a A^{-\frac{1}{3}}\left|D_x\right|^{\frac{2}{3} } t }     \nabla  |D_x|^{\frac13} d\rt\|_{Y_{m,\epsilon}}^\frac12 \lt\|e^{a A^{-\frac{1}{3}}\left|D_x\right|^{\frac{2}{3} } t }     \nabla^2 |D_x|^\frac13   d\rt\|_{Y_{m,\epsilon}}^\frac12 \alabel{12}
		\end{align*}
		due to \eqref{eq:GN1}.	
		
			\underline{\textbf{Step II: Estimates of $\Phi_2(k,l)$.}}
		Using  \eqref{eq:trick1}, we get
		\begin{align*}   
			\int_{\mathcal{R}_{\text{res}}(k,l)}\Phi_2(k,l) d k d l
			&\lesssim\lt\|e^{a A^{-\frac{1}{3}}|k|^{\frac{2}{3} }t}\langle k\rangle^{m  }\langle\frac{1}{k}\rangle^\epsilon  |k|^\frac13 \lt \| \q^3 d_k\rt\|_{L_y^2}\rt\|_{L_k^2}      \lt\|e^{a A^{-\frac{1}{3}}|l|^{\frac{2}{3} } t } |l|^{-\frac16}  \lt\|\py \Delta_l^{-1} \omega_l \rt\|_{L_y^{\infty}}\rt\|_{L_l^1} \\
			&\quad \times\lt\|e^{a A^{-\frac{1}{3}}|k-l|^{\frac{2}{3} } t }\langle k-l\rangle^{m   }\langle\frac{1}{k-l}\rangle^\epsilon  |k-l|^\frac32 \lt\|  \q d_{k-l}\rt\|_{L_y^2}\rt\|_{L_{k-l}^2} \\
			&\lesssim  \lt\|e^{a A^{-\frac{1}{3}}\left|D_x\right|^{\frac{2}{3} } t }   \nabla\lt( \nabla^2 |D_x|^\frac13 d\rt)\rt\|_{Y_{m,\epsilon}}\lt\|e^{a A^{-\frac{1}{3}}\left|D_x\right|^{\frac{2}{3} } t }     \omega\rt\|_{Y_{m,\epsilon}} \\
			&\quad \times\lt   \|e^{a A^{-\frac{1}{3}}\left|D_x\right|^{\frac{2}{3} } t }   |D_x|^\frac13 \nabla^2 |D_x|^\frac13   d\rt\|_{Y_{m,\epsilon}} ^{\frac56}\lt\|e^{a A^{-\frac{1}{3}}\left|D_x\right|^{\frac{2}{3} } t }   |D_x|^\frac13 \nabla |D_x|^\frac13   \rt\|_{Y_{m,\epsilon}}^{\frac16},\alabel{21}
		\end{align*}
         where we used H$\text{\"{o}}$lder inequality and \eqref{eq:GN2} to get
        \begin{align*}
            \lt\|e^{a A^{-\frac{1}{3}}|l|^{\frac{2}{3} } t }  |l|^{-\frac16}  \|\py \Delta_{l}^{-1}\omega_l\|_{L_y^{\infty}}\rt\|_{L_l^1}
            &\lesssim\lt\|e^{a A^{-\frac{1}{3}}|l|^{\frac{2}{3} } t }  |l|^{-\frac23}  \lt\|\nabla_l\py \Delta_{l}^{-1}\omega_l \rt\|_{L_y^{2}}\rt\|_{L_l^1}\\
            &\lesssim \lt\|e^{a A^{-\frac{1}{3}}|l|^{\frac{2}{3} } t } \langle l \rangle^{m} \langle \frac{1}{l} \rangle^{\epsilon} \lt\|\nabla_l\py \Delta_{l}^{-1}\omega_l\rt\|_{L_y^{2}}\rt\|_{L_l^2}\lt\||l|^{-\frac23}\langle l \rangle^{-m} \langle \frac{1}{l} \rangle^{-\epsilon}\rt\|_{L^2_l}\\
            &\lesssim \lt\|e^{a A^{-\frac{1}{3}}\left|D_x\right|^{\frac{2}{3} }t}  \omega \rt\|_{Y_{m,\epsilon}}
        \end{align*}
        for $\epsilon>\frac16.$
		Using \eqref{eq:trick2}--\eqref{eq:trick3.5} again, we have
		\begin{align*}   
			&\quad\int_{\mathcal{R}_{\text{HL}}(k,l)} \Phi_2(k,l) d k d l \\
		 &\lesssim\lt\|e^{a A^{-\frac{1}{3}}|k|^{\frac{2}{3} } t }\langle k\rangle^{m  }\langle\frac{1}{k}\rangle^\epsilon  |k|^\frac13 \lt\| \q^3 d_k\rt\|_{L_y^2}\rt\|_{L_k^2}            \lt \|e^{a A^{-\frac{1}{3}}|l|^{\frac{2}{3} } t }\langle l\rangle^{m  }\langle\frac{1}{l}\rangle^\epsilon|l|^{\frac43} \lt\| \py \Delta_l^{-1} \omega_l  \rt\|_{L_y^{2}}\rt\|_{L_l^2} \\
			& \quad \times\lt\||k-l|^{-\frac13}\langle k-l\rangle^{-m}\langle \frac{1}{k-l}\rangle^{-\epsilon}\rt\|_{L_{k-l}^2}\lt\|e^{a A^{-\frac{1}{3}}|k-l|^{\frac{2}{3} } t }\langle k-l\rangle^{m  }\langle\frac{1}{k-l}\rangle^\epsilon   |k-l|^\frac13  \| \q  d_{k-l}\|_{L_y^\infty}\rt\|_{L_{k-l}^2},
		\end{align*}
        		and
		\begin{align*}    
			&\quad\int_{\mathcal{R}_{\text{LH}}(k,l)} \Phi_2(k,l) d k d l \\
			& \lesssim\lt\|e^{a A^{-\frac{1}{3}}|k|^{\frac{2}{3} } t }\langle k\rangle^m\langle\frac{1}{k}\rangle^\epsilon  |k|^\frac13  \| \q^3 d_k\|_{L_y^2}\rt\|_{L_k^2}  
            \lt\|e^{a A^{-\frac{1}{3}}|l|^{\frac{2}{3} } t }\langle l\rangle^{ m  }\langle\frac{1}{l}\rangle^\epsilon|l|^\frac43 \|\py \Delta_l^{-1} \omega_l \|_{L_y^{2}}\rt\|_{L_l^2} \\
			& \quad \times\lt\|\left(\langle\frac{1}{k}\rangle^\epsilon+\langle\frac{1}{k}\rangle^{\frac{1}{3}-\epsilon}\right)\langle k\rangle^{-m}\rt\|_{L_k^2}\lt\|e^{a A^{-\frac{1}{3}}|k-l|^{\frac{2}{3} } t }\langle k-l\rangle^{m  }\langle\frac{1}{k-l}\rangle^\epsilon   |k-l|^\frac13  \| \q  d_{k-l}\|_{L_y^\infty}\rt\|_{L_{k-l}^2}.
		\end{align*}
      By \eqref{eq:GN1}, both terms can be controlled by \eqref{12}.
            
		\underline{\textbf{Step III: Estimates of $\Phi_3(k,l)$.}}
		From \eqref{eq:GN1} and \eqref{eq:trick1}, we infer
		\begin{align*}   
			\int_{\mathcal{R}_{\text{res}}(k,l)} \Phi_3(k,l)d k d l
			&\lesssim\lt\|e^{a A^{-\frac{1}{3}}|k|^{\frac{2}{3} }t}\langle k\rangle^{m  }\langle\frac{1}{k}\rangle^\epsilon  |k|^\frac13 \lt \| \q^3 d_k\rt\|_{L_y^2}\rt\|_{L_k^2}      \lt\|e^{a A^{-\frac{1}{3}}|l|^{\frac{2}{3} } t }  |l| \lt\|\q \Delta_l^{-1} \omega_l \rt\|_{L_y^{2}}\rt\|_{L_l^1} \\
			&\quad \times\lt\|e^{a A^{-\frac{1}{3}}|k-l|^{\frac{2}{3} } t }\langle k-l\rangle^{m   }\langle\frac{1}{k-l}\rangle^\epsilon  |k-l|^\frac13\lt \|  \py d_{k-l}\rt\|_{L_y^\infty}\rt\|_{L_{k-l}^2} \\
			&\lesssim  \lt\|e^{a A^{-\frac{1}{3}}\left|D_x\right|^{\frac{2}{3} } t }   \nabla \lt(\nabla^2 |D_x|^\frac13 d\rt)\rt\|_{Y_{m,\epsilon}}\lt\|e^{a A^{-\frac{1}{3}}\left|D_x\right|^{\frac{2}{3} } t }   \px \nabla \Delta^{-1}   \omega\rt\|_{Y_{m,\epsilon}} \\
			&\quad \times \lt\|e^{a A^{-\frac{1}{3}}\left|D_x\right|^{\frac{2}{3} } t }   \nabla|D_x|^\frac13   d\rt\|_{Y_{m,\epsilon}}^\frac12 \lt\|e^{a A^{-\frac{1}{3}}\left|D_x\right|^{\frac{2}{3} } t }     \nabla^2 |D_x|^\frac13   d\rt\|_{Y_{m,\epsilon}}^\frac12 .\alabel{31}
		\end{align*}
		By \eqref{eq:trick2}--\eqref{eq:trick3.5}, we have
		\begin{align*}   
			&\quad\int_{\mathcal{R}_{\text{HL}}(k,l)} \Phi_3(k,l)d k d l \\&\lesssim\lt\|e^{a A^{-\frac{1}{3}}|k|^{\frac{2}{3} } t }\langle k\rangle^{m  }\langle\frac{1}{k}\rangle^\epsilon  |k|^\frac13 \lt\| \q^3 d_k\rt\|_{L_y^2}\rt\|_{L_k^2}            \lt \|e^{a A^{-\frac{1}{3}}|l|^{\frac{2}{3} } t }\langle l\rangle^{m  }\langle\frac{1}{l}\rangle^\epsilon|l|^{\frac43}\lt\| \q \Delta_l^{-1} \omega_l  \rt\|_{L_y^{2}}\rt\|_{L_l^2} \\
			& \quad \times\lt\||k-l|^{-\frac13}\langle k-l\rangle^{-m}\langle \frac{1}{k-l}\rangle^{-\epsilon}\rt\|_{L_{k-l}^2}\lt\|e^{a A^{-\frac{1}{3}}|k-l|^{\frac{2}{3} } t }\langle k-l\rangle^{m  }\langle\frac{1}{k-l}\rangle^\epsilon   |k-l|^\frac13 \|   \partial_yd_{k-l}\|_{L_y^\infty}\rt\|_{L_{k-l}^2} ,
		\end{align*}
        and
		\begin{align*}    
			&\quad\int_{\mathcal{R}_{\text{LH}}(k,l)} \Phi_3(k,l) d k d l\\
			& \lesssim\lt\|e^{a A^{-\frac{1}{3}}|k|^{\frac{2}{3} } t }\langle k\rangle^m\langle\frac{1}{k}\rangle^\epsilon  |k|^\frac13  \| \q^3 d_k\|_{L_y^2}\rt\|_{L_k^2}  \lt\|e^{a A^{-\frac{1}{3}}|l|^{\frac{2}{3} } t }\langle l\rangle^{ m  }\langle\frac{1}{l}\rangle^\epsilon|l|^{\frac43}  \|\q \Delta_l^{-1} \omega_l \|_{L_y^{2}}\rt\|_{L_l^2} \\
			& \quad \times\lt\|\left(\langle\frac{1}{k}\rangle^\epsilon+\langle\frac{1}{k}\rangle^{ \frac13-\epsilon}\right)\langle k\rangle^{-m}\rt\|_{L_k^2} \lt\|e^{a A^{-\frac{1}{3}}|k-l|^{\frac{2}{3} } t }\langle k-l\rangle^{m  }\langle\frac{1}{k-l}\rangle^\epsilon   |k-l|^\frac13  \|   \partial_yd_{k-l}\|_{L_y^\infty}\rt\|_{L_{k-l}^2}.
		\end{align*}
        Likewise, these two terms are bounded by \eqref{12}.
			
		\underline{\textbf{Step IV: Estimates of $\Phi_4(k,l)$.}}
        On one hand, by \eqref{eq:trick1}
		we have
		\begin{align*}   
			\int_{\mathcal{R}_{\text{res}}(k,l)} \Phi_4(k,l)d k d l&\lesssim\lt\|e^{a A^{-\frac{1}{3}}|k|^{\frac{2}{3} }t}\langle k\rangle^{m  }\langle\frac{1}{k}\rangle^\epsilon  |k|^\frac13 \lt \| \q^3 d_k\rt\|_{L_y^2}\rt\|_{L_k^2}      \lt\|e^{a A^{-\frac{1}{3}}|l|^{\frac{2}{3} } t } |l|  \lt \|  \Delta_l^{-1} \omega_l \rt\|_{L_y^{\infty}}\rt\|_{L_l^1} 
            \\
			&\quad \times\lt\|e^{a A^{-\frac{1}{3}}|k-l|^{\frac{2}{3} } t }\langle k-l\rangle^{m   }\langle\frac{1}{k-l}\rangle^\epsilon  |k-l|^\frac13 \lt\|  \q\py d_{k-l}\rt\|_{L_y^2}\rt\|_{L_{k-l}^2} 
            .\alabel{41}
		\end{align*}

		On the other hand, by \eqref{eq:trick2}--\eqref{eq:trick3.5} we have 
		\begin{align*}   
			&\quad\int_{\mathcal{R}_{\text{HL}}(k,l)}\Phi_4(k,l) d k d l\\
			& \lesssim\lt\|e^{a A^{-\frac{1}{3}}|k|^{\frac{2}{3} } t }\langle k\rangle^{m  }\langle\frac{1}{k}\rangle^\epsilon  |k|^\frac13 \lt\| \q^3 d_k\rt\|_{L_y^2}\rt\|_{L_k^2}            \lt \|e^{a A^{-\frac{1}{3}}|l|^{\frac{2}{3} } t }\langle l\rangle^{m  }\langle\frac{1}{l}\rangle^\epsilon|l|^{\frac32 } \lt\|   \Delta_l^{-1} \omega_l  \rt\|_{L_y^{\infty}}\rt\|_{L_l^2} \\
			& \quad \times\lt\||k-l|^{-\frac12}\langle k-l\rangle^{-m}\langle \frac{1}{k-l}\rangle^{-\epsilon}\rt\|_{L_{k-l}^2} \lt\|e^{a A^{-\frac{1}{3}}|k-l|^{\frac{2}{3} } t }\langle k-l\rangle^{m  }\langle\frac{1}{k-l}\rangle^\epsilon   |k-l|^\frac13  \| \py\q  d_{k-l}\|_{L_y^2}\rt\|_{L_{k-l}^2},
		\end{align*}
		and
		\begin{align*}    
			&\quad\int_{\mathcal{R}_{\text{LH}}(k,l)} \Phi_4(k,l) d k d l\\
			& \lesssim\lt\|e^{a A^{-\frac{1}{3}}|k|^{\frac{2}{3} } t }\langle k\rangle^m\langle\frac{1}{k}\rangle^\epsilon  |k|^\frac13  \| \q^3 d_k\|_{L_y^2}\rt\|_{L_k^2}  \lt\|e^{a A^{-\frac{1}{3}}|l|^{\frac{2}{3} } t }\langle l\rangle^{ m  }\langle\frac{1}{l}\rangle^\epsilon|l|^\frac32 \| \Delta_l^{-1} \omega_l \|_{L_y^{\infty}}\rt\|_{L_l^2} \\
			& \quad \times\lt\|\left(\langle\frac{1}{k}\rangle^\epsilon+\langle\frac{1}{k}\rangle^{\frac{1}{2}-\epsilon}\right)\langle k\rangle^{-m}\rt\|_{L_k^2} \lt\|e^{a A^{-\frac{1}{3}}|k-l|^{\frac{2}{3} } t }\langle k-l\rangle^{m  }\langle\frac{1}{k-l}\rangle^\epsilon   |k-l|^\frac13  \lt\| \py\q  d_{k-l}\rt\|_{L_y^2}\rt\|_{L_{k-l}^2}.
		\end{align*}

        Hence, it follows from
        \begin{align*}
        \lt\|e^{a A^{-\frac{1}{3}}|l|^{\frac{2}{3} } t } |l|  \lt \|  \Delta_l^{-1} \omega_l \rt\|_{L_y^{\infty}}\rt\|_{L_l^1}
& \lesssim \lt\|e^{a A^{-\frac{1}{3}}|l|^{\frac{2}{3} } t }   \langle l\rangle^m\langle \frac1l\rangle^{\epsilon} l \lt\|  \nabla_l\Delta_l^{-1} \omega_l \rt\|_{L_y^{2}}\rt\|_{L_l^2}\lt\||l|^{-\frac12}\langle l\rangle^{-m}\langle \frac1l\rangle^{-\epsilon}\rt\|_{L^2_l}\\
 &\lesssim   \lt\|e^{a A^{-\frac{1}{3}}\left|D_x\right|^{\frac{2}{3} } t }    \px \nabla \Delta^{-1} \omega\rt\|_{Y_{m,\epsilon}}   
        \end{align*}
        and \eqref{eq:GN2} that
        \begin{align*}
            \int_{\mathcal{R}^2 } \Phi_4(k,l) d k d l 
            &\lesssim \lt\|e^{a A^{-\frac{1}{3}}\left|D_x\right|^{\frac{2}{3} } t }    \nabla \lt(\nabla^2|D_x|^\frac13 d\rt) \rt\|_{Y_{m,\epsilon}}  \lt\|e^{a A^{-\frac{1}{3}}\left|D_x\right|^{\frac{2}{3} } t }   \px\nabla\Delta^{-1}  \omega\rt\|_{Y_{m,\epsilon}}   \\
			&  \times    \lt\|e^{a A^{-\frac{1}{3}}\left|D_x\right|^{\frac{2}{3} } t }     \nabla^2 |D_x|^\frac13   d\rt\|_{Y_{m,\epsilon}}  .\alabel{42}
        \end{align*}


\underline{\textbf{Step V: Collecting.}}
		Collecting \eqref{11}--\eqref{42}, 
        and using the definition of $X_{a,m,\epsilon}$ in \eqref{eq:X norm}, we get
		\begin{align*}
			&\quad \frac1A  \left| \Re\left( \p^2   \lt( u \cdot \nabla d\rt) \left\lvert\, \mathcal{M} e^{2 a A^{-\frac{1}{3}}\left|D_x\right|^{\frac{2}{3} } t}\langle D_x\rangle^{2 m}\langle\frac{1}{D_x}\rangle^{2 \epsilon} \p^2  |D_x|^\frac23 d\right.\right)  \right|dt \\
			&\lesssim \frac{1}{A^\frac13} \lt\|      \nabla^2|D_x|^\frac13 d \rt\|_{X_{a,m,\epsilon}}\lt \|    \omega\rt \|_{X_{a,m,\epsilon}}\lt(\lt\|      \nabla^2|D_x|^\frac13 d \rt\|_{X_{a,m,\epsilon}} +\lt\|     \nabla|D_x|^\frac13 d \rt\|_{X_{a,m,\epsilon}} \rt) .
		\end{align*}
		The proof is complete.
	\end{proof}

    \section{The energy estimate for the vorticity $\omega$} \label{sec.5}
    In this section, we give energy estimates for the velocity (vorticity, see Proposition \ref{lem:est of omega}), and then provide the proof.
    
	\begin{Prop}\label{lem:est of omega}
		For $0<a<\frac{1}{16(1+2 \pi)}$ and $0<\epsilon< \frac12 <m$, we have
		\begin{align*}
			\|         \omega  \|_{X_{a,m,\epsilon}}^2 
			&\leq   
			C\|        \omega_{\mathrm{in}} \|_{Y_{m,\epsilon}}^2 + \frac{C}{A^\frac12}\|     \omega \|_{X_{a,m,\epsilon}}^3   \\
			&\quad +  \frac{C}{A^\frac13} \|       \omega \|_{X_{a,m,\epsilon}} \|     \nabla|D_x|^{\frac13} d \|_{X_{a,m,\epsilon}}^{\frac12} \lt \|     \lt(\px^2,\py^2\rt)|D_x|^{\frac13} d \rt\|_{X_{a,m,\epsilon}} \\
            &\qquad\times     \lt(\lt\|\nabla|D_x|^{\frac13} d \rt\|_{X_{a,m,\epsilon}}^{\frac12} +\lt\|     \lt(\px^2,\py^2\rt)|D_x|^{\frac13} d \rt\|_{X_{a,m,\epsilon}}^{\frac12}\rt).
		\end{align*}
	\end{Prop}
	\begin{proof}
		Recall that 
		\begin{align} \label{eq:omega}
			\partial_t \omega+y \partial_x \omega-\frac{1}{A} \Delta \omega=-\frac{1}{A}\big[ u \cdot \nabla \omega   - \px \left(  \py d\cdot\Delta d\right) +\py\left(  \px d\cdot\Delta d\right) \big].
		\end{align}
		One applies Proposition \ref{prop:est of f} to \eqref{eq:omega} and has
		\begin{equation}\begin{aligned} \label{eq:est of omega}
				\|   \omega  \|_{X_{a,m,\epsilon}}^2 
				&\lesssim    
				\| \omega_{\mathrm{in}} \|_{Y_{m,\epsilon}}^2 + \frac1{A^\frac12}\|    \omega \|_{X_{a,m,\epsilon}}^3   \\
				&\quad + \frac1A\int_{0}^{t} \left| \Re\left( \px \left(  \py d\cdot\Delta d\right) \left\lvert\, \mathcal{M} e^{2 a A^{-\frac{1}{3}}\left|D_x\right|^{\frac{2}{3} }t }\langle D_x\rangle^{2 m}\langle\frac{1}{D_x}\rangle^{2 \epsilon} \omega \right.\right) \right| dt \\
				&\quad + \frac1A\int_{0}^{t} \left| \Re\left( \py\left(  \px d\cdot\Delta d\right) \left\lvert\, \mathcal{M} e^{2 a A^{-\frac{1}{3}}\left|D_x\right|^{\frac{2}{3} }t }\langle D_x\rangle^{2 m}\langle\frac{1}{D_x}\rangle^{2 \epsilon} \omega \right.\right) \right| dt .
		\end{aligned}\end{equation}
		Then we complete the proof by estimating the remaining terms in three steps.
		
		\underline{\textbf{Step I: Estimates of $\px \left(  \py d\cdot\Delta d\right)$.}}
		Thanks to the Fourier transform, we have
		\begin{align*}
			&  \quad\left|   \Re\left( \px \left(  \py d\cdot\Delta d\right) \left\lvert\, \mathcal{M} e^{2 a A^{-\frac{1}{3}}\left|D_x\right|^{\frac{2}{3} } t}\langle D_x\rangle^{2 m}\langle\frac{1}{D_x}\rangle^{2 \epsilon} \omega\right.\right) \right|  \\
            &\lesssim  \int_{\mathbb{R}^2} e^{2 a A^{-\frac{1}{3}}|k|^{\frac{2}{3} } t }\langle k\rangle^{2 m   }\langle\frac{1}{k}\rangle^{2 \epsilon}  |k| \left|\int_{\mathbb{R}} \mathcal{M}\left(k, D_y\right)     \omega_k(y)  \py d_l(y)    \cdot \Delta_{k-l} d_{k-l}(y)  d y\right| d k d l
		\end{align*}
		
		$\bullet$~\textbf{In $\mathcal{R}_{\text{res}}(k,l)$.} By \eqref{eq:trick1} and Corollary \ref{cor:L infty} ($\alpha=\beta=\frac13$), we have
		\begin{align*}  
			&\quad \int_{\mathcal{R}_{\text{res}}(k,l)} e^{2 a A^{-\frac{1}{3}}|k|^{\frac{2}{3} } t }\langle k\rangle^{2 m   }\langle\frac{1}{k}\rangle^{2 \epsilon}  |k| \left|\int_{\mathbb{R}} \mathcal{M}\left(k, D_y\right)    \omega_k(y) \py  d_l(y)    \cdot \Delta_{k-l} d_{k-l}(y)  d y\right| d k d l \\
			&\lesssim \lt\|e^{a A^{-\frac{1}{3}}|k|^{\frac{2}{3} }t}\langle k\rangle^{m  }\langle\frac{1}{k}\rangle^\epsilon  |k|^\frac13  \lt\| \omega_k \rt\|_{L_y^2}\rt\|_{L_k^2}    \lt \|e^{a A^{-\frac{1}{3}}|l|^{\frac{2}{3} } t }   \lt\|  \py d_l \rt\|_{L_y^{\infty}}\rt\|_{L_l^1} \\
			&\quad \times\lt\|e^{a A^{-\frac{1}{3}}|k-l|^{\frac{2}{3} } t }\langle k-l\rangle^{m   } \langle\frac{1}{k-l}\rangle^\epsilon   |k-l|^\frac23   \lt\| \Delta_{k-l} d_{k-l}\rt\|_{L_y^2}   \rt\|_{L_{k-l}^2} \\
			&\lesssim \lt\|e^{a A^{-\frac{1}{3}}\left|D_x\right|^{\frac{2}{3} } t }   |D_x|^\frac13 \omega\rt \|_{Y_{m,\epsilon}} \lt\|e^{a A^{-\frac{1}{3}}\left|D_x\right|^{\frac{2}{3} }t}      \py |D_x|^\frac13   d\rt\|_{Y_{m,\epsilon}}^\frac12       \\
			&\quad \times \lt\|e^{a A^{-\frac{1}{3}}\left|D_x\right|^{\frac{2}{3} }t}    \py^2  |D_x|^\frac13 d\rt\|_{Y_{m,\epsilon}}^\frac12 \lt\|e^{a A^{-\frac{1}{3}}\left|D_x\right|^{\frac{2}{3} } t }   |D_x|^\frac13 \nabla^2 |D_x|^\frac13  d \rt\|_{Y_{m,\epsilon}} .  \alabel{eq:est of py d py^2 d 11}
		\end{align*}
		
    		$\bullet$~\textbf{In $\mathcal{R}_{\text{HL}}(k,l)$.}
		By \eqref{eq:GN1} and \eqref{eq:trick2},
		we get
		\begin{align*}   
			&\quad\int_{\mathcal{R}_{\text{HL}}(k,l)}   e^{2 a A^{-\frac{1}{3}}|k|^{\frac{2}{3} } t }\langle k\rangle^{2 m   }\langle\frac{1}{k}\rangle^{2 \epsilon} |k|   \left|\int_{\mathbb{R}} \mathcal{M}\left(k, D_y\right)   \omega_k(y)  \py  d_l(y)    \cdot \Delta_{k-l} d_{k-l}(y)  d y\right| d k d l \\
			&\lesssim \lt \|e^{a A^{-\frac{1}{3}}|k|^{\frac{2}{3} }t}\langle k\rangle^{m   }\langle\frac{1}{k}\rangle^\epsilon  |k|^\frac13  \lt\|   \omega_k\rt\|_{L_y^2}\rt\|_{L_k^2}     \lt\|e^{a A^{-\frac{1}{3}}|l|^{\frac{2}{3} } t } \langle l\rangle^{m   }\langle\frac{1}{l}\rangle^\epsilon |l|^\frac23 \lt\| \py   d_l \rt\|_{L_y^{\infty}}\rt\|_{L_l^2} \\
			&\quad \times\lt\|e^{a A^{-\frac{1}{3}}|k-l|^{\frac{2}{3} } t }\langle k-l\rangle^{m   } \langle\frac{1}{k-l}\rangle^\epsilon    |k-l|^\frac13 \lt\|\Delta_{k-l} d_{k-l }\rt\|_{L_y^2}\rt \|_{L_{k-l}^2} 
			\lt\||k-l|^{-\frac13}\langle k-l\rangle^{-m}\langle\frac{1}{k-l}\rangle^{-\epsilon}\rt\|_{L_{k-l}^2}\\
			&\lesssim  \lt\|e^{a A^{-\frac{1}{3}}\left|D_x\right|^{\frac{2}{3} } t }    |D_x|^\frac13 \omega \rt\|_{Y_{m,\epsilon}} \lt\|e^{a A^{-\frac{1}{3}}\left|D_x\right|^{\frac{2}{3} }t}    |D_x|^\frac13  \py  |D_x|^\frac13  d\rt\|_{Y_{m,\epsilon}}^\frac12       \\
			&\quad \times \lt\|e^{a A^{-\frac{1}{3}}\left|D_x\right|^{\frac{2}{3} }t}    |D_x|^\frac13  \py^2  |D_x|^\frac13  d\rt\|_{Y_{m,\epsilon}}^\frac12 \lt\|e^{a A^{-\frac{1}{3}}\left|D_x\right|^{\frac{2}{3} } t }   \nabla^2 |D_x|^\frac13  d\rt \|_{Y_{m,\epsilon}} .  \alabel{eq:est of py d py^2 d 12} 
		\end{align*}
		
		$\bullet$~\textbf{In $\mathcal{R}_{\text{LH}}(k,l)$.}
		The inequalities \eqref{eq:GN1}, \eqref{eq:trick3} and  \eqref{eq:trick3.5} imply
		\begin{align*}  
			&\quad\int_{\mathcal{R}_{\text{LH}}(k,l)} e^{2 a A^{-\frac{1}{3}}|k|^{\frac{2}{3} } t }\langle k\rangle^{2 m   }\langle\frac{1}{k}\rangle^{2 \epsilon}  |k| \left|\int_{\mathbb{R}} \mathcal{M}\left(k, D_y\right)   \omega_k(y)  \py d_l(y)    \cdot \Delta_{k-l} d_{k-l}(y)  d y\right| d k d l \\
			& \lesssim \lt\|e^{a A^{-\frac{1}{3}}|k|^{\frac{2}{3} } t }\langle k\rangle^{m }\langle\frac{1}{k}\rangle^\epsilon |k|^\frac13  \| \omega_k\|_{L_y^2}\rt\|_{L_k^2}  \lt\|e^{a A^{-\frac{1}{3}}|l|^{\frac{2}{3} } t }\langle l\rangle^{m   }\langle\frac{1}{l}\rangle^\epsilon  |l|^{\frac23}  \|  \py d_l\|_{L_y^\infty}\rt\|_{L_l^2} \\
			& \quad \times  \lt \|e^{a A^{-\frac{1}{3}}|k-l|^{\frac{2}{3} } t }\langle k-l\rangle^{m   } \langle\frac{1}{k-l}\rangle^\epsilon    |k-l|^\frac13 \| \Delta_{k-l} d_{k-l }\|_{L_y^2}  \rt\|_{L_{k-l}^2} 
			\lt\|  \left(\langle\frac{1}{k}\rangle^\epsilon  +\langle\frac{1}{k}\rangle^{\frac13 -\epsilon}  \right)\langle k\rangle^{-m}\rt\|_{L_k^2},
		\end{align*}
		which can be controlled by the same bound as \eqref{eq:est of py d py^2 d 12}.

		\underline{\textbf{Step II: Estimates of $\py\left(  \px d\cdot\Delta d\right)$.}} 
        Note that
		\begin{align*}
			&   \left|   \Re\left( \py\left(  \px d\cdot\Delta d\right) \left\lvert\, \mathcal{M} e^{2 a A^{-\frac{1}{3}}\left|D_x\right|^{\frac{2}{3} } t}\langle D_x\rangle^{2 m}\langle\frac{1}{D_x}\rangle^{2 \epsilon} \omega\right.\right) \right|  \\
			&\lesssim    \int_{\mathbb{R}^2} e^{2 a A^{-\frac{1}{3}}|k|^{\frac{2}{3} } t }\langle k\rangle^{2 m   }\langle\frac{1}{k}\rangle^{2 \epsilon}   \left|\int_{\mathbb{R}} \mathcal{M}\left(k, D_y\right) \py   \omega_k(y)  l d_l(y)    \cdot \Delta_{k-l} d_{k-l}(y)  d y\right| d k d l.
		\end{align*}
		
		$\bullet$~\textbf{In $\mathcal{R}_{\text{res}}(k,l)$.} From \eqref{eq:trick1} and Corollary \ref{cor:L infty} ($f_l=|l|^\frac23 d_l$,  $\alpha=1$, $\beta=0$), we deduce that
		\begin{align*}  
			&\quad \int_{\mathcal{R}_{\text{res}}(k,l)} e^{2 a A^{-\frac{1}{3}}|k|^{\frac{2}{3} } t }\langle k\rangle^{2 m   }\langle\frac{1}{k}\rangle^{2 \epsilon}   \left|\int_{\mathbb{R}} \mathcal{M}\left(k, D_y\right)   \py \omega_k(y)  l d_l(y)    \cdot \Delta_{k-l} d_{k-l}(y)  d y\right| d k d l \\
			&\lesssim \lt\|e^{a A^{-\frac{1}{3}}|k|^{\frac{2}{3} }t}\langle k\rangle^{m  }\langle\frac{1}{k}\rangle^\epsilon   \lt\|\py \omega_k \rt\|_{L_y^2}\rt\|_{L_k^2}     \lt\|e^{a A^{-\frac{1}{3}}|l|^{\frac{2}{3} } t } |l|^\frac23 \lt \|   d_l \rt\|_{L_y^{\infty}}\rt\|_{L_l^1} \\
			&\quad \times\lt\|e^{a A^{-\frac{1}{3}}|k-l|^{\frac{2}{3} } t }\langle k-l\rangle^{m   } \langle\frac{1}{k-l}\rangle^\epsilon   |k-l|^\frac13   \lt\| \Delta_{k-l} d_{k-l}\rt\|_{L_y^2}  \rt \|_{L_{k-l}^2} \\
			&\lesssim \lt\|e^{a A^{-\frac{1}{3}}\left|D_x\right|^{\frac{2}{3} } t }    \nabla \omega\rt \|_{Y_{m,\epsilon}} \lt\|e^{a A^{-\frac{1}{3}}\left|D_x\right|^{\frac{2}{3} }t}      |D_x|^\frac13 \px|D_x|^{\frac13}d \rt\|_{Y_{m,\epsilon}}^\frac12       \\
			&\quad \times\lt \|e^{a A^{-\frac{1}{3}}\left|D_x\right|^{\frac{2}{3} }t}    |D_x|^{\frac13}\py|D_x|^{\frac13}d\rt\|_{Y_{m,\epsilon}}^\frac12 \lt\|e^{a A^{-\frac{1}{3}}\left|D_x\right|^{\frac{2}{3} } t }   \nabla^2 |D_x|^\frac13  d \rt\|_{Y_{m,\epsilon}} .  \alabel{eq:est of px d py^2 d py omega 1}
		\end{align*}

		$\bullet$~\textbf{In $\mathcal{R}_{\text{HL}}(k,l)$.}
		By \eqref{eq:GN1}, \eqref{eq:trick2} and $1 \lesssim |l|^\frac16 |k-l|^{-\frac16}$,
		we get
		\begin{align*}   
			&\quad\int_{\mathcal{R}_{\text{HL}}(k,l)}   e^{2 a A^{-\frac{1}{3}}|k|^{\frac{2}{3} } t }\langle k\rangle^{2 m   }\langle\frac{1}{k}\rangle^{2 \epsilon}   \left|\int_{\mathbb{R}} \mathcal{M}\left(k, D_y\right)   \py \omega_k(y)  l d_l(y)    \cdot \Delta_{k-l} d_{k-l}(y)  d y\right| d k d l \\
			&\lesssim  \lt\|e^{a A^{-\frac{1}{3}}|k|^{\frac{2}{3} }t}\langle k\rangle^{m   }\langle\frac{1}{k}\rangle^\epsilon   \lt\| \py \omega_k\rt\|_{L_y^2}\rt\|_{L_k^2}    \lt \|e^{a A^{-\frac{1}{3}}|l|^{\frac{2}{3} } t } \langle l\rangle^{m   }\langle\frac{1}{l}\rangle^\epsilon |l|^\frac76 \lt\|   d_l \rt\|_{L_y^{\infty}}\rt\|_{L_l^2} \\
			&\quad \times\lt\|e^{a A^{-\frac{1}{3}}|k-l|^{\frac{2}{3} } t }\langle k-l\rangle^{m   } \langle\frac{1}{k-l}\rangle^\epsilon    |k-l|^\frac13 \lt\|\Delta_{k-l} d_{k-l }\rt\|_{L_y^2}\rt \|_{L_{k-l}^2} 
			\lt\||k-l|^{-\frac12}\langle k-l\rangle^{-m}\langle\frac{1}{k-l}\rangle^{-\epsilon}\rt\|_{L_{k-l}^2}\\
			&\lesssim  \lt\|e^{a A^{-\frac{1}{3}}\left|D_x\right|^{\frac{2}{3} } t }    \nabla \omega \rt\|_{Y_{m,\epsilon}} \lt\|e^{a A^{-\frac{1}{3}}\left|D_x\right|^{\frac{2}{3} }t}    |D_x|^\frac13  \px|D_x|^{\frac13} d\rt\|_{Y_{m,\epsilon}}^\frac12       \\
			&\quad \times \lt\|e^{a A^{-\frac{1}{3}}\left|D_x\right|^{\frac{2}{3} }t}    |D_x|^{\frac13}\py|D_x|^{\frac13} d\rt\|_{Y_{m,\epsilon}}^\frac12 \lt\|e^{a A^{-\frac{1}{3}}\left|D_x\right|^{\frac{2}{3} } t }   \nabla^2 |D_x|^\frac13  d \rt\|_{Y_{m,\epsilon}} .   \alabel{eq:est of px d py^2 d py omega 2}
		\end{align*}
		$\bullet$~\textbf{In $\mathcal{R}_{\text{LH}}(k,l)$.}
		By \eqref{eq:GN1}, \eqref{eq:trick3} and \eqref{eq:trick3.5},
		we have
		\begin{align*}   
			&\quad\int_{\mathcal{R}_{\text{LH}}(k,l)}   e^{2 a A^{-\frac{1}{3}}|k|^{\frac{2}{3} } t }\langle k\rangle^{2 m   }\langle\frac{1}{k}\rangle^{2 \epsilon}   \left|\int_{\mathbb{R}} \mathcal{M}\left(k, D_y\right)   \py \omega_k(y)  l d_l(y)    \cdot \Delta_{k-l} d_{k-l}(y)  d y\right| d k d l \\
			&\lesssim  \lt\|e^{a A^{-\frac{1}{3}}|k|^{\frac{2}{3} }t}\langle k\rangle^{m   }\langle\frac{1}{k}\rangle^\epsilon   \lt\| \py \omega_k\rt\|_{L_y^2}\rt\|_{L_k^2}    \lt \|e^{a A^{-\frac{1}{3}}|l|^{\frac{2}{3} } t } \langle l\rangle^{m   }\langle\frac{1}{l}\rangle^\epsilon |l|^\frac76 \lt\|   d_l \rt\|_{L_y^{\infty}}\rt\|_{L_l^2} \\
			&\quad \times\lt\|e^{a A^{-\frac{1}{3}}|k-l|^{\frac{2}{3} } t }\langle k-l\rangle^{m   } \langle\frac{1}{k-l}\rangle^\epsilon    |k-l|^\frac13 \lt\|\q^2 d_{k-l }\rt\|_{L_y^2}\rt \|_{L_{k-l}^2} 
			\lt\|\langle k\rangle^{-m}\lt(\langle \frac{1}{k}\rangle^{\epsilon}+\langle \frac{1}{k}\rangle^{\frac12-\epsilon}\rt)\rt\|_{L_{k}^2},
		\end{align*}
which can be controlled by the last term of \eqref{eq:est of px d py^2 d py omega 2}.

		\underline{\textbf{Step III: Collecting.}} 
		Combining \eqref{eq:est of py d py^2 d 11}--\eqref{eq:est of px d py^2 d py omega 2},  we get
		\begin{align*}
			&\quad \frac1A \int_{0}^{t} \left| \Re\left( \px \left(  \py d\cdot\Delta d\right) \left\lvert\, \mathcal{M} e^{2 a A^{-\frac{1}{3}}\left|D_x\right|^{\frac{2}{3} }t }\langle D_x\rangle^{2 m}\langle\frac{1}{D_x}\rangle^{2 \epsilon} \omega \right.\right) \right| dt\\
			&\lesssim \frac{1}{A^\frac23}\lt \|       \omega \rt\|_{X_{a,m,\epsilon}} \lt\|     \nabla|D_x|^{\frac13} d \rt\|_{X_{a,m,\epsilon}}^{\frac12} \lt\|     \lt(\px^2,\py^2\rt)
|D_x|^{\frac13} d \rt\|_{X_{a,m,\epsilon}}^{\frac32} 
		\end{align*}
        and
        \begin{align*}
			&\quad \int_{0}^{t} \left| \Re\left( \py\left(  \px d\cdot\Delta d\right) \left\lvert\, \mathcal{M} e^{2 a A^{-\frac{1}{3}}\left|D_x\right|^{\frac{2}{3} }t }\langle D_x\rangle^{2 m}\langle\frac{1}{D_x}\rangle^{2 \epsilon} \omega \right.\right) \right| dt \\
			&\lesssim \frac{1}{A^\frac13} \|       \omega \|_{X_{a,m,\epsilon}} \lt\|     \nabla|D_x|^{\frac13} d \rt\|_{X_{a,m,\epsilon}}  \lt \|     \lt(\px^2,\py^2\rt)|D_x|^{\frac13} d \rt\|_{X_{a,m,\epsilon}}, 
		\end{align*}
		which, together with \eqref{eq:est of omega}, complete the proof.

	\end{proof}
	
	\section{Proof of Proposition \ref{main prop2}} \label{sec.6}
	In this section, we close the energy to complete the bootstrap argument, that is, we give the proof of Proposition \ref{main prop2}.
	
	\begin{proof}[\bf Proof of Proposition \ref{main prop2}]
		By \eqref{eq:bootstap2} and Young inequality, we have
            $$ \lt\|   |D_x|^\frac13 \nabla d \rt\|_{X_{a,m,\epsilon}}^2 \leq C\lt( \lt\|      |D_x|^\frac13 d  \rt\|_{X_{a,m,\epsilon}}^2+  \lt\|      |D_x|^\frac13 \nabla^2 d  \rt\|_{X_{a,m,\epsilon}}^2\rt)\leq C K^2. $$
        It follows from Proposition \ref{lem:est of |D_x|13 d} that for $0\leq t \leq T$, there holds
		\begin{align*} 
			\lt \|      |D_x|^\frac13 d  \rt\|_{X_{a,m,\epsilon}}^2  
			&\leq  C\lt\|      |D_x|^\frac13  d_{\mathrm{in}} \rt\|_{Y_{m,\epsilon}}^2 	+ \frac{C} {A^\frac12} \lt\|   \omega \rt\|_{X_{a,m,\epsilon}}  \lt  \|   |D_x|^\frac13 d \rt\|_{X_{a,m,\epsilon}}^2
			\\
			&\quad + \frac{C}{A^{\frac{1}{3}}} \lt\|    |D_x|^\frac13 d\rt  \|_{X_{a,m,\epsilon}}^{2}   \lt\|   \nabla |D_x|^\frac13 d \rt\|_{X_{a,m,\epsilon}}^2  \\
			&\quad + \frac{C}{A^{\frac{1}{6}}}\lt \|  |D_x|^\frac13 d^1 \rt\|_{X_{a,m,\epsilon}}\lt \|    |D_x|^\frac13 d\rt \|_{X_{a,m,\epsilon}}  \lt\|   \nabla |D_x|^\frac13 d \rt\|_{X_{a,m,\epsilon}}  .\\
            &\leq  C\lt\|      |D_x|^\frac13  d_{\mathrm{in}} \rt\|_{Y_{m,\epsilon}}^2 	+ \frac{C} {A^\frac16} K^2\lt\|    |D_x|^\frac13 d \rt\|_{X_{a,m,\epsilon}}^{2}+\frac{C} {A^\frac12}\lt\|  |D_x|^\frac13 d \rt\|_{X_{a,m,\epsilon}}^2 K^2.
		\end{align*}
Similarly, by Proposition \ref{lem:est of pypy dx13 d} we have
		\begin{align*} 
			&\quad\lt\| \lt (\px^2,\py^2\rt) |D_x|^\frac13 d  \rt\|_{X_{a,m,\epsilon}}^2 \\
			&\leq C\lt\|   \lt (\px^2,\py^2\rt) |D_x|^\frac13  d_{\mathrm{in}} \rt\|_{Y_{m,\epsilon}}^2  + CA\lt  \|     |D_x|^\frac13 d\rt\|_{X_{a,m,\epsilon}} \lt\|    \lt(\px^2,\py^2\rt)|D_x|^\frac13 d\rt\|_{X_{a,m,\epsilon}}+ \frac{C}{A^{\frac16}}  K^4
		\end{align*}
	and Proposition \ref{lem:est of omega} implies 
\begin{align*}
			\|         \omega  \|_{X_{a,m,\epsilon}}^2 
			&\leq  C  
			 \|        \omega_{\mathrm{in}} \|_{Y_{m,\epsilon}}^2 + \frac{C}{A^{\frac13}}  K^3.
		\end{align*}

		  Hence, \eqref{eq:bootstap2} implies 
          \begin{align*}
		E(t)& =       A^{\delta}\lt\||D_x|^{\frac13} d\rt\|_{X_{a,m,\epsilon}}  +   \lt\|   \lt (\px^2,\py^2\rt) |D_x|^{\frac13}d \rt\|_{X_{a,m,\epsilon}}     +\lt \| \omega\rt \|_{X_{a,m,\epsilon}}  \\
        & \leq  CA^{\delta}\lt\|      |D_x|^\frac13  d_{\mathrm{in}} \rt\|_{Y_{m,\epsilon}} + C\lt\|     \lt (\px^2,\py^2\rt) |D_x|^\frac13  d_{\mathrm{in}} \rt\|_{Y_{m,\epsilon}} + C\lt\|        \omega_{\mathrm{in}} \rt\|_{Y_{m,\epsilon}}	+ \frac{C} {A^{\frac1{12}}} K^2+C A^{\frac{1-\delta}{2}}K.
	\end{align*}
	Then, choose $\delta>1$ and  
    $$
    \bar{A}_1 = C(m,\epsilon,\delta)K^{\max\{\frac{2}{\delta-1},24 \}} 
    $$ 
    such that if $A> \bar{A}_1$, it follows that
		\begin{align*}
			1\geq  \frac{C} {A^{\frac1{12}}} K^2+C A^{\frac{1-\delta}{2}}K.
		\end{align*}
		Next, by \eqref{eq:smallness2} we define
		\begin{align*}
			E(t) &\leq  C\left( \lt\|    \lt (\px^2,\py^2\rt) |D_x|^\frac13  d_{\mathrm{in}} \rt\|_{Y_{m,\epsilon}}+\|        \omega_{\mathrm{in}} \|_{Y_{m,\epsilon}}  +1 \right)	 =: K,
		\end{align*}
		which completes the proof by noting that
        $$
        A\geq C(m,\epsilon,\delta)\left( \lt\|   \lt (\px^2,\py^2\rt) |D_x|^\frac13  d_{\mathrm{in}} \rt\|_{Y_{m,\epsilon}}+\lt\|        \omega_{\mathrm{in}} \rt\|_{Y_{m,\epsilon}}	+1\right) ^{\max\{\frac{2}{\delta-1},24 \}}. 
        $$

	\end{proof}

	\section*{Declarations}
	\begin{itemize}
		\item \textbf{Acknowledgements} 
		W. Wang was supported by National Key R\&D Program of China (No.2023YFA1009200) and NSFC under grant 12471219.  The research of J. Wei is partially supported by GRF from RGC of Hong Kong
		entitled ``New frontiers in singularity formations in nonlinear partial differential equations".
		\item \textbf{Conflict of interest} The authors declare that they have no conflict of interest.
		\item \textbf{Data Availability} Data sharing is not applicable to this article as no datasets were generated or analyzed during the current study.
	\end{itemize}
	
	\bibliographystyle{myamsalpha}
	\bibliography{EL_CWYrefs}
\end{document}